\documentclass[11pt,reqno]{amsart}
\usepackage{amssymb, bbm, color}
\usepackage[makeroom]{cancel} 
\usepackage{amsfonts}
\usepackage{mathrsfs}
\usepackage{amsmath}
\usepackage{graphicx}
\usepackage{hyperref}
\usepackage{float}
\usepackage{epstopdf}
\usepackage{color}
\usepackage{bm}
\usepackage{comment}
\usepackage{soul}
\usepackage{esint}
\usepackage[shortlabels]{enumitem}
\usepackage{eufrak}

\allowdisplaybreaks
\newtheorem{theorem}{Theorem}[section]
\newtheorem{proposition}[theorem]{Proposition}
\newtheorem{lemma}[theorem]{Lemma}
\newtheorem{corollary}[theorem]{Corollary}
\newtheorem{remark}[theorem]{Remark}

\newtheorem{example}[theorem]{Example}
\newtheorem{definition}[theorem]{Definition}

 \def\Ex{\mathfrak E}

\def\mcB{\mathcal{B}}

\def\mcE{\mathcal{E}}
\def\mcF{\mathcal{F}}

\def\mcQ{\mathcal{Q}}

\def\sE{\mathcal{E}}
\def\sF{\mathcal{F}}
\def\sG{\mathcal{G}}
\def\sX{\mathcal{X}}

\def\R{{\mathbbm R}}
\def\N{{\mathbbm N}}
\def\Z{{\mathbbm Z}}
\def\D{{\mathbbm D}}
\def\bP{{\mathbbm P}}
\def\bE{{\mathbbm E}}

 \makeatletter
\@namedef{subjclassname@2020}{%
  \textup{2020} Mathematics Subject Classification}
\makeatother

\def\eps{\varepsilon}
\def\wt{\widetilde}

\def\<{\langle}
\def\>{\rangle}

\numberwithin{equation}{section}

\begin{document}
\title[Convergence of resistances on generalized {S}ierpi\'{n}ski carpets]{Convergence of resistances on generalized {S}ierpi\'{n}ski carpets}

\author{Shiping Cao}
\address{Department of Mathematics, University of Washington, Seattle, WA 98195, USA}
\email{spcao@uw.edu}
\thanks{}

\author{Zhen-Qing Chen}
\address{Department of Mathematics, University of Washington, Seattle, WA 98195, USA}\email{zqchen@uw.edu}
\thanks{}


\date{May 7, 2023}


\maketitle

\begin{abstract}
	We positively answer the open question of Barlow and Bass about the convergence of renormalized effective resistance between opposite faces of Euclidean domains approximating a generalized {S}ierpi\'{n}ski carpet.
\end{abstract}

\section{Introduction}\label{sec1}

Denote by $\N$, $\Z$ and $\R$  the set of all natural numbers, integers and real numbers, respectively. 
Let $(S_n)_{n\in \{ 0\} \cup \mathbb{N}}$ be the simple random walk on $\mathbb{Z}^d$ with $d\geq 1$, and $\D ( [0,\infty); \R^d)$ the space of right continuous functions
having left limits on $[0, \infty)$ taking values in $\R^d$ equipped with the Skorohod topology. 
The well-known Donsker's invariance principle states that
$\{n^{-1/2} S_{[nt]}; t\geq 0\} $ converges weakly as $n\to\infty$ in  the Skorohod space  $\D ( [0,\infty); \R^d)$  to a  Brownian motion on $\R^d$. 

On generalized Sierpi\'nski carpets ({\it GSC}\,s), an interesting question is whether an analogy to the Donsker's invariance principle   holds, where instead of studying the scaling limit of random walks, a more natural choice is to consider the scaling limit of reflected Brownian motions on their approximation domains. The problem is difficult, but the picture becomes clearer over the times, with important contributions from \cite{BB1,BB2,BB3,BB4,BB5,BBKT,KZ}.

To describe the setting of this paper, we first recall the definition of {\it GSC}\,s and their 
approximation domains (also called pre-carpets in some literatures) from \cite{BB5,BBKT}. In this paper, we use := as a way of definition. Let  $d\geq 2$ and $L_F\geq 3$ be integers. 
Let $F_0:=[0,1]^d$ be the unit cube in $\R^d$ and set
$\mathcal{Q}_0:=\{F_0\}$. For each integer $n\geq 1$,
we divide $F_0$ into $L_F^{nd}$ sub-cubes of length $L_F^{-n}$: 
\begin{equation}\label{e:1.1a}
\mathcal{Q}_n:=\Big\{\prod_{i=1}^d[(l_i-1)/L_F^n,l_i/L_F^n]:1\leq l_i\leq L_F^n,i=1,2,\cdots,d \Big\}.
\end{equation}
For each set $A\subset \R^d$ and $n\geq 0$, we denote 
\begin{equation}\label{e:1.2a}
\mcQ_n(A):=\{Q\in \mcQ_n: \operatorname{int}(Q)\cap A\neq \emptyset\},
\end{equation}
where $\operatorname{int}(Q)$ stands for
 the interior of $Q$ in $\R^d$. Let $F_1\subsetneq F_0$ be a union of cubes in $\mcQ_1$, and iteratively, we define
$F_n:=\bigcup_{Q\in \mcQ_1(F_1)} \Psi_Q(F_{n-1})$ for $n\geq 2$, where for each $Q\in \bigcup_{n=0}^\infty\mcQ_n$,   $\Psi_Q$ is an orientation preserving affine map of $F_0$ onto $Q$. We call $F:=\bigcap_{n=0}^\infty F_n$ a generalized Sierpi\'nski carpet ({\it GSC} ) if the following conditions (SC1)--(SC4) hold.
\begin{enumerate}[label={(SC\arabic*)}]
	\item (Symmetry) $F_1$ is preserved by all the isometries of the unit cube $F_0$.
	
	\item (Connectedness) The interior ${\rm int}(F_1)$ of $F_1$ is connected.
	
	\item (Non-diagonality) Let $n\geq 1$ and $B\subset F_0$ be a cube of side length $2L_F^{-n}$, which is the union of $2^d$ distinct elements of $\mathcal{Q}_n$. Then if $\text{int}(F_1\cap B)$ is non-empty, it is connected.
	
	\item (Borders included) $F_1$ contains the line segments $\{x=(x_1,x_2,\cdots,x_d):0\leq x_1\leq 1,x_2=x_3=\cdots=x_d=0\}$.
\end{enumerate}
Let $m_F:=\#\mcQ_1(F)$. Then  $d_f:=\frac{\log m_F}{\log L_F}$ is the Hausdorff dimension of $F$. In words, $F_1$ is obtained from the unit cube $F_0$ by removing a symmetric pattern of $L_F^d-m_F$ number of sub-cubes of length $L_F^{-1}$, and we require $F_1$ to satisfy condiditions (SC1)-(SC4). Then we repeat the procedure of removing a same pattern from surviving small cubes infinitely many times to get a compact set $F$. 
The standard {S}ierpi\'{n}ski carpet in $\R^2$ corresponds to the case of $L_F=3$, $m_F=8$ and $F_1$ being obtained from the unit square $F_0$ by removing the central square of length $1/3$; see Figure \ref{fig1} for a picture of the standard $SC$ in $\R^2$ and  Figure \ref{fig2} for its approximation domains $F_0,F_1,F_2$. Figure \ref{SC3d} shows the {S}ierpi\'{n}ski sponge, which is an example of {\it GSC} in $\R^3$.

\begin{figure}[htp]
	\includegraphics[width=3.5cm]{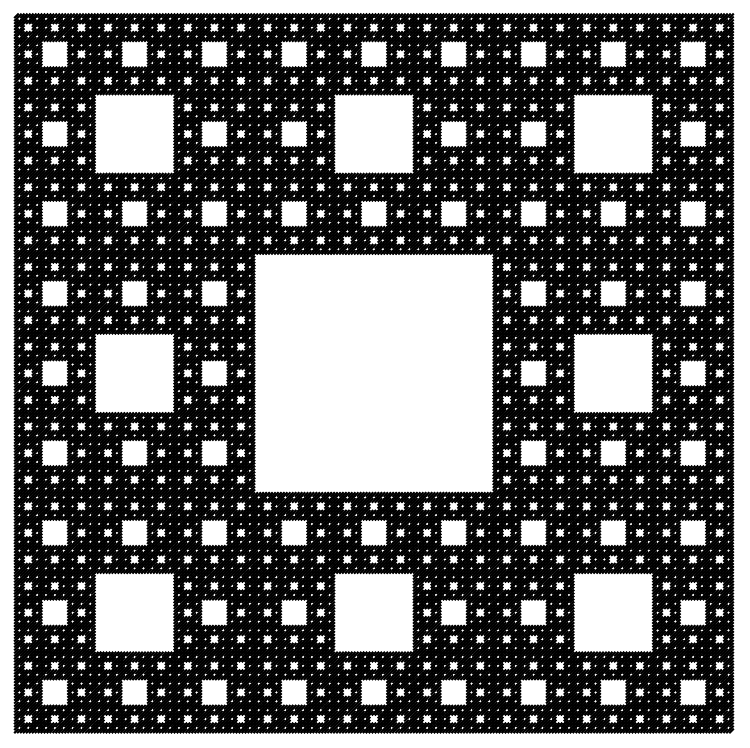}
	\caption{The standard {S}ierpi\'{n}ski carpet in $\R^2$}\label{fig1}
\end{figure}

\begin{figure}[htp]
	\includegraphics[width=3.5cm]{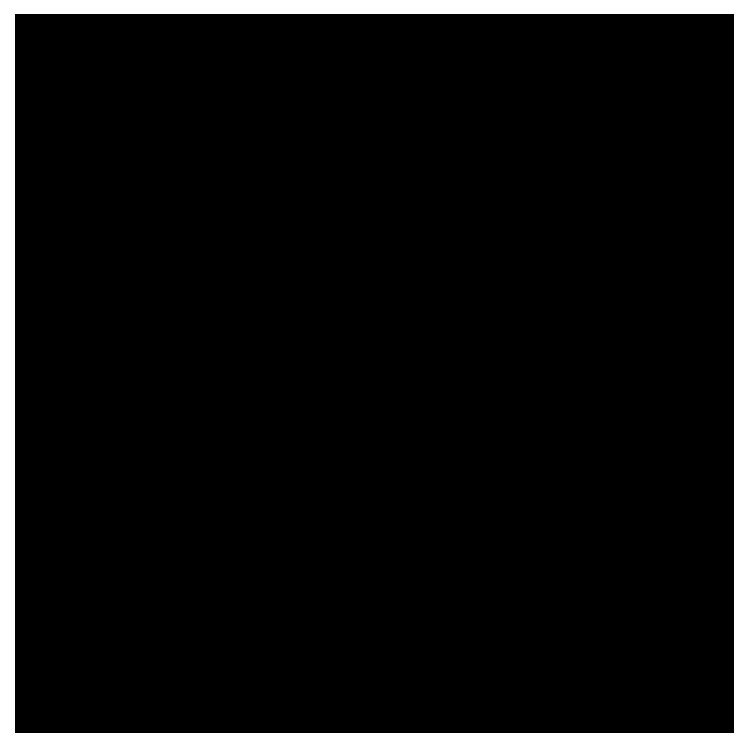}\qquad 
	\includegraphics[width=3.5cm]{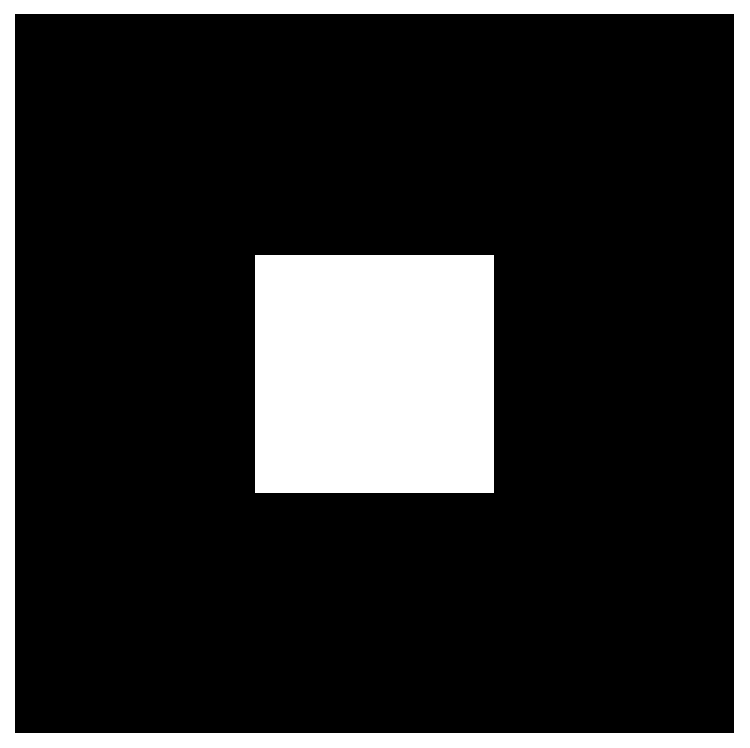}\qquad
	\includegraphics[width=3.5cm]{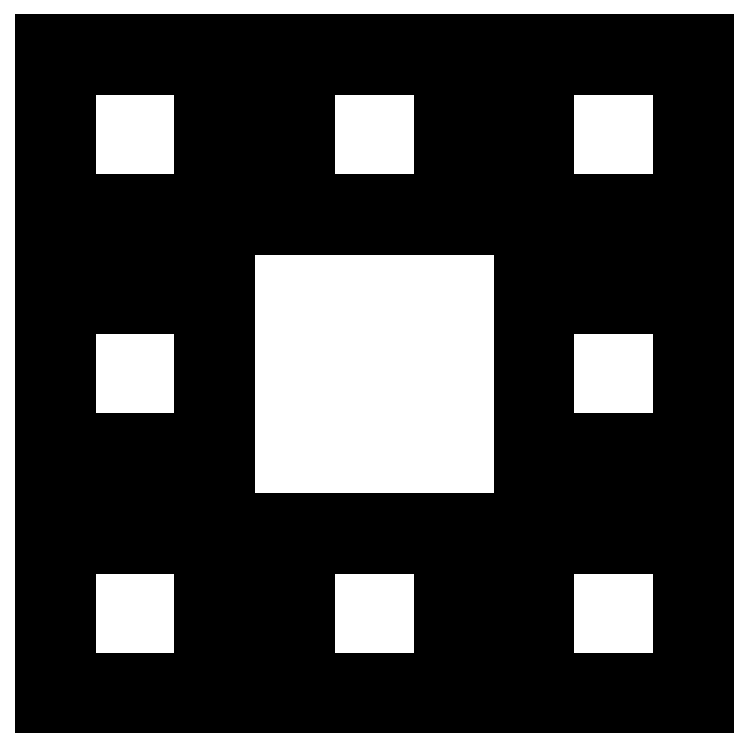}
	\caption{Approximating domains $F_0$, $F_1$ and $F_2$ of the standard {S}ierpi\'{n}ski carpet }\label{fig2}
\end{figure}\vspace{-0.6cm}
 
 \begin{figure}[htp]
 	\includegraphics[width=5cm]{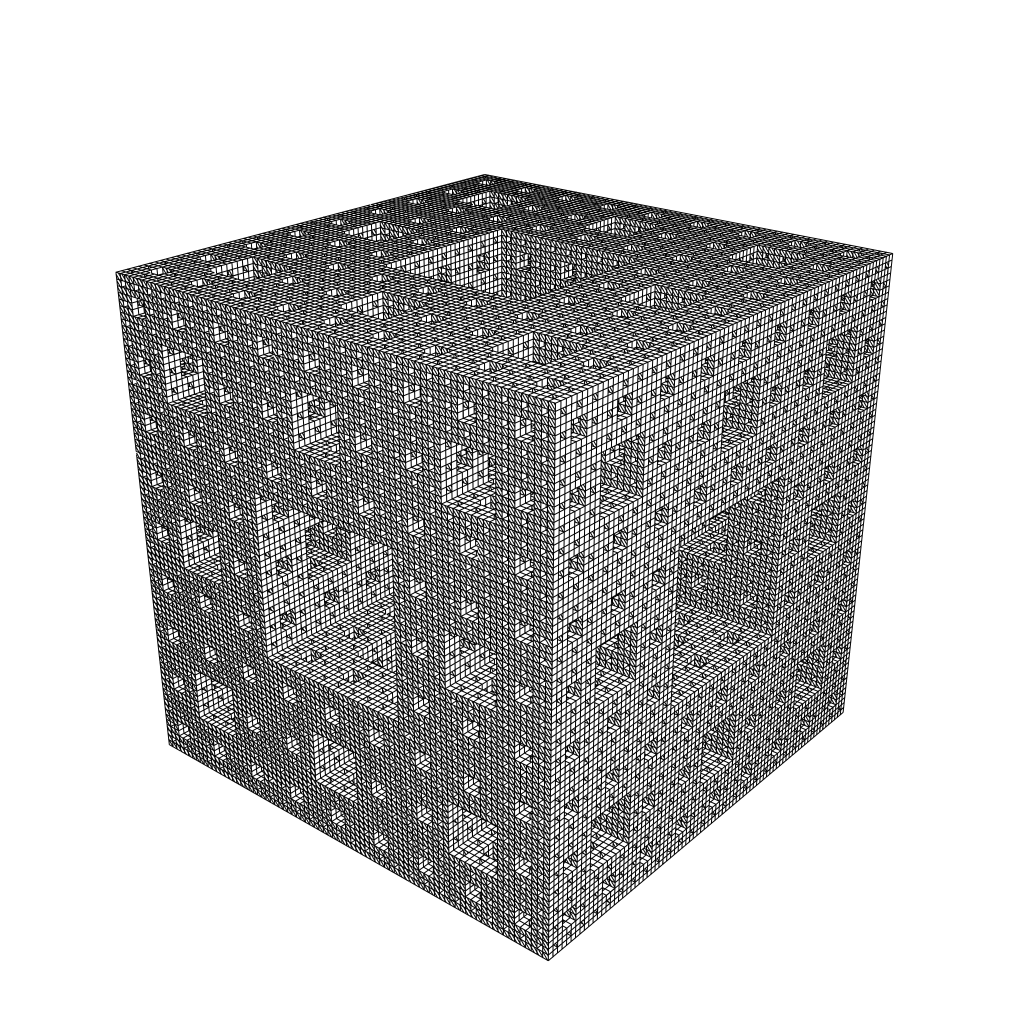}\vspace{-0.3cm}
 	\caption{The {S}ierpi\'{n}ski sponge in $\R^3$}\label{SC3d}
 \end{figure}
 
 \bigskip
 
{\it GSC}s are infinitely ramified fractals. The study of Brownian motions on {\it GSC}\,s was initiated by Barlow and Bass \cite{BB1} in 1989, where they constructed  Brownian motions (also known as locally symmetric diffusions) on a planar {\it GSC} using a probabilistic approach as the scaling subsequential limits of reflected Brownian motions $W_t^{(F_n)}$ on $F_n$ (more precisely, as the weak subsequential limits of $W_{\lambda_nt}^{(F_n)}$ for some constant time scaling factors $\lambda_n\in (0,\infty),n\geq 0$). The same construction extends to higher dimensional cases in \cite{BB5}. The scaling factor $\lambda_n$ is given by
\begin{equation}\label{e:1.1}
\lambda_n=(\rho_F m_F /L_F^2)^n \quad \hbox{for } n\geq 0,
\end{equation} 
where $\rho_F >0$ is the resistance scaling factor for $F$. It is known \cite[Proposition 5.1]{BB5} that 
\begin{equation}\label{e:1.4a}
\bar \rho_F := \rho_F m_F /L_F^2 \geq 1  \quad \hbox{and} \quad   \rho_F \leq 2^{1-d}L_F.
\end{equation}
Thus  $  \rho_F>1$ for any {\it GSC} in $\R^2$ and $0< \rho_F<1$ for any {\it GSC} in $\R^d$ with $d\geq 3$ and $L_F<2^{d-1}$.  
 But there is a {\it GSC} in $\R^3$ having $ \rho_F>1$; see \cite[Section 9]{BB5}.
  Using the resistance estimates \cite{BB3,Mc} and the elliptic Harnack inequalities \cite{BB1,BB5}, Barlow and Bass \cite{BB4,BB5}
	established the sub-Gaussian heat kernel estimates for a Brownian motion on {\it GSC}\,s:
\begin{equation}\label{e:1.2}
\begin{aligned}
 & \frac{c_1}{t^{d_f/d_w}}\exp \left({-c_2 \left(\frac{\rho(x,y)^{d_w}  }{t}\right)^{{1}/{(d_w-1)}}} \right) \leq p(t,x,y)\\ 
 & \hskip 0.2truein \leq \frac{c_3}{t^{d_f/d_w}}\exp \left({-c_4 \left(\frac{\rho(x,y)^{d_w}  }{t}\right)^{{1}/{(d_w-1)}}} \right) 
 \quad \hbox{for }  (t, x, y) \in (0, 1] \times F\times F,
\end{aligned}
\end{equation}
where 
$$
d_w:= \log (\rho_F m_F )/\log L_F >2 
$$
is called the walk dimension of $F$; see \cite[Theorem 1.3 and Remark 5.4]{BB5} or \cite[Remark 4.33]{BBKT}. In literature, $d_s:=\frac{2d_f}{d_w}= \frac{2\log m_F}{\log (\rho_F m_F)}$ is called the spectrum dimension of $F$. Note that $d_s<2$ if and only if  $\rho_F>1$. 
Observe also that $ \bar\rho_F$ in \eqref{e:1.4a} can be expressed as 
\begin{equation}\label{e:1.6a}
\bar \rho_F= L_F^{d_w-2}  
\quad \hbox{and so} \quad 
\lambda_n = \bar \rho_F^n = L_F^{(d_w-2)n}.
\end{equation}

In \cite{KZ}, Kusuoka and Zhou used Dirichlet forms for random walks on fractal-like finite graphs to establish the existence of scale invariant (self-similar) diffusion processes on two-dimensional {\it GSC}\,s, which have the same type of heat kernel estimates. Recently, Grigoryan and Yang \cite{GY} gave a purely analytic construction of a self-similar local regular Dirichlet form on the two-dimensional standard {S}ierpi\'{n}ski carpet $F$ using approximation of stable-like non-local closed forms on $F$.
 
Almost twenty years after \cite{BB4}, Barlow, Bass, Kumagai and Teplyaev \cite[Theorem 1.2]{BBKT} established that 
strongly local, regular, irreducible, locally symmetric Dirichlet forms on $F$ are unique up to a constant multiple. 
These strongly local, regular, irreducible, locally symmetric Dirichlet forms on $F$ are exactly the Dirichlet forms associated with the locally symmetric diffusion processes constructed on GSCs in \cite{BB1, BB5, KZ} up to to a  constant time change. Since $\lambda_n = \bar \rho_F^n$ by  \eqref{e:1.6a}, we define $X_t^{(F_n)}:= W^{(F_n)}_{\bar \rho_F^n t}$ the normally reflected Brownian motion on $F_n$ running with speed $\bar \rho_F^n = L_F^{(d_w-2)n}$.  As mentioned earlier, it is proved in 
\cite{BB1, BB5} that for any subsequence of $\{ X_t^{(F_n)}; n\geq 0\}$, there is a sub-subsequence that converges weakly 
in the space $C([0, \infty); \R^d)$ to a Brownian motion on $F$. 
However, even with the unique result from \cite{BBKT}, it remained open till now whether the sequence  $ \{ X^{(F_n)}; n\geq 0\}$ itself converges; 
 see \cite[Remark 2.13]{Ba2013}. The first main result of this paper is to show that the process $X^{(F_n)} $ converges weakly in the space $C([0, \infty); \R^d)$ equipped with local uniform topology as $n\to \infty$.

Throughout this paper, we use $X^{(F)}$ to denote the locally symmetric diffusion on a {\it GSC} $F$ so that mean time of $X^{(F)}$ starting from ${\bf 0}:=(0,\ldots, 0) \in \R^d$ to hit   the faces of $\partial F_0$ not containing  ${\bf 0} $ is 1. 
We call $X^{(F)}$ a Brownian motion on $F$.
The symmetric strongly local regular Dirichlet form on $L^2(F; \mu)$ associated with $X^{(F)}$ will be denoted as $(\sE^{(F)}, \sF^{(F)})$, which is irreducible and locally symmetric in the sense of \cite[Definition 2.15]{BBKT}. Here $\mu$ is the $d_f$-dimensional
Hausdorff measure on $F$ normalized so that $\mu(F)=1$.

\begin{theorem}\label{T:1.1}
There is a constant $c_0>0$ so that for each $x\in F$ and $x_n\in F_n,n\geq 0$ such that $x_n\to x$ as $n\to\infty$,  the law of $ (X^{(F_n)}_t)_{t\geq 0}$ starting from $x_n$ converges weakly to some conservative continuous Markov process 
	$ \big(X^{(F)}_{t/c_0 }\big)_{t\geq 0}$  starting from $x$ as $n\to \infty$ in the  space $C([0,\infty); \R^d)$  equipped with local uniform topology,
	and $X^{(F)}$ is a locally symmetric diffusion on the {\it GSC} $F$.
\end{theorem}

\medskip

 \begin{remark}\label{R:1.2} \rm 
Let $\tau_n$ denote the time that  $X^{(F_n)}$ starting from ${\bf 0}$ to hit the faces of $\partial F_0$ not containing  ${\bf 0} $.
Then by Theorem \ref{T:3.11},  \eqref{e:3.11a}  and the proof of Theorem \ref{T:1.1} of this paper, 
$\lim_{n\to \infty} \bE^{(F_n)}_{\bf 0} \left[ \tau_n \right]$ exists as a positive number and $c_0$ is this limit. 
\end{remark}

Our proof for Theorem \ref{T:1.1} uses Mosco convergence of Dirichlet forms on varying spaces as developed in Kuwae and Shioya \cite{KS}. Denote by $m$ the Lebesgue measure on $\R^d$. For $n\geq 0$, let $\mu_n$ be the normalized Lebesgue measure on
$F_n$ so that $\mu_n(F_n)=1$; that is 
$$
\mu_n = m_F^{-n} L_F^{nd} \, m|_{F_n} = L_F^{(d-d_f)n} \, m|_{F_n}.
$$
Denote by $W^{1,2}(F_n)$ the Sobolev space of order $(1, 2)$ on $F_n$:  
$$
W^{1,2}(F_n):=\{f\in L^2(F_n; m): \nabla f \in L^2(F_n; m)\}.
$$
The Dirichlet form of $X^{(F_n)}$ on $L^2(F_n; \mu_n)$ is $(\mcE^{(F_n)},W^{1,2}(F_n))$, where 
$$
\mcE^{(F_n)}(f,g)=\frac{1}{2} L_F^{(d_w-d_f+d-2)n} \int_{F_n}  \nabla f(x) \,\nabla g(x) m(dx) 
= \frac{1}{2} L_F^{(d_w- 2)n} \int_{F_n}  \nabla f(x) \,\nabla g(x) \mu_n(dx)
$$
for 
$f,g\in W^{1,2}(F_n)$.  It is a  strongly local regular Dirichlet form  on $L^2(F_n; \mu_n)$. 

\medskip

 As mentioned in \cite[Remark 5.4]{BBKT}, which we will present as Theorem  \ref{T:3.11} and give a proof,
 there is a sequence of constants $\{\alpha_n; n\geq 0\}$ that  are bounded between two positive constants such that  
$\big\{ X^{(F_n)}_{\alpha_n t}; t\geq 0 \big\} $ converges weakly to $X^{(F)}$ in the space $C([0, \infty); \R^d)$  equipped with local uniform topology. Consequently, we show in Theorem \ref{T:3.12} 
 that $\alpha_n\mcE^{(F_n)}$ is Mosco convergent to $\mcE^{(F)}$. There is a close relationship between the Mosco convergence of the Dirichlet forms and the convergence of finite-dimension distributions of the associated processes;  see Theorem \ref{thm211} below.  So the crux work is to show the convergence of $\alpha_n$.
 
 \begin{theorem}\label{T:1.3}
The limit $\lim_{n\to \infty} \alpha_n$ exists and equals the constant $c_0$ in Theorem \ref{T:1.1}.
Moreover, the Dirichlet form $(\sE^{(F_n)}, W^{1,2}(F_n))$ on $L^2(F_n; \mu_n)$ is Mosco convergent to 
$(c_0^{-1} \sE^{(F)}, \sF^{(F)})$ on $L^2(F;\mu)$ as $n\to \infty$ in the sense of Definition \ref{def28}. 
\end{theorem}\medskip

This   in particular  answers a question raised by Barlow  \cite[p.9 and Remark 2.13]{Ba2013} on the convergence of $\sE^{(F_n)}$.
The above two results are  
closely related to an open question of Barlow and Bass \cite{BB3}  concerning the asymptotic behavior of effective resistance.  
  The effective resistance $R_n$ between two opposite faces of $F_n$ with respect to 
the Dirichlet form $(\mcE^{(F_n)},W^{1,2}(F_n))$ is defined by 
\begin{equation}\label{e:1.4}
R_n^{-1} := \inf\Big\{\mcE^{(F_n)}(f,f):   f\in W^{1,2}(F_n)\cap C(F_n),   \, f|_{  F_n  \cap \{x_1=0\} }=0, \, f|_{  F_n \cap \{ x_1= 1\}  } =1 \Big\} , 
\end{equation}  
where  for $i\in \{1,2,\cdots,d\}$, $x_i$ stands for the $i$th coordinate of $x\in \R^d$. It is shown  in  \cite[Theorem 5.1]{BB3}  that for the standard Sierpi\'nski carpet in $\R^2$,
  $$
  1/4\leq R_n\leq 4  \quad \hbox{for every } n\geq 1.
  $$
 Note that our definition of the resistance $R_n$ is the normalized version of the resistance defined in \cite{BB3}; that is, our $R_n$ is 
 $R_n/\rho^n$ in \cite{BB3} with   $\rho:= \rho_F$. The open question posed as Problem 1 in Barlow and Bass \cite{BB3} is whether the limit of $R_n$ exists. 
The second main result of this paper answers this question affirmatively.

\begin{theorem}\label{T:1.4}
	The limit $\lim\limits_{n\to\infty} R_n$ exists as  a positive real number. 
\end{theorem}

Theorems \ref{T:1.1}. \ref{T:1.3} and \ref{T:1.4} play a fundamental role in our study of stochastic homogenization on unbounded generalized {S}ierpi\'{n}ski carpets in our forthcoming paper \cite{CC}. 
 
 \medskip
 
A major step in our proof of the convergence of $\alpha_n$ and $R_n$ is to construct a sequence of functions $h_n\in C(F_n)\cap W^{1,2}(  F_n)$ that takes value 0 and 1 on a pair of opposite faces of $F_n$ and is strongly convergent  to $h$ in $L^2$ with  $\lim_{n\to \infty} \mcE^{(F_n)}(h_n)=c_0^{-1}\mcE^{(F)}(h)$, 
where $h$ is the  continuous function on $F$ that takes $0$ or  $1$ on a pair of opposite faces of $F$ and  is $\mcE^{(F)}$-harmonic elsewhere. 
To construct such functions, we establish   trace theorems on $F$ and $F_n$, respectively, 
which extend an earlier result  of Hino and Kumagai \cite{HK}, and a property that the energy measures of $\sE^{(F)}$-harmonic functions within the $\eps$-neighborhood of the boundary of a cell decays at a polynomial rate in $\eps>0$; see Theorems \ref{T:A.3}, \ref{T:A.4} and \ref{thmB1} in the  two Appendices. We also establish several properties for GSCs that were previously either put as assumptions or were proved under some additional conditions in the literature. They include, for instance, Lemma \ref{lemma38}, Lemma \ref{lemmaHK} and Corollary \ref{coroA5}.  These results shied new lights on Brownian motions on the GSC $F$.

\medskip

We remark that there are other ways to approximate   a {\it GSC}, for example the cell graphs \cite{KZ} and the graphical $SC$'s \cite{BB6}. The approach of  this paper can be modified to establish the corresponding convergence results 
on these approximating spaces to $F$. 

\medskip 

The rest of the paper is organized as follows. In Section \ref{sec2}, we recall the definition of Mosco convergence of Dirichlet forms defined on varying spaces from \cite{KS}. We also carefully define various concepts of convergence of functions and present some of their properties that 
will be used in establishing Mosco convergence in this paper. In Section \ref{sec3}, we first recall the uniform elliptic Harnack inequalites from \cite{BB5}.
We then present lower bound estimates of effective resistances between a cell and the complement of its neighborhood and establish 
the Mosco convergence of $\alpha_m \sE^{(F_m)}$ to $\sE^{(F)}$. Sections \ref{sec4} and \ref{sec5} are the most important parts of the paper. In Section \ref{sec4}, we establish an uniform estimate of $ \big|\sqrt{\mcE^{(F_n)}(g_n)}-\sqrt{\mcE^{(F)}(g)}\big|$ 
in terms of  the Besov norms on the boundary and energy measure near the boundary, where $g_n\in W^{1,2}(F_n),n\geq 0$ and $g\in \mcF^{(F)}$ are $\sE^{(F_n)}$- and $\sE^{(F)}$-harmonic functions having
 the same average values on level-$m$ sub-faces of the boundary. We state the results in slightly more generality for future applications. In particular, Corollary \ref{coro48} will not be used in this paper but will be needed in a forthcoming paper 
\cite{CC}. In Section \ref{sec5}, we construct the desired  approximating functions $h_n$ mentioned above and establish the energy estimates of $h_n$. The proofs for Theorems \ref{T:1.1}, \ref{T:1.3} and \ref{T:1.4} 
are given at the end of Section \ref{sec5}, which also use trace theorems and energy measure boundary decay property of $\sE^{(F)}$-harmonic functions. The proof of these results are given in two Apppedixes.  In Appendix \ref{secA}, we prove a trace theorem that relates energy measures near the boundary of $F$ (see Figure \ref{fig3} in Section \ref{sec4}) to some weighted Besov energies on the boundary of $F$. In Appendix \ref{secB}, 
for any continuous function  in $\sF^{(F)}$  that is $\sE^{(F)}$-harmonic  in a cell of $F$, we show that its   energy measure   in the $\eps$-neighborhood of  the boundary of the  cell in $F$ decays  at a polynomial rate in $\eps$. This result extends a result of similar type in Hino and Kumagai \cite[Proposition 3.8]{HK}. Our proof is different from theirs.  
 
\bigskip

For the reader's convenience, the following is a list of notations used in this paper.

\medskip

$F$: generalized {S}ierpi\'{n}ski carpet in $[0, 1]^d$; Section \ref{sec1}

$F_n$: stage $n$ construction of the  {\it GSC} $F$; Section \ref{sec1}

$\partial_o F:= F \cap \partial F_0$ and $\partial_o F_m:= F_m \cap \partial F_0$;  Section \ref{sec2}

$L_F\geq 3$: length scale of $F$; Section \ref{sec1}

$\mathcal{Q}_n$: sub-cubes of length $L_F^{-n}$ defined in \eqref{e:1.1a}

$\mathcal{Q}_n(A):=\{Q\in\mcQ_n:\operatorname{int}(Q)\cap A\neq\emptyset\}\}$: defined in \eqref{e:1.2a} 

$\Psi_Q$: orientation preserving affine map of $F_0=[0, 1]^d$ onto $Q \in \cup_{n=0}^\infty \mathcal{Q}_n $; Section \ref{sec1}

$m_F:= \#  \mcQ_1(F)$:  mass scale of $F$;  Section \ref{sec1}

$m_I:=\#\{Q\in \mcQ_1(F_1):Q\cap \{x_1=0\} \neq \emptyset\}$; Section \ref{sec2}

$\rho_F$: resistance scaling factor for $F$; \eqref{e:1.1}

$d_f:=\frac{\log m_F}{\log L_F}$: Hausdorff dimension of $F$; Section \ref{sec1}

$d_w:= \log (\rho_F m_F )/\log L_F \geq 2$: walk dimension of $F$; Section \ref{sec1}

$d_I:=\frac{\log m_I}{ \log L_F}$: Hausdorff dimension of the outer boundary $\partial_o F:=F\cap \partial F_0$ of $F$; Section \ref{sec2}

$\mu_n$:     normalized Lebesgue measure on $F_n$ so that $\mu_n(F_n)=1$; Section \ref{sec1}

$\mu$: normalized $d_f$-dimensional Hausdorff measure on $F$ so that $\mu(F)=1$; Section \ref{sec1}

$l(A)$: space of real-valued measurable functions on a measurable space $A$;  Section \ref{sec2}

$F_Q:= F\cap Q$  and $F_{m, Q}:=F_m\cap Q$ for $Q\in\mcQ_n$;  Section \ref{sec2}

$F_{\mathcal A}:=\cup_{Q\in \mathcal A} F_Q$ and $F_{m, \mathcal A}:=\cup_{Q\in \mathcal A} F_{m,Q}$ for $\mathcal A\subset\mcQ_n$; Section \ref{sec2}

$\partial_{i,s}F_{Q}$ and $\partial_{i,s}F_{m,Q}$: faces of cells.  
$\partial_{i,s}F := \partial_{i,s}F_{F_0}$ and  $\partial_{i,s}F_{m}:=\partial_{i,s}F_{m,F_0}$; Section \ref{sec2}

 $\partial_oF_{Q}:=F_Q\cap \partial Q$ and $\partial_oF_{m,Q}:=F_{m,Q}\cap \partial Q$: outer boundary of cells; Section \ref{sec2} 

$\nu$: normalized $d_I$-dimensional Hausdorff measure on $\bigcup_{n=0}^\infty\bigcup_{Q\in \mcQ_n(F)}\partial_o F_Q$ so that

 \qquad   $\nu(\partial_o F)=1$; Section \ref{sec2}

$\nu_n$: normalized $(d-1)$-dimensional Hausdorff measure on  $\bigcup_{n=0}^\infty\bigcup_{Q\in \mcQ_n(F_m)}\partial_o F_{m,Q}$
 so that

\qquad  $\nu_m(\partial_O F_n)=1$; Section \ref{sec2}

$[f]_{\mu|_A}:= \frac{1}{\mu(A)}\int_A f(x)\mu(dx)$;  Section \ref{sec2} 

$f_n\rightarrowtail f$: Definition \ref{def27}

$\mcE^{(F_{m,Q})} $: Dirichlet form on $F_{m, Q}$;  \eqref{e:3.1}

$(\mcE^{(F_Q)}, \mcF^{(F_Q)})$: Dirichlet form on $F_Q$; \eqref{e:3.2}-\eqref{e:3.3}

$w_Q$: a function in  $C(F)\cap \mcF^{(F)}$ appeared in Lemma \ref{lemma37}

 $\mathcal{S}_Q:=\{Q'\in \mcQ_n(F):Q'\cap Q\neq \emptyset\}$ for $Q\in \mcQ_n(F),n\geq 1$; above Lemma \ref{lemma37} 

$\mu_{\<f\>}$: energy measure of $f\in \sF$; Section \ref{sec34}

$\eth_nF$ and $\eth_nF_m$: collection of level $n$ sub-faces of $\partial_o F$ and $\partial_o F_m$; Section \ref{sec4}

$A\sim A'$ for $A\not=A'\in \eth_k F$: $A\cap A'\neq \emptyset$  or   $A,A'\subset B $  for some $ B \in \eth_{k-1}F$;  \eqref{e:4.2a}  of Section \ref{sec4}

$A\sim A'$  for $A\not= A'\in \eth_kF_m$:    $A\cap A'\neq \emptyset$ or $A,A'\subset B$ for some $B\in \eth_{k-1}F_m$; Section \ref{sec4}
	
$I_k[f]$: discrete energy defined with the average of $f\in L^2(\partial_oF;\nu)$ on $\eth_kF$; Section \ref{sec4}
	
$I_k^{(m)}[f]$: discrete energy defined with the average of $f\in L^2(\partial_oF_m;\nu_m)$ on $\eth_kF_m$; Section \ref{sec4}

$\Lambda_n[f]:=\sum_{k=n}^\infty L_F^{k(d_w-d_f)}I_k[f]$; Section \ref{sec4}

$\Lambda_k^{(m)}[f]:=\sum_{k=n}^\infty \varphi_m(L_F^{-k})I_k^{(m)}[f]$; Section \ref{sec4}

$\Lambda(\partial_o F)$ and $\Lambda^{(m)}(\partial_o F_m)$: Besov space on $\partial_oF$ and $\partial_oF_m$; Section \ref{sec4}

$\mcB_n:=\{Q\in \mcQ_n:Q\cap \partial F_0\neq \emptyset\}$, boundary shells; Section \ref{sec4}

$\mcB_n(A):=\{Q\in \mcB_n:\operatorname{int}(Q)\cap A\neq \emptyset\}$; Section \ref{sec4}  
 
$\mathscr{C}_n$: $(\# \eth_n F)$-dimensional linear subspace of  $C(\partial_o F)$; Lemma \ref{lemma44}
 
$\mathscr{C}_{m,n}$: $(\# \eth_n F)$-dimensional linear subspace of $C(\partial_o F_m)$; Lemma \ref{lemma44}

${\mathcal H} f$: harmonic extension of $f\in C(\partial_o F)$ to $F$; Definition \ref {def45}

$\Theta^{(m)}_n$:  linear maps from $\mathscr{C}_n$ to $W^{1,2}(F_m)\cap C(F_m)$; Lemma \ref{lemma46}

\section{Convergence of functions and forms}\label{sec2}
In this section, we introduce several notions   of convergence of functions (based on \cite{KS} and \cite{Cao}). We also review the definition of Mosco convergence on varying state spaces from  Kuwae and Shioya \cite{KS} adapted to our setting as a {\it GSC} and its approximation domains are embedded in $\R^d$, while \cite[Section 2.2-2.6]{KS} are about general Hilbert spaces.
  The reader is referred to  
 \cite{Mosco} for   Mosco convergence  of symmetric closed bilinear forms
on a common Hilbert space.

Throughout  this paper,   $(\mathcal{X},\rho)$ is  a locally compact separable
 metric space. We use $\operatorname{int}(A)$ to denote the interior of $A\subset \mathcal{X}$ and $\bar A$ its closure. For two subsets $A_1,A_2\subset \mathcal{X}$,  we denote by  $\rho (A_1,A_2):=\inf_{x\in A_1,y\in A_2}\rho (x,y)$ the distance between them. We will simply abbreviate $\rho(\{x\},A)$ to $\rho (x,A)$. The open ball (relative to $A\subset \sX$) of radius $r$ centered at $x$ is denoted by $B_A(x,r)=\{y\in A:\rho (x,y)<r\}$, and we often omit $A$ in the notation when there is no confusion about the underlying space. 

 For a measurable subset $A\subset \sX$, 
we use the notation $l(A)$  to denote the space of real-valued measurable functions on  $A$, and $C(A)=C(A,\rho)$ to denote the space of real-valued continuous functions on $A$ equipped
with the  supremum norm $\|f\|_\infty:=\sup_{x\in A}|f(x)|$.
Denote by  $C_c(A)$  the subspace of $C(A)$ consisting of continuous functions on $A$ with compact supports in the metric space $(A, \rho )$, where the support of $f\in C(A)$ is defined to be the closure of $\{x\in A:f(x)\neq 0\}$. 

If $\mu$ is a Radon measure on $\mathcal{X}$, we denote by $\mu|_A$ the restriction of $\mu$ on  $A$, i.e.  $\mu|_A(B)=\mu(B \cap A)$ for every Borel measurable $ B\subset \sX$. We can identify  $L^2(\sX;\mu|_A)$ with $L^2(A;\mu|_A)$, and abbreviate them to $L^2(A; \mu)$ from time to time.

If $\nu$ is a Radon measure on $A$ and $f\in l(\mathcal{X})$ such that $f|_A\in L^1(A; \nu)$, we use the notation
\[
[f]_{\nu}:=\fint_A f(x)\nu(dx):=\frac{1}{\nu(A)}\int_A f(x)\nu(dx)
\] 
for the $\nu$-weighted average of $f$ on $A$. In particular, if $\mu$ is a Radon measure on $\mathcal{X}$ and $\mu|_A\neq 0$, 
then   $[f]_{\mu|_A}:= \frac{1}{\nu(A)}\int_A f(x)\mu(dx)$. 

\smallskip

We assume the following setting throughout this section.

\smallskip

\noindent \textbf{\emph{Basic setting}}:
 \emph{$(\mathcal{X},\rho)$ is a locally compact separable metric space. Let $\{A_n; n\geq 1\}$ and $A$ be closed subsets of $(\mathcal{X},\rho)$ so  that 
\[
\delta_H(A_n,A)<\infty \hbox{ for each } n\geq 1 \quad \hbox{and} \quad  \lim\limits_{n\to\infty}\delta_H(A_n,A)=0,
\]
where $\delta_{H}(B,B'):=\max \Big\{\sup\limits_{x\in B'}\rho (x,B),\sup\limits_{x\in B} \rho (x,B') \Big\}$
is the Hausdorff metric between $B,B'\subset \mathcal{X}$.  Let $\mu$ be a Radon measure on $A$, and let $\mu_n$ be a Radon measure 
on $A_n$ for $n\geq 1$ 
 so that  $\mu_n$ converges weakly to $\mu$ on $(\mathcal{X},\rho )$ (viewing them as measures on $\mathcal{X}$) as $n\to \infty$,  that is,  
\[
\lim\limits_{n\to\infty}\int_{A_n}f(x)\mu_n(dx)=\int_A f(x)\mu(dx) \quad\hbox{for every }  f\in C_c(\mathcal{X}).
\]}\vspace{0.1cm}

To apply the definition of strong and weak convergence in $L^2$-spaces from \cite{KS}, we introduce a sequence of auxiliary maps 
$\mathcal{I}_n$. However, we will show that the definition of strong and weak convergence in $L^2$ is in fact independent of the choice of $\mathcal{I}_n$. 

\begin{lemma}\label{lemma21}
There is a sequence of Borel measurable maps $\mathcal{I}_n:A_n\to A$ such that 
$$
\lim_{n\to\infty}\sup_{x\in A_n}\rho(\mathcal{I}_n(x),x)=0.
$$
\end{lemma}

\begin{proof}
For each $n$, let $\{x_{n,m}\}_{m=1}^\infty$ be a countable $\frac1n$-net of $(\mathcal{X},\rho )$, that is, 
 $ \bigcup_{m=1}^\infty B_{\mathcal{X}}(x_{n,m}, 1/n) = \sX$. Set 
\[
U_{n,m}=B_{\mathcal{X}}(x_{n,m},  1/n)\setminus  \bigcup_{k=1}^{m-1} B_{\mathcal{X}}(x_{n,k},  1/n) \quad  \hbox{for }  n,m\geq 1. 
\]
Then, $\mathcal{X}=\bigsqcup_{m=1}^\infty U_{n,m}$ for every  $n\geq 1$, where `$\bigsqcup$' means disjoint union. 
Fix $n\geq 1$. For each $m\geq 1$ for which $U_{n,m}\cap A_n\neq\emptyset$, let  $a_{n,m}\in A$ be such that 
$\rho(a_{n, m}, U_{n,m}\cap A_n) < \delta_H(A_n,A)+\frac1n$. 
Define $\mathcal{I}_n(x)=a_{n,m}$ for every $ x\in U_{n,m}\cap A_n  \not=  \emptyset $.  
Since $\sup_{x\in  U_{n,m}\cap A_n  \not=  \emptyset}  \rho(\mathcal{I}_n(x), x) 
 <\delta_H(A_n,A)+\frac3n$, the maps $\{\mathcal{I}_n; n\geq 1\}$ have the desired property.
   \end{proof}

\subsection{Objects related to {\it GSC} }\label{sec21}
In this paper, we focus on {\it GSC}\,s and their approximation domains, which are embedded in $\R^d$. We denote the Euclidean metric on $\R^d$ by $\rho$. We use the notation $x=(x_1,x_2,\cdots,x_d)$ to denote a point in $\R^d$. From time to time, we use the notation $tx$ to denote the point $tx=(tx_1,tx_2,\cdots,tx_d)$, where $t\in \R$ and $x=(x_1,x_2,\cdots,x_d)\in \R^d$. 

In the following, we introduce some notations about the cell structures and the measures on a {\it GSC} $F$ and its approximation domains.\medskip

 \noindent(\textbf{\emph{Cells}}). 
 (a)   For   $Q\in \cup_{n=0}^\infty \mcQ_n$,    define $F_Q:=Q\cap F$.    For $\mathcal{A}\subset\mcQ_n$, from time to time we also write $F_{\mathcal{A}}:=\bigcup_{Q\in \mathcal{A}}F\cap Q$. \emph{We call $F_Q$ an $n$-cell of $F$ if $Q\in \mcQ_n(F)$. }

(b)   For $m\geq 0 $ and $Q\in  \cup_{n=0}^\infty  \mcQ_n$, define  $F_{m,Q}:=Q\cap F_m$.
 For $\mathcal{A}\subset \mcQ_n$,   from time to time we write $F_{m,\mathcal{A}}:=\bigcup_{Q\in \mathcal{A}}F_m\cap Q$. 
 \emph{We call $F_{m,Q}$ an $n$-cell of $F_m$ if $Q\in \mcQ_n(F_m)$.} 

\begin{remark} 
\rm  $\mcQ_n(F_m)=\mcQ_n(F)$ for each $m\geq n\geq 0$. 
\end{remark}

\smallskip

\begin{example}\label{example22} \rm 
\begin{enumerate}
\item[\rm (a)] Let $m\geq 0$ and $\mathcal{A}\subset\mcQ_m(F)$. For $n\geq 0$, set $A_n= F_{n,\mathcal{A}}$ and $A=F_{\mathcal{A}}$. Clearly, $A_n$ converges to $A$ in the Hausdorff metric in $\R^n$. Also, $\mu_n|_{A_n}$ converges weakly to $\mu|_A$ on $\R^d$.  Hence $\{A_n; \mu_n|_{A_n};  n\geq  0\}$ and $A; \mu|_A$ satisfy the {\bf Basic Setting} laid out above Lemma \ref{lemma21}, so does any subsequence $\{n_k; k\geq 1\}$.
	
	\smallskip

 \item[\rm (b)] Throughout this paper, we use  $\partial_o F$ and $\partial_o F_m$ to denote the outer faces of $F$ and $F_m$; that is, 
$\partial_o F := F\cap \partial F_0$ and $\partial_o F_m:= F_m \cap \partial F_0$.
 We can also view $F_n\setminus \partial_o F_n,n\geq 0$ and $F\setminus \partial_o F$ as subsets in the metric space $(F_0\setminus \partial F_0,\rho )$. 
\end{enumerate} 
\end{example}

\medskip

We introduce some more notations about the faces of cells. 
 
\smallskip

\noindent(\textbf{\emph{Faces of cells}}). (a). Let $n\geq 0$, $Q\in\mcQ_n(F)$, $i\in \{1,2,\cdots,d\}$ and $s\in \{0,1\}$. We call $\partial_{i,s}F_Q:=\Psi_Q\big(F\cap \{x_i=s\}\big)$ a face of $F_Q$. When $Q=F_0$, we simply write  $\partial_{i,s}F$ for $\partial_{i,s}F_Q$. In addition, we define $\partial_o F_Q :=\Psi_Q (\partial_o F)= F_Q \cap \partial Q$. 

\smallskip

(b). Let $n\geq 0$, $m\geq 0$, $Q\in \mcQ_n(F_m)$, $i\in \{1,2,\cdots,d\}$ and $s\in \{0,1\}$. 
We call $\partial_{i,s}F_{m,Q} :=\Psi_Q\big(F_{(m-n)\vee 0}\cap\{x_i=s\}\big)$ a face of $F_{m,Q}$. 
When $Q=F_0$, we simply write  $\partial_{i,s}F_m $ for $\partial_{i,s}F_{m,Q}$. 
In addition, 
we define $\partial_o F_{m,Q}:= \Psi_Q (\partial_o F_{(m-n)\vee 0})=F_{m, Q}\cap \partial Q$. 

\medskip

\noindent(\textbf{\emph{Dimension $d_I$}}).
Let $m_I:=\#\{Q\in \mcQ_1(F):Q\cap\{x_1=0\}\neq \emptyset\}$ and define $d_I:=\frac{\log m_I}{ \log L_F}$, which is 
the Hausdorff dimension of $\partial_o F$. Note that $d_I=1$ when $d=2$.\\  

\noindent(\textbf{Measure on faces}). 
Let $\nu$ be the normalized $d_I$-dimensional Hausdorff measure on \hfill \break 
$\bigcup_{n=0}^\infty\bigcup_{Q\in \mcQ_n(F)}\partial_o F_Q$ 
so that $\nu(\partial_oF)=1$. For $m\geq 0$, let $\nu_m$ be the normalized $(d-1)$-dimensional Hausdorff measure on $\bigcup_{n=0}^\infty\bigcup_{Q\in \mcQ_n(F_m)}\partial_o F_{m,Q}$ so that $\nu_m(\partial_o F_m)=1$.

\begin{example}\label{example23}
For each $Q\in \mcQ_n(F)$,  the {\bf Basic setting} is satisfied for $\partial_o F_{m,Q}$, $\nu_m|_{\partial_o F_{m,Q}}$ for $m\geq 0$ and $\partial_o  F_Q$, $\nu|_{\partial_o F}$;  the {\bf Basic setting} is satisfied for each $\partial_{k,s} F_{m,Q}$, $\nu_m|_{\partial_{k,s} F_{m,Q}}$ for $ m\geq 0$ and $\partial_{k,s} F_Q$, $\nu|_{\partial_{k,s} F}$ with $i\in \{1,\cdots,d\}$ and $k\in \{0,1\}$ as well. 
The same holds for each subsequence $\{m_k; k\geq 1\}$.
\end{example}

\medskip

\subsection{Convergence of functions}
 
In this subsection, we fix a sequence $\mathcal{I}_n$ as in Lemma \ref{lemma21}.

\begin{lemma}\label{lemma24}
	$\mu_n\circ \mathcal{I}_n^{-1}$ converges weakly to $\mu$ on $(A,\rho )$. As a consequence, 
	\[\lim\limits_{n\to\infty}\|f\circ\mathcal{I}_n\|_{L^2(A_n; \mu_n)}=\|f\|_{L^2(A; \mu)}  \quad\hbox{for }  f\in C_c(A).\] 
\end{lemma}
\begin{proof}
	Let $f\in C_c(A)$. Then by using the Tietze extension theorem and local compactedness, we can find $g\in C_c(\mathcal{X})$ such that $g|_{A}=f$. By the assumption $\mu_n\Rightarrow \mu$ on $(\mathcal{X},\rho )$, we have 
	\[
	\lim\limits_{n\to\infty}\int_{A_n}g(x)\mu_n(dx)=\int_Ag(x)\mu(dx)=\int_A f(x)\mu(dx).
	\] 
	In addition, by local compactness, we can find a compact neighborhood $D$ of the support of $g$ in $\mathcal{X}$ such that $g(x)-f\circ \mathcal{I}_n(x)=0$ for any large enough $n$ and $x\in A_n\setminus D$. Hence, 
	\[
	\lim_{n\to\infty} \Big| \int_{A_n}\big(g(x)-f\circ \mathcal{I}_n(x)\big)\mu_n(dx) \Big|\leq \lim\limits_{n\to\infty}\mu_n(D)\cdot\|g|_{A_n}-f\circ \mathcal{I}_n\|_\infty=0,
	\]
	where we use the facts that $\limsup\limits_{n\to\infty}\mu_n(D)\leq \mu(D)$ and that $\lim\limits_{n\to\infty}\|g|_{A_n}-f\circ \mathcal{I}_n\|_\infty=0$ by the uniform continuity of $g$. Combining the above equalities, we see that 
	\[
	\lim\limits_{n\to\infty}\int_A f(x)\mu_n\circ \mathcal{I}_n^{-1}(dx)=\lim\limits_{n\to\infty}\int_{A_n} f\circ \mathcal{I}_n(x)\mu_n(dx)=\lim\limits_{n\to\infty}\int_{A_n}g(x)\mu_n(dx)=\int_Af(x)\mu(dx).
	\]
	This finishes the proof of the lemma.
\end{proof}

The following definition of  strong and weak convergence is adapted from \cite[Section 2.2]{KS}, where these notions are defined 
for general Hilbert spaces.

\begin{definition}\label{def25}
	Let $f_n\in L^2(A_n; \mu_n),n\geq 1$ and let $f\in L^2(A; \mu)$. 
	\begin{enumerate}
	\item[\rm (a)] We say $f_n\to f$ strongly in $L^2$ if and only if there is a sequence $\{u_j; j\geq 1\}\subset C_c(A)$ such that $\lim\limits_{j\to\infty}\|u_j-f\|_{L^2(A; \mu)}=0$ and
	\begin{equation}\label{eqn22}
		\lim\limits_{j\to\infty}\limsup\limits_{n\to\infty}\|f_n-u_j\circ \mathcal{I}_n\|_{L^2(A_n; \mu_n)}=0.
	\end{equation}
	
	\item[\rm (b)]  We say $f_n\to f$ weakly in $L^2$ if and only if
	\[
	\lim\limits_{n\to\infty}(f_n,g_n)_{L^2(A_n; \mu_n)}=(f,g)_{L^2(A; \mu)}.
	\]
	for any $g_n\in L^2(A_n; \mu_n)$ that converges strongly in $L^2$ to $g\in L^2(A; \mu)$ as $n\to \infty$.
	\end{enumerate} 
\end{definition}

 \medskip
 
In the following lemma, we see that the above definition of strong convergence reflects the fact that $A_n,n\geq 1$ and $A$ are embedded in a same space. We will show in Lemma \ref{lemma26} that the definition of strong convergence in Definition \ref{def25} 
 is independent of the choice of the maps $\{\mathcal{I}_n; n\geq 1\}$.

\begin{lemma}\label{lemma26}
	Let $g\in C_c(\mathcal{X})$, then $g|_{A_n}\to g|_A$ strongly in $L^2$. Moreover, for  $f_n\in L^2(A_n; \mu_n),n\geq 1$ and $f\in L^2(A; \mu)$, $f_n\to f$ strongly in $L^2$ if and only if there is a sequence $g_j\in C_c(\mathcal{X}),j\geq 1$ such that $\lim\limits_{j\to\infty}\|g_j|_A-f\|_{L^2(A; \mu)}=0$ and
	\begin{equation}\label{eqn23}
		\lim\limits_{j\to\infty}\limsup\limits_{n\to\infty}\|f_n-g_j|_{A_n}\|_{L^2(A_n; \mu_n)}=0.
	\end{equation}
\end{lemma}

\begin{proof}
	We first prove the first statenent of the lemma. Let $f=g|_A$, then one can see that 
	\begin{equation}\label{eqn24}
		\lim\limits_{n\to\infty}\|g|_{A_n}-f\circ\mathcal{I}_n\|_{L^2(A_n;\nu_n)}\leq \lim\limits_{n\to\infty} \sqrt{\mu_n(D)}\cdot\|g|_{A_n}-f\circ \mathcal{I}_n\|_\infty=0
	\end{equation} 
	for some compact set $D\subset \mathcal{X}$ just as in the proof of Lemma \ref{lemma24}. Hence, by taking $u_j=f$ for each $j\geq 1$ in \eqref{eqn22}, we can see that $g|_{A_n}\to g|_A$ strongly in $L^2$. 
	
	\smallskip
	
	For the second statement of the lemma, given $u_j\in C_c(A)$, by  the Tietze extension theorem there is $g_j\in C_c(\mathcal{X})$ so that  $g_j|_A=u_j$; conversely, given $g_j\in C_c(\mathcal{X})$,   $u_j:=g_j|_A \in C_c(A)$. Applying (\ref{eqn24}) to $g_j$ and $u_j$, we have  $\lim\limits_{n\to\infty}\|g_j|_{A_n}-u_j\circ\mathcal{I}_n\|_{L^2(A_n; \mu_n)}=0$. Thus 
	\[
	\limsup_{n\to\infty}\|f_n-g_j|_{A_n}\|_{L^2(A_n;\nu_n)}=\limsup_{n\to\infty}\|f_n-u_j\circ \mathcal{I}_n\|_{L^2(A_n; \mu_n)},
	\]
	which shows the equivalence of \eqref{eqn22} and \eqref{eqn23}.
\end{proof}

The following is a natural analog of locally uniform convergence of functions. There is also a version of Arzel\`a–Ascoli theorem for it.

\begin{definition}\label{def27}
Let $f_n\in l(A_n),n\geq 1$ and $f\in l(A)$. We say that $f_n\rightarrowtail f$ if and only if $\lim_{n\to \infty} f_n(x_n)= f(x)$ holds for any sequence $x_n\in A_n,n\geq 1$ and $x$ such that $x_n\to x$. 
\end{definition}

\begin{remark}\label{R:2.9}
 \rm If $A_n=A$ for  all $ n\geq 1$ and $\{f_n, n\geq 1\} \cup \{f\} \subset   C(A)$,   then $f_n \rightarrowtail f$ if and only if 
$f_n$ converges to $f$    locally uniformly   $A$. 
\end{remark}

\begin{lemma}\label{lemma28}
	Let $f_n\in C(A_n),n\geq 1$. 
	
	\begin{enumerate}
	\item[\rm (a)]  If $\{f_n\}_{n\geq 1}$ is locally uniformly bounded and equicontinuous, that is,  if 
	$\sup\limits_{n\geq 1}\|f_n|_K\|_\infty<\infty $  and 
 	$	\lim\limits_{\delta\to 0}\sup\{|f_n(x)-f_n(y)|:n\geq 1,\ x,y\in A_n\cap K,\rho(x,y)<\delta\}=0$
		for every  compact $K\subset \mathcal{X}$, then there are $f\in C(A)$ and a subsequence $\{f_{n_k},k\geq 1\}$ so that $f_{n_k}\rightarrowtail f$.  
	
	\smallskip
	
	\item[\rm (b)] If $f_n\rightarrowtail f$ for some $f\in l(A)$, then $f\in C(A)$. Let $g\in C(\mathcal{X})$ such that $g|_A=f$. Then, we can find $g_n\in C(\mathcal{X}),n\geq 1$ so that $g_n|_{A_n}=f_n$ for every $n\geq 1$  and $g_n\to g$ locally uniformly on $\mathcal{X}$ (i.e. $\|g_n|_K-g|_K\|_\infty\to 0$ for each compact $K\subset \mathcal{X}$) as $n\to \infty$.
	
	\smallskip

	\item[\rm (c)] If there is $f\in C(A)$ such that for each subsequence $f_{n_k}$ there is a further subsequence $f_{n_{k(l)}}$ such that $f_{n_{k(l)}}\rightarrowtail f$ as $l\to\infty$, then $f_n\rightarrowtail f$. 
	\end{enumerate}
\end{lemma}

\begin{proof}
	(a) follows from the same proof of \cite[Lemma 2.2]{Cao}, where the separability of $\sX$ is used. 
	
	(b). The first claim follows from \cite[Lemma 2.2(b)]{Cao}. The second claim also follows by a similar argument of \cite[Proposition 2.3]{Cao}. Let 
	\[
	\wt {A}:=\big(\bigcup_{n=1}^\infty (\{1/n\}\times A_n)\big)\bigcup (\{0\}\times \mathcal{X}),
	\]  
	which is a closed subset of $\wt {\mathcal{X}}:=[0,1]\times \mathcal{X}$ equipped with the product topology. Define $\wt{f}\in C(\wt {A})$ by
	\[
	\wt{f}(t,x)=\begin{cases}
		f_n(x),&\text{ if }t=\frac1n,x\in A_n,\\
		g(x),&\text{ if }t=0,x\in \mathcal{X}.
	\end{cases}
	\]
    Then, by Tietze extension theorem, we can find $\wt{g}\in C(\wt{\mathcal{X}})$ such that $\wt{g}|_{\wt{A}}=\wt{f}$. It suffices to take
    $g_n(x)=\wt{g}(1/n,  x)$ for $n\geq 1$. 
    
    (c) is proven by contradiction. If $f_n\not\rightarrowtail f$, we can find $x_n\in A_n,n\geq 1$ and $x\in A$ such that $x_n\to x$ but  $f_n(x_n)\not\to f(x)$. Then, there is a subsequence $\{n_k; k\geq 1\}$ and $\varepsilon>0$ such that $|f_{n_k}(x_{n_k})-f(x)|>\varepsilon$ for every $ k\geq 1$.
  This leads to a contradiction to the assumption of (c) as  $\{f_{n_k}; k\geq 1\}$ does not have any subsequence  $f_{n_{k(l)}}\rightarrowtail f$. 
\end{proof}

 Since $(\sX, \rho )$ is a locally compact separable metric space, 
there is an increasing sequence of relatively compact open subsets $\{B_j; j\geq 1\}$ so that $\bigcup_{j=1}^\infty  B_j =\mathcal{X}$.

\begin{lemma}\label{lemma29}
 \begin{enumerate}
\item[\rm (a)]    Let  $f_{m,n}\in C(A_n)$,  $n,m\geq 1$,   $f_m\in C(A)$, $ m\geq 1$,  and  $f\in C(A)$. Suppose that 
 $f_{m,n}\rightarrowtail f_m$ as $n\to\infty$ for each $m\geq 1$, and $f_m$ converges to $f$ locally uniformly as $m\to\infty$. Then there exist 
  $\{m(n) ; n\geq 1\}\subset \N$ so that $m(n)\to\infty$ and $f_{m(n),n}\rightarrowtail f$ as $n\to\infty$. 

\smallskip

\item[\rm (b)]   Let $f_n\in C(A_n)$,  $n\geq 1$.  
If $f_n\rightarrowtail f \in C(A)$ and $\lim\limits_{j\to\infty}\limsup\limits_{n\geq 1}\|f_n\|_{L^2(A_n\setminus \bar B_j;\mu_n)}=0$, then $f\in L^2(A;\mu)$ and  $f_n\to f$ strongly in $L^2$.
\end{enumerate} 
\end{lemma}

\begin{proof}
(a).  By Tietze extension theorem and Lemma \ref{lemma28} (b),
there are $\{g;  g_m, m\geq 1\} \subset  C(\mathcal{X})$ and $\{g_{m,n}, m\geq 1\}\in C(A_n)$ so that 
   $g|_A=f$,    $g_m|_A=f_m$ for $m\geq 1$, and    $g_{m,n}|_{A_n}=f_{m,n}$ for $n,m\geq 1$.
  In addition, 
\begin{eqnarray*}
	&g_m\to g\text{ locally uniformly as }m\to\infty,\\
	&g_{m,n}\to g_m\text{ locally uniformly as }n\to\infty  \ \hbox{for every }  m\geq 1.
\end{eqnarray*} 
Note that  the metric $\rho_\infty(h,h'):=\sum_{j=1}^\infty 2^{-j}\big(\|(h-h')|_{\bar B_j}\|_\infty\wedge 1\big)$ characterizes the   locally uniform convergence on $\sX$. Define 
$m(n):=\max\{1\leq m\leq n: \rho_\infty(g_{m,n},g_m)\leq \frac{1}{m}\}$. Then $m(n)\to\infty$ and $ \lim_{n\to \infty} \rho_\infty(g_{m(n),n},g) = 0$. This establishes (a). 

(b). We apply Lemma \ref{lemma28} (b) to $f_n\rightarrowtail f$ as $n\to\infty$ to find $g_n\in C(\sX)$ and $g\in C(\sX)$ so that $f_n=g_n|_{A_n}$ for $n\geq 1$, $f=g|_A$ and $g_n\to g$ locally uniformly as $n\to \infty$. For each $j\geq 1$, we fix $u_j\in C_c(\mathcal{X})$ such that $u_j|_{\bar B_j}=1$ and $0\leq u_j\leq 1$. Denote the support of $u_j$ by $D_j$. Then, 
\begin{equation}\label{eqn23a}
\begin{aligned}
		\lim_{n\to\infty}\|f_n\,(u_j|_{A_n})-(g\,u_j)|_{A_n}\|_{L^2(A_n; \mu_n)}
		=&\lim_{n\to\infty}\|(g_n\,u_j)|_{A_n}-(g\,u_j)|_{A_n}\|_{L^2(A_n; \mu_n)}\\
		\leq &\lim_{n\to\infty}\|g_n|_{D_j}-g|_{D_j}\|_{\infty}\,\sqrt{\mu_n(D_j)}=0.
\end{aligned}
\end{equation}
This in particular implies that $\sup_{n\geq 1}\|f_n\|_{L^2(A_n;\mu_n)}<\infty$. Combining \eqref{eqn23a} with the assumption  $\lim\limits_{j\to\infty}\limsup\limits_{n\geq 1}\|f_n\|_{L^2(A_n\setminus \bar B_j;\mu_n)}=0$, we get
\begin{equation}\label{eqn24a}
\lim\limits_{j\to\infty}\limsup\limits_{n\geq 1}\|f_n-(g\,u_j)|_{A_n}\|_{L^2(A_n; \mu_n)}=\lim\limits_{j\to\infty}\limsup\limits_{n\geq 1}\|f_n-f_n\,(u_j|_{A_n})\|_{L^2(A_n; \mu_n)}=0.
\end{equation}
Next for each $j\geq 1$, $\|(g\,u_j)|_A\|_{L^2(A;\mu)}=\lim\limits_{n\to\infty}\|f_n(u_j|_{A_n})\|_{L^2(A_n;\mu_n)}\leq\sup_{n\geq 1}\|f_n\|_{L^2(A_n;\mu_n)}$  by \cite[Lemma 2.1(2),(5)]{KS}, Lemma \ref{lemma26} and \eqref{eqn23a}.
 Hence, $\sup_{j\geq 1}\|(gu_j)|_A\|_{L^2(A;\mu)}<\infty$, which implies that $f=g|_A\in L^2(A;\mu)$ and $(gu_j)|_{A}\to f$ in $L^2(A;\mu)$ norm. It follows by \eqref{eqn24a} and Lemma \ref{lemma26} that $f_n\to f$ strongly in $L^2$.
\end{proof}

When $\sX$ is compact, Lemma \ref{lemma29}(b) in particular shows that uniform convergence `$\rightarrowtail$' implies strong convergence in $L^2$. It is also known that strong convergence in $L^2$ implies weak convergence in $L^2$ 
by \cite[Lemma 2.1 (4)]{KS}.

\subsection{Convergence of quadratic forms}
Recall that the Basic Setting is in force, under which the sequence of closed sets $     \{A_n; n\geq 1\}$ in $(\sX, \rho )$  converges to $A$ 
in the Hausdorff metric and the sequence of measures $\mu_n$ on $A_n$ converges weakly to the measure $\mu$ on $A$. 
The following definition of Mosco convergence is taken from \cite[Definitions 2.8 and 2.11]{KS}.

\begin{definition}[Mosco convergence]\label{def28}
	Let $\bar{\R}=\R\cup\{+\infty\}\cup\{-\infty\}$ be the set of extended real numbers. Suppose that 
	 $\sE^{(n)}:L^2(A_n; \mu_n)\to \bar{\R}$, $n\geq 1$,  and $\sE:L^2(A; \mu)\to \bar{\R}$. 
	 	We say $\sE^{(n)}$ is Mosco convergent to $\sE$ if and only if the following hold.
	
	\begin{enumerate}
\item[\rm (M1)]   For any $f_n\in L^2(A_n; \mu_n)$ that converges weakly in $L^2$ to $f\in L^2(A; \mu)$ as $n\to \infty$, we have 
	\[\liminf_{n\to\infty} \sE^{(n)}(f_n)\geq \sE(f).\]
	
	\item[\rm (M2)]  For each $f\in L^2(A; \mu)$, there exists $f_n\in L^2(A_n; \mu_n)$ that converges strongly in $L^2$ to  $f$ as $n\to \infty$ with 
	\[\limsup_{n\to\infty} \sE^{(n)}(f_n)\leq \sE(f).\]
	\end{enumerate} 
\end{definition}

We are in particular interested in quadratic forms. It is well-known that there is a one to one correspondence between non-negative lower-semicontinuous quadratic forms (with extended real values) on a Hilbert space $H$ and closed symmetric non-negative definite bilinear forms (see \cite[Section 1 (b)]{Mosco}) described as follows. Let $\sE:H  \to \R$ be a non-negative lower-semicontinuous quadratic forms, then one can define a closed symmetric non-negative definite symmetric bilinear form $(\sE,\hbox{Dom}(\sE))$ by the parallelogram law:
set $\hbox{Dom} (\sE)=\{f\in H:\sE(f)< + \infty\}$ and 
$$
\sE(f,g):=\frac{1}{4}\Big(\sE(f+g)-\sE(f-g)\Big) \quad \hbox{for }  f,g\in \hbox{Dom} (\sE).
$$
 Conversely, given a closed symmetric non-negative definite bilinear form $(\sE,\hbox{Dom}(\sE))$, we can define a non-negative lower-semicontinuous quadratic forms by 
 $$
 \sE(f):= \begin{cases} \sE(f,f) \quad &\hbox{if } f\in \hbox{Dom}(E), \cr
  + \infty   &\hbox{if  } f\in H\setminus \hbox{Dom}(\sE). 
   \end{cases}
   $$

\medskip

Throughout this paper, we do not distinguish in notation between a non-negative symmetric closed bilinear form and its associated non-negative lower-semicontinuous quadratic form. The following result taken from \cite{KS} extends the characterization of the Mosco convergence of the non-negative symmetric closed bilinear forms in terms of  the strong  semigroup (or resolvent) convergence from on a common Hilbert space to  the setting  on varying Hilbert spaces. See also \cite[Appendix]{CKK} and \cite[Theorem 2.5]{K} for related work.

\begin{theorem}[Theorem 2.4 of \cite{KS}]  \label{thm211}
	Let $\sE$ be a densely defined non-negative lower-semicontinuous quadratic form on $L^2(A; \mu)$, let $\{T_t ;  t\geq 0\}$ be the associated strongly continuous semigroup of symmetric contraction operators on $L^2(A; \mu)$, and let $U_\lambda=\int_0^\infty e^{-\lambda t}T_{t}dt$,
	$\lambda>0$,  be the resolvent operators.
	 	For $n\geq 1$, let $\sE^{(n)}$ be a densely defined non-negative lower-semicontinuous quadratic form on $L^2(A_n; \mu_n)$, and let $\{T^{(n)}_t; t\geq 0\}$ be the associated strongly continuous semigroup of symmetric contraction operators on $L^2(A_n; \mu_n)$, and let $U^{(n)}_\lambda=\int_0^\infty e^{-\lambda t}T^{(n)}_{t}dt$,  $\lambda>0$,  be the resolvent operators.
	 	The following statements are equivalent to each other.  
	
	\begin{enumerate}
	\item[\rm (a)] $\sE^{(n)}$ is Mosco convergent to $\sE$.
	  	
	\item[\rm (b)] $T_t^{(n)}$ strongly converges in $L^2$ to $T_t$ for some $t>0$. 
	
	\item[\rm (c)]  $T_t^{(n)}$ strongly converges in $L^2$ to $T_t$ for any $t\geq 0$.
	
	\item[\rm (d)]  $U_\lambda^{(n)}$ strongly converges in $L^2$ to $U_\lambda$ for some $\lambda>0$.
	
	\item[\rm (e)]    $U_\lambda^{(n)}$ strongly converges in $L^2$ to $U_\lambda$ for any $\lambda>0$.
	\end{enumerate} 
Here, for $O_n:L^2(A_n; \mu_n)\to L^2(A_n; \mu_n),n\geq 1$ and $O:L^2(A; \mu)\to L^2(A; \mu)$, we say $O_n$ strongly converge in $L^2$ to $O$  if $O_nf_n\to Of$ strongly in $L^2$ for any $f_n\in L^2(A_n; \mu_n),n\geq 1$ and $f\in L^2(A; \mu)$ such that $f_n\to f$ strongly in $L^2$.
\end{theorem}

\smallskip

  We end this section with a criteria of strong convergence of operators. 

\begin{proposition}\label{prop211}
Suppose that $U:L^2(A;\mu)\to L^2(A;\mu)$ is  a bounded operator, and for each $n\geq 1$,
 $U^{(n)}:L^2(A_n;\mu_n)\to L^2(A_n;\mu_n)$ is  a bounded operator  with  $\sup_{n\geq 1}\|U^{(n)}\|_{L^2\to L^2}<\infty$, 
 where $\|U^{(n)}\|_{L^2\to L^2}:=\sup\{\|U^{(n)}f\|_{L^2(A_n;\mu_n)}:f\in L^2(A_n;\mu_n),\ \|f\|_{L^2(A_n;\mu_n)}=1\}$. 
Then, $U^{(n)}$ converges strongly in $L^2$ to $U$ if and only if $U^{(n)}(g|_{A_n})\to U(g|_A)$ strongly in $L^2$ for every  $g\in C_c(\mathcal{X})$.
\end{proposition}

\begin{proof}
The `only if' part follows immediately from the definition of strong convergence of operators and   Lemma \ref{lemma26}.  

It remains to prove the `if' part. Let $f\in L^2(A;\mu)$ and $f_n\in L^2(A_n;\mu_n),n\geq 1$ such that $f_n\to f$ strongly in $L^2$.
 We need to show $U^{(n)}f_n\to Uf$ strongly in $L^2$. Since $f_n\to f$ strongly in $L^2$, by Lemma \ref{lemma26}, we can find a sequence $g_j\in C_c(\mathcal{X}),j\geq 1$ such that 
\begin{eqnarray*}
&\lim\limits_{j\to\infty}\|g_j|_A-f\|_{L^2(A;\mu)}=0,\\
&\lim\limits_{j\to\infty}\limsup\limits_{n\to\infty}\|g_j|_{A_n}-f_n\|_{L^2(A_n;\mu_n)}=0.
\end{eqnarray*}
Then, since $U$ is bounded and $\sup_{n\geq 1}\|U^{(n)}\|_{L^2\to L^2}<\infty$, we have
\begin{eqnarray}
\label{e:2.4}&\lim\limits_{j\to\infty}\|U(g_j|_A)-Uf\|_{L^2(A;\mu)}=0,\\
\label{e:2.5}&\lim\limits_{j\to\infty}\limsup\limits_{n\to\infty}\|U^{(n)}(g_j|_{A_n})-U^{(n)}f_n\|_{L^2(A_n;\mu_n)}=0.
\end{eqnarray}
By the assumption (note that we are proving the `if' part), for each $j\geq 1$, we have $U^{(n)}(g_j|_{A_n})\to U(g_j|_A)$ strongly in $L^2$ as $n\to\infty$. Hence, by using Lemma \ref{lemma26} again, for each $j\geq 1$, we can find $h_j\in C_c(\mathcal{X})$ such that 
\begin{eqnarray}
\label{e:2.6}&\|h_j|_A-U(g_j|_A)\|_{L^2(A;\mu)}\leq 1/j,\\
\label{e:2.7}&\limsup\limits_{n\to\infty}\|h_j|_{A_n}-U^{(n)}(g_j|_{A_n})\|_{L^2(A_n;\mu_n)}\leq 1/j.
\end{eqnarray}
Combining \eqref{e:2.4} and \eqref{e:2.6} yields $\lim\limits_{j\to\infty}\|h_j|_A-Uf\|_{L^2(A;\mu)}=0$; while    \eqref{e:2.5} together with
 \eqref{e:2.7} gives  $\lim\limits_{j\to\infty}\limsup\limits_{n\to\infty}\|h_j|_{A_n}-U^{(n)}f_n\|_{L^2(A_n;\mu_n)}=0$. Hence, $U^{(n)}f_n\to Uf$ strongly in $L^2$ by Lemma \ref{lemma26}, noticing that $h_j\in C_c(\mathcal{X})$ for each $j\geq 1$.
\end{proof}

\section{Dirichlet forms on {\it GSC}s }\label{sec3}

In this section, we first review some well-known properties of the Dirichlet form 
of Brownian motion on $F$ and that of approximating reflected Brownian motions on  $F_n$ for $n\geq 0$, and recall the uniform elliptic Harnack inequalites from \cite{BB5}.
We then present lower bound estimates of effective resistances between a cell and the complement of its neighborhood and establish 
the Mosco convergence of $\alpha_m \sE^{(F_m)}$ to $\sE^{(F)}$ for some sequence of positive numbers 
$\{\alpha_m; m\geq 1\}$ that are bounded between two positive numbers.

\medskip

Recall the definitions of cells, and measures $\mu$ and $\mu_n$ on $F$ and $F_n$, respectively,
from Subsection \ref{sec21}.

\medskip

\noindent(\textbf{Dirichlet forms on $F_m$}). For each $n\geq 0$ and $Q\in \mcQ_n(F_m)$,   let $\big(\mcE^{(F_{m,Q})},W^{1,2}(F_{m,Q})\big)$ be the renormalized Dirichlet form on $F_{m,Q}$ on $L^2(F_{m,Q};\mu_m)$ defined by 
\begin{equation}\label{e:3.1} 
\mcE^{(F_{m,Q})}(f,g)=L_F^{(d_w-2)m}\int_{F_{m,Q}}\nabla f(x)\nabla g(x)\mu_m(dx)\quad\hbox{for }   f,g\in W^{1,2}(F_{m,Q}).
\end{equation}
In particular, when $Q=F_0$,  
\[ 
\mcE^{(F_m)}(f,g)=L_F^{(d_w-2)m}\int_{F_m}\nabla f(x)\nabla g(x)\mu_m(dx)
\quad\hbox{for }  f,g\in W^{1,2}(F_m).
\]

\medskip

\noindent(\textbf{Dirichlet forms on $F$}). Recall that $(\mcE^{(F)},\mcF^{(F)})$ is the
 strongly local, regular, irreducible locally symmetric Dirichlet form on $L^2(F; \mu)$
 associated with $X^{(F)}$ that has unit expected time of its 
 first visit to the faces of $\partial F_0$ not containing ${\bf 0}$ when starting from ${\bf 0}$. 
For each $Q\in\mcQ_n(F),n\geq 0$, we define $(\mcE^{(F_Q)},\mcF^{(F_Q)})$ by
\begin{eqnarray}
	&\mcF^{(F_Q)}=\{f\in L^2(F_Q;\mu):f\circ \Psi_Q\in \mcF^{(F)}\},  \label{e:3.2}\\
	&\mcE^{(F_Q)}(f,g)=L_F^{n(d_w-d_f)}\mcE^{(F)}(f\circ \Psi_Q,g\circ \Psi_Q)
	\quad\hbox{for } f,g\in \mcF^{(F_Q)}.  \label{e:3.3}
\end{eqnarray}

\medskip   

\begin{remark} \rm  With slight abuse of notations, for $f\in l(F)$ and $Q\in \mcQ_n(F),n\geq 0$, we abbreviate $f\circ (\Psi_Q|_F)$ to $f\circ \Psi_Q$; similarly, for $f\in l(F_m),m\geq 0$ and $Q\in \mcQ_n(F),n\geq 0$, we abbreviate $f\circ (\Psi_Q|_{F_{(m-n)\vee 0}})$ to $f\circ \Psi_Q$.
\end{remark}

\begin{lemma}\label{lemma31}
$(\mcE^{(F)},\mcF^{(F)})$ is self-similar: 
\begin{equation}\label{eqn31}
	\mcF^{(F)}\cap C(F)=\left\{f\in C(F):\  f|_{F_Q}\in\mcF^{(F_Q)}\text{ for any } Q\in \mcQ_1(F) \right\},
\end{equation}
and
\begin{equation}\label{eqn32}
	\mcE^{(F)}(f)=\sum_{Q\in\mcQ_1(F)}
	\mcE^{(F_Q)}(f |_Q)  \quad \text{for any } f\in\mcF^{(F)}.
 \end{equation}
\end{lemma}

\begin{proof} The lemma follows from the construction of Kusuoka-Zhou \cite{KZ} and Corollary 1.3 of \cite{BBKT}. The second half of Kusuoka-Zhou's paper is under  the strong recurrence assumption  that $d_f<d_w$. However this condition can be dropped  since we know the elliptic Harnack inequality holds by the coupling argument \cite{BB5} on any GSC.
 One can also use the sub-Gaussian heat kernel estimate and Lemma \ref{lemma38} of this paper that $d_f-d_w>d_I$  to construct a strongly local, regular, irreducible, locally symmetric Dirichlet form on $F$  satisfying (\ref{eqn31}) and (\ref{eqn32}) using the method of \cite[Section 4]{CQ2}. Then, by the uniqueness theorem from \cite{BBKT}, we know it is the same as $(\mcE^{(F)},\mcF^{(F)})$.
 \end{proof}

\begin{remark} \rm    With some abuse of notations, we write $\mcE^{(F_Q)}(f)$ instead of $\mcE^{(F_Q)}(f|_{F_Q})$ for short. Similar notation will be used for $\mcE^{(F_{m,Q})}$ later in this paper.
  \end{remark}

In the rest of the paper,  $\big((X^{(F)}_t)_{t\geq 0},\mathbb{P}^{(F)}_x\big)$ denotes the diffusion process associated with $(\mcE^{(F)},\mcF^{(F)})$ on $L^2(F;\mu)$; for $m\geq 0$,  and  $\big((X^{(F_m)}_t)_{t\geq 0},\mathbb{P}^{(F_m)}_x\big)$ denotes the diffusion process associated with $(\mcE^{(F_m)},\mcF^{(F_m)})$ on $L^2(F_m; \mu_m)$. Since the two-sided  sub-Gaussian heat kernel estimates  \eqref{e:1.2} holds for 
 $\big((X^{(F)}_t)_{t\geq 0},\mathbb{P}^{(F)}_x\big)$  and the two-sided Gaussian heat kernel estimates holds for $\big((X^{(F_m)}_t)_{t\geq 0},\mathbb{P}^{(F_m)}_x\big)$, $m\geq 0$, the diffusion processes $X^{(F)}$ and $X^{(F_m)}$ are Feller processes having strong Feller property. In particular, these processes can start from every point in $F$ and $F_m$, respectively.

\subsection{Uniform elliptic Harnack inequality (EHI)}\label{sec31}

Barlow and Bass  \cite[Theorem 1.1]{BB5}  established a scale-invariant  uniform elliptic Harnack principle on $F_m,m\geq 0$
by a coupling argument, The same idea works on $F$, see  \cite[Proposition 4.22]{BBKT}. We  state their results here.

\smallskip 

\noindent(\textbf{\emph{Harmonic functions}}). Let $(\mcE,\mcF)$ be a regular Dirichlet form on $L^2(\mathcal{X};\mu)$. Let $f\in \mcF$ and let $U\subset \mathcal{X}$ be an open subset. We say $f$ is $\mcE$-harmonic in $U$ if $\mcE(f,g)=0$ for every $g\in C_c(\mathcal{X})\cap \mcF$ whose support  is contained in $U$.

\medskip

There is   a well-known probabilistic characterization of harmonic functions. Let  $\big((X_t)_{t\geq 0},\mathbb{P}_t\big)$ be the Hunt process associated with $(\mcE,\mcF)$,    
\begin{equation}\label{e:3.4}
\sigma_{A}:=\inf\{t\geq 0: X_t\in A\}  \quad \hbox{and} \quad \dot \sigma_{A}:=\inf\{t >  0: X_t\in A\}
\end{equation} 
be the entry time and hitting time of $A$, respectively. Then, if $h\in \mcF\cap L^\infty(\mathcal{X}; \mu)$ is $\mcE$-harmonic in $U$ and $h$ is quasi-continuous, then $h(x)=\mathbb{E}_x[h(X_{\sigma_{\mathcal{X}\setminus U}})]$ for q.e. $x\in E_U:=\{y\in U:\mathbb{P}_y(\sigma_{\mathcal{X}\setminus U}<\infty)=1\}$. See \cite[Proposition 2.5]{BBKT} for a proof.  A more general equivalence result between the analytic and probabilistic notions of harmonicity can be found in \cite{chen}.

\begin{theorem}[\cite{BB5,BBKT}]\label{thm32}
There exists $C\in(1,\infty)$ depending only on $F$ so that the following hold. 
\begin{enumerate} 
\item[\rm (a)] For any $x\in F$, $r>0$ and non-negative $h\in \mcF^{(F)}$ that is $\mcE^{(F)}$-harmonic in $B_F(x,r)$,  
\[
h(y)\leq C\, h(z) \quad\hbox{for every }   y,z\in B_F(x,r/2).
\]

\item[\rm (b)] For any $m\geq 0$, $x\in F_m$, $r>0$ and non-negative $h\in W^{1,2}(F_m)$ that  $\mcE^{(F_m)}$-harmonic in  $B_{F_m}(x,r)$, 
\[
h(y)\leq C\,  h(z)  \quad\hbox{for every }  y,z\in B_{F_m}(x,r/2).
\]
\end{enumerate} 
\end{theorem}

\medskip

\subsection{Effective Resistances}\label{sec33}
There are two main ingredients in the construction of the Brownian motion on $F$: uniform elliptic Harnack inequality and resistance estimates. We quickly review resistance estimates in this part. 

\begin{lemma}[\cite{BB3,Mc}]\label{lemma36}
There exists $C \in [1,\infty)$ depending on $F$ such that $C^{-1}<R_n <C$ for each $n\geq 0$, where $R_n$ is the resistance of $F_n$ defined by \eqref{e:1.4}.  
\end{lemma}

The next result gives the lower bound estimates of resistances between a cell and the complement of its neighborhood. 
  In the sequel,  for   $Q\in \mcQ_n(F)$ with $n\geq 1$,  we set  
  $$
  \mathcal{S}_Q:=\{Q'\in Q_n(F):Q'\cap Q\neq \emptyset\}.
  $$ 
 
 \medskip

\begin{lemma}\label{lemma37}
There is a constant C depending only on F such that the following hold.
 \begin{enumerate} 
\item[\rm (a)] For each $n\geq 1$ and $Q\in \mcQ_n(F)$, there is $w_Q\in C(F)\cap \mcF^{(F)}$ so that $0\leq w_Q\leq 1$, $w_Q|_{F_Q}=1$, $w_Q|_{F\setminus F_{\mathcal{S}_Q}}=0$, and $\mcE^{(F)}(w_Q)\leq C  L_F^{(d_w-d_f)n}$.
		
\smallskip
		
\item[\rm (b)] For each $n\geq 1$, $m\geq 0$ and $Q\in \mcQ_n(F_m)$, there is $w^{(m)}_Q\in C(F_m)\cap W^{1,2}(F_m)$ so that $0\leq w_{m,Q}\leq 1$, $w^{(m)}_Q|_{F_{m,Q}}=1$, $w^{(m)}_Q|_{F_{Q'}}=0$ if $Q'\in \mcQ_n(F_m)$ and $Q'\cap Q=\emptyset$, and $\mcE^{(F_m)}(w^{(m)}_Q)\leq C\,\varphi_m(L_F^{-n})$, where  
\begin{equation}\label{eqn35}
			\varphi_m(r):=
			\begin{cases}
				r^{d_f-d_w}   &\text{ if }r\geq L_F^{-m},\\
				L_F^{(d_w-d_f+d-2)m}\,r^{d-2}  &\text{ if }0<r<L_F^{-m}.
			\end{cases}
		\end{equation}
	\end{enumerate}
\end{lemma}

\begin{proof} 
(a)   Note that there are at most $N$ different types of $ F_{\mathcal \mathcal{S}_Q} =\bigcup_{Q'\in \mathcal{S}_Q} F_{Q'}$  in the following sense. There is an integer $N\geq 1$ and a finite collection $\{Q^{(i)}:1\leq i\leq N\}\subset \bigcup_{k=1}^\infty \mcQ_k(F)$ such that for any $\wt Q\in \bigcup_{k=1}^\infty \mcQ_k(F)$,   
there are some $1\leq j\leq N$ and a  similarity map $\Psi:\R^d\to \R^d$ so that $\Psi  (F_{\wt Q})=F_{Q^{(j)}}$
and $\Psi \big(F_{\mathcal S_{\wt Q}}  \big)= F_{\mathcal S_{Q^{(j)}}}$.
 For each $1\leq i\leq N$, we fix a function $w_i \in C(F)\cap \mcF^{(F)}$ such that $w_i=1$ on $F_{Q^{(i)}}$ and  some neighborhood of the support of $w_i$ is contained in $F_{\mathcal S_{Q^{(i)}}}$. 
For a general $Q\in \mcQ_n(F)$ with $n\geq 1$, let 
 $1\leq j\leq N$ and some similarity map $\Psi:\R^d\to \R^d$ be such  that  $\Psi (F_Q)=F_{Q^{(j)}}$ and 
   $\Psi (F_{\mathcal S_Q}  )=F_{{\mathcal S}_{Q^{(j)}}}$.
Define $w_Q$  by 
\[
w_Q(x)=\begin{cases}
	w_j \circ \Psi (x) \quad&\text{ if } x\in F_{\mathcal S_Q} ,\\ 
    0 \quad&\text{ if }x\in F\setminus F_{\mathcal S_Q} . 
\end{cases}
\]
By the self-similar property of $(\mcE^{(F)},\mcF^{(F)})$, we see that $w_Q\in C(F)\cap \mcF^{(F)}$ and the desired estimate holds with $C$ depending only $\{w_i; 1\leq i\leq N\}$ and $F$. 

(b)  The case that $m>n$ is an immediate consequence of \cite[Theorem 5.8]{Mc} and Lemma \ref{lemma36}. The case that $m\leq n$ is trivial. 
\end{proof}

Next, we  establish the following estimate  using  Lemmas \ref{lemma36} and \ref{lemma37},
which was assumed as condition (A7) in \cite[p.582]{HK}.

\begin{lemma}\label{lemma38}
	$d_I>d_f-d_w$. 	
\end{lemma}

\begin{proof} 
Let $m\geq 2$. For each $1\leq k\leq m$, let  
\[
A_k := [ 1-L_F^{-k+1}+L_F^{-k},  1-L_F^{-k+1}+2L_F^{-k} ] \times [0, 1]^{d-1}, 
\] 
and define $h_{m,k}\in C(F_m)\cap W^{1,2}(F_m)$ by 
\[
h_{m,k} (x)=\max \left\{w^{(m)}_Q(x):Q\in \mathcal{Q}_k(F_m\cap A_k)\right\},\quad  x\in F_m,
\]
where $w^{(m)}_Q$ is the function in Lemma \ref{lemma37}(b). Note  that $h_{m,k}=1$ on $F_m\cap A_k$, $0\leq h_{m,k} \leq 1$ and
$$
\#\mathcal{Q}_k(F_m\cap A_k)=m_I^{k-1}\cdot\mcQ_1(F\cap A_1).
$$
 We have by Lemma \ref{lemma37}(b) that  that $\mcE^{(F_m)}(h_{m,k})\leq C_1\cdot m_I^k\cdot L_F^{(d_w-d_f)k}$ for some constant $C_1$ depending only on $F$.

Next, for $1\leq k\leq m$,  define  $g_{m,k} $ by     
\[
g_{m,k}(x)=\begin{cases}
h_{m,k}(x)    &\text{ if }x\in   F_m \cap\{x\in \R^d: x\leq   1-L_F^{-k+1}+L_F^{-k}\},\\
1     &\text{ if }x\in  F_m \cap \{x\in \R^d: x >  1-L_F^{-k+1}+L_F^{-k} \} ,
\end{cases}
\]
which is in $C(F_m)\cap W^{1,2}(F_m)$. 
By the strong local property of $(\mcE^{(F_m)},W^{1,2}(F_m))$, we have 
\[
\mcE^{(F_m)}(g_{m,k})\leq\mcE^{(F_m)}(h_{m,k})\leq C_1  m_I^k\cdot L_F^{(d_w-d_f)k}=C_1  L_F^{(d_w-d_f+d_I)k}.
\]
	
Finally, we let $f_m=\sum_{k=1}^m\frac{1}{k}g_{m,k} \in C(F_m)\cap W^{1,2}(F_m)$. Note that  $f_m=0$ on $\partial_{1,0}F$,
 $f_m=\sum_{k=1}^m\frac{1}{k}$ on  $\partial_{1,1}F_m$, and,  by the strong  local property of $\big(\mcE^{(F_m)},W^{1,2}(F_m)\big)$,  
  \begin{equation}\label{e:3.5}
\mcE^{(F_m)}(f_m) =\sum_{k=1}^m\frac{1}{k^2}\mcE^{(F_m)}(g_{m,k})\leq C_1 \sum_{k=1}^m \frac{1}{k^2}L_F^{(d_w-d_f+d_I)k}.
\end{equation}
 Finally,   by Lemma \ref{lemma36} and   the definition of $R_m$, we have 
$$
C^{-1}\leq R_m\leq \frac{\mcE^{(F_m)}(f_m)}{\left(\sum_{k=1}^m 1/k \right)^2} \quad \hbox{for every   } m \geq 1,
$$ 
where $C\geq 1$ is a constant independent of $n$. 
It follows that  $\lim_{m\to \infty} \mcE^{(F_m)}(f_m) =\infty$ and hence by  \eqref{e:3.5} we have $d_w-d_f+d_I>0$. 
\end{proof} 

   Recall that for an  open subset $O\subset F$,  its  $\mcE^{(F)}$-capacity of $O$  is defined 
as $\operatorname{Cap}^{(F)}(O)=\inf\{\mcE^{(F)}_1(f):f\in \mcF,\ f| \geq 1 \hbox{ $\mu$-a.e. on } O\}$, where $\mcE^{(F)}_1(f):=\mcE^{(F)}(f)+\|f\|_{L^2(F;\mu)}^2$. For a general Borel subset $A\subset F$,   $\operatorname{Cap}^{(F)}(A):=\inf\{\operatorname{Cap}^{(F)}(O):O\text{ is an open subset of }F,\ A\subset O\}$.

\begin{lemma}\label{lemmaHK}
 There is a constant $c>0$ so that 
\begin{equation}\label{e:3.7}
\nu (K) \leq c \operatorname{Cap}^{(F)}(K) \quad \hbox{for every compact subset } K\subset \partial_o F.
\end{equation}
Consequently, the normalized $d_I$-dimensional Hausdorff  measure $\nu $ charges no subset of $\partial_o F$ having zero $\mcE^{(F)}$-capacity
and the boundary $ \partial_oF $ has positive $\mcE^{(F)}$-capacity and 
 \end{lemma}

\begin{proof}
Property \eqref{e:3.7} follows directly from  Lemma \ref{lemma38} and \cite[Propositon 4.2(b)]{HK}.
This shows that $F$ is of positive $\mcE^{(F)}$-capacity as $\nu (\partial_o F)=1 $, and 
$\nu$ does not charge on sets of zero $\mcE^{(F)}$-capacity and hence is a smooth measure
of $(\sE^{(F)}, \sF^{(F)})$ in the sense of \cite{CF, FOT}.   
 \end{proof} 
 
\begin{remark}  \rm
The second part of Lemma  \ref{lemmaHK}  was assumed as condition (A8)  in    \cite[p.583]{HK}.
{\it It is shown in \cite[Theore, 4.2]{HK} that \eqref{e:3.7} holds under the assumption of $d_I>d_f-d_w$.}
\end{remark}

\subsection{Mosco convergence}\label{sec32} 
In this subsection, we will show that $\alpha_m\mcE^{(F_m)}$ is Mosco convergent  to $\mcE^{(F)}$.
 
\begin{theorem} \label{T:3.11}
 There exists a sequence $\{\alpha_k; k\geq 0\}$ with  
$$ 
0<\inf_{k\geq 0}\alpha_k\leq\sup_{k\geq 0}\alpha_k<\infty,
$$
so that for  any $x_k\in F_k$, $k\geq 1$,  with  $\lim_{k\to \infty} x_k = x\in F$,     
	 $  (X^{(F_k)}_{\alpha_k t})_{t\geq 0}$ under $\mathbb{P}^{(F_k)}_{x_k}$ 
	 converges weakly in $C ([0,\infty); \R^d)$ to   $ (X^{(F)}_t)_{t\geq 0}$ under $ \mathbb{P}^{(F)}_x$.
  \end{theorem}

\begin{proof} This result is essentially stated in \cite[Remark 5.4]{BBKT}. For reader's convenience, we spell out a detailed proof here. It is established in \cite{BB1, BB5} that $\{Y^{(n)}:=( W^{(F_{n})}_{\bar \rho_F^n t})_{t\geq0}; n\geq 1\}$ is tight both in probability law  (in the sense  of \cite[Theorem 5.1]{BB1}) and in resolvents (in the sense of \cite[Propositions 6.1 and 6.2]{BB1}, where $\bar\rho_F=L_F^{d_w-2}$.
  As mentioned previously, these results from \cite{BB1} hold for $d\geq 3$ as well due to the elliptic Harnack inequality established in \cite{BB5}. 
 Denote by $\tau$ the first hitting time of the faces of $\partial F_0$ not containing the origin ${\bf 0}:=(0, \ldots, 0)$. 
  Let $a_n := \bE^{Y^{(n)}}_{\bf 0} \tau$.  
    Then in view of tightness of the 0-resolvents from \cite[Proposition 6.1]{BB1} with $f=1$ and the non-degeneracy of  any sub-sequential limit process, 
   $\{a_n; n\geq 1\}$ are bounded between two positive constants. 
  
 For every subsequence, there is a sub-subsequence $\{Y^{(n_{k_j})}\}$ that converges to a limit process $Y$ in the above two senses. 
 By the  convergence of the 0-resolvents (cf.  \cite[Proposition 6.1]{BB1}),   we have $\lim_{j\to \infty} a_{n_{k_j}} =  a:=\bE^Y_{\bf 0} \tau$. 
 It follows that $\{Y^{(n_{k_j})}_{t/a_{n_{k_j}}}; t\geq 0\}$ converges weakly to $ Z:=\{Y_{t/a}; t\geq 0\}$.  Note that  the 
mean time of $Z$ starting from ${\bf 0}$ to hit   the faces of $\partial F_0$ not containing  ${\bf 0} $ is 1 and that 
the Dirichlet form of $Z$ on $L^2(F; \mu)$ is strongly local, regular, irreducible and locally symmetric.
Hence by the uniqueness result from \cite{BBKT}, the process $Z$ has the same distribution as $X^{(F)}$. 
Since this holds for any subsequence, we conclude that 
for  any $x_k\in F_k$, $k\geq 1$,   with  $\lim_{k\to \infty} x_k = x\in F$,     
$\{ (Y^{(n)}_{\lambda_{n}t/a_n})_{t\geq0}; n\geq 1\}$ under $\mathbb{P}_{x_n}^{Y^{(n)}}$  converges weakly in $C ([0,\infty); \R^d)$ to $(X^{(F)}_t)_{t\geq 0}$ under $\mathbb{P}^{(F)}_x$. This proves the theorem with $\alpha_k=1/a_k$.
\end{proof}

The main goal of the paper is to show that $\{\alpha_k; k\geq 0\}$ in Theorem \ref{T:3.11} has a limit and thus one could take $\alpha_j \equiv \alpha $ for all $j\geq 1$ there. A consequence of this result is that 
\begin{equation}\label{e:3.11a}
\lim_{k\to \infty} \bE^{(F_k)}_{\bf 0} \tau =\alpha^{-1}  = \alpha^{-1} \bE^{(F )}_{\bf 0} \tau  . 
\end{equation} 
It will further imply the convergence of the resistance $R_n$ defined by \eqref{e:1.4}; that is,  $\lim_{k\to \infty}R_k$ exists as a finite positive number.
This  gives an affirmative answer to a long standing open problem raised by  Barlow and Bass \cite[Problem 1]{BB3}.  
This property will play a key role in our study of  quenched invariance principle on generalized unbounded 
{S}ierpi\'{n}ski carpets  in i.i.d. uniformly elliptic random environments.
   
\medskip

 For $\lambda >0$, denote by
 $U^{(F)}_\lambda$  the $\lambda$-resolvent operator for $\big((X^{(F)}_t)_{t\geq 0},\mathbb{P}^{(F)}_x\big)$; for $m\geq 0$, 
denote by $U^{(F_m)}_\lambda$ to denote the $\lambda$-Resolvent operator for $\big((X^{(F_m)}_{\alpha_mt})_{t\geq 0},\mathbb{P}^{(F_m)}_x\big)$. The following lemma is proven in \cite[Section 3.1]{BBKT} using the uniform elliptic Harnack inequality in Theorem \ref{thm32} and exit  time estimates from \cite[Proposition 5.5]{BB5}.

\begin{lemma}\label{lemma34}
For each $\lambda>0$, the class of functions 
\[
\{U_\lambda^{(F_m)}f_m:\ m\geq 0,\ f_m\in L^\infty(F_m; \mu_m),\ \|f_m\|_\infty\leq 1\}
\]
is uniformly bounded and equicontinuous.
\end{lemma}

\begin{theorem}\label{T:3.12}
	$\alpha_m\mcE^{(F_m)}$ on $L^2(F_m; \mu_m)$ is Mosco convergent to $\mcE^{(F)}$ on $L^2(F;\mu)$ in the sense of Definition \ref{def28} as $m\to \infty$. 
\end{theorem}

\begin{proof}
We fix $\lambda>0$ and show that (d) of Theorem \ref{thm211} holds true. For each $g \in C(F_0)$ and $m\geq 0$, $U_\lambda^{(F_m)}(g|_{F_m})\in C(F_m)$ and $U_\lambda (g|_F)\in C(F)$ by
 the strong Feller property of  $X^{(F_m)}$ and $X^{(F)}$.
Moreover,  $U_\lambda^{(F_m)}(g|_{F_m})\rightarrowtail U_\lambda^{(F)}(g|_{F})$ by Theorem \ref{T:3.11}. Since $F_0$ is compact and $\mu_m$ and $\mu$ are probability measures on $F_m$ and $F$, respectively, $U_\lambda^{(F_m)}(g|_{F_m})\to U_\lambda^{(F)}(g|_F)$ strongly in $L^2$ by Lemma \ref{lemma29} (b). Noting that  $\|U_\lambda^{(F)}\|_{L^2\to L^2}\leq \lambda^{-1}$ and $\|U_\lambda^{(F_m)}\|_{L^2\to L^2}\leq \lambda^{-1}$ for each $m\geq 0$, it follows from  Proposition \ref{prop211}  that $U_\lambda^{(F_m)}$ converges strongly to $U^{(F)}$. 
\end{proof}

\medskip

\subsection{Energy measures}\label{sec34} We end this section with a quick review of energy measures and Poincar\'e inequalities. \\

\noindent(\textbf{\emph{Energy measure}}). Let $(\mcE,\mcF)$ be a strongly local regular Dirichlet form on $L^2(\mathcal{X};\mu)$. For $f\in L^\infty(\mathcal{X}; \mu)\cap \mcF$, we define the energy measure $\mu_{\< f\>}$ as the unique Radon measure on $K$ such that 
\begin{equation}\label{e:3.11}
\int_K g(x)\mu_{\< f\>}(dx)=\mcE(f,fg)-\frac12\mcE(f^2,g),\quad\  g\in C_c(\mathcal{X})\cap \mcF.
\end{equation} 
For general $f\in \mcF$, we define $\mu_{\< f\>}=\lim\limits_{n\to\infty}\mu_{\< f_n\>}$ with $f_n=(f\wedge n)\vee (-n)$, whose limit is known to exist (see \cite{CF, FOT}). Note that our choice is different from \cite{FOT} by a constant multiplier $\frac{1}{2}$
so that $\mu_{\< f\>}(\mathcal{X})=\mcE(f)$ for $f\in \mcF$.\\

We use the notation $(\mcE^{(*)},\mcF^{(*)})$ (sometimes we use $W^{1,2}(*)$ instead of $\mcF^{(*)}$), where $*$ usually represents the underlying space, to denote various Dirichlet forms.
The corresponding energy measure of $f\in \mcF^{(*)}$ is denoted as 
 $\mu^{(*)}_{\<f\>}$. We will also use the notation $(\bar{\mcE}^{(*)},\mcF^{(*)})$ (or $(\bar{\mcE}^{(*)},W^{1,2}(*)$) from time to time, and use $\bar{\mu}^{(*)}_{\<f\>}$ for the associated energy measure of $f\in \mcF^{(*)}$ (or $f\in W^{1,2}(*)$).\\

\begin{lemma}\label{lemma39}
	There are constants $C>0$ and $c\in (0,1)$ such that the following hold.
	
	\begin{enumerate} 
		\item[\rm (a)]  For each $f\in \mcF^{(F)}$, $x\in F$ and $r\in (0, 1]$,  
		\[
		r^{d_f-d_w} \fint_{B_F(x,c r)}\big(f(y)-[f]_{\mu|_{B_F(x,cr)}}\big)^2\mu(dy)\leq C  \mu^{(F)}_{\< f \>}\big(B_F(x,r)\big). 
		\]
		
		\item[\rm (b)] For each $m\geq 0$, $f\in W^{1,2}(F_m)$, $x\in F_m$ and $r\in (0, 1]$,  
		\[
		\varphi_m(r) \fint_{B_{F_m}(x,c r)}\big(f(y)-[f]_{\mu_m|_{B_{F_m}(x,cr)}}\big)^2\mu_m(dy)\leq C \mu^{(F_m)}_{\<f\>}\big(B_{F_m}(x,r)\big). 
		\]
	\end{enumerate} 
\end{lemma}

\begin{proof} (a) follows from  the heat kernel estimates \eqref{e:1.2} for  $X^{(F)}$ 
	and its stable characterization (see, e.g., \cite{AB, BBK06,  GHL}). (b) follows from \cite[Theorem 7.3]{BB5}. 
\end{proof}

\section{Harmonic functions with assigned mean boundary values}\label{sec4} 

In this section, we prove a key result, Proposition \ref{prop41}, of the paper. The other two key results, trace theorems and a theorem about decreasing rate of energy measures near the boundary of cells, are proved in appendixes.

\medskip

We start with a result that will be needed in the sequel. It asserts that  every boundary point in $\partial_o F$ is regular for $\partial_oF$ with respect to the Brownian motion
 $X^{(F)}$ on $F$, that is, $\bP^{(F)}_x (\dot \sigma_{\partial_o F}=0)=1$ for every $x\in \partial_o F$, where the notation $\dot \sigma_{A}$ is defined in \eqref{e:3.4}.

 \begin{proposition}\label{P:3.5}
Every point of $\partial_o F$ is a regular point for $ \partial_o F_0$ with respect to $X^{(F)}$.
\end{proposition}

\begin{proof}  
For each $n\geq 2$, define $\mcB_n(F)=\{Q\in \mcQ_n(F):  Q\cap \partial_oF\neq \emptyset\}$ and $A_n=F_{\mcQ_n(F)\setminus \mcB_n(F)}\cap F_{\mcB_n(F)}$. 
Note that $\sigma_{A_n}=\sigma_{F_{\mcQ_n(F)\setminus \mcB_n(F)}}$ $\bP_x^{(F)}$-a.s. for each $x\in \partial_oF$ and $n\geq 2$. It suffices to show that
there is a constant $C_1\in (0,1)$ so that 
\begin{equation}\label{eqnprop310}
\bP^{(F)}_z(\sigma_{\partial_oF}<\sigma_{A_{n-1}})\geq C_1 \quad \hbox{for }   n\geq 3 \hbox{ and } z\in A_n. 
\end{equation}
Indeed, assuming \eqref{eqnprop310},    we have by the strong Markov property of $X^{(F)}$ that for each $x\in \partial_oF$ and $n\geq 3$,
\[
\bP^{(F)}_x(\dot{\sigma}_{\partial_oF}<\sigma_{\sigma_{A_{n-1}}})\geq \bP^{(F)}_x(\sigma_{\partial_oF}\circ \theta_{\sigma_{A_n}}<\sigma_{A_{n-1}}\circ \theta_{\sigma_{A_{n}}})\geq C_1,
\] 
where $\theta$ is the time shift operator  for $X^{(F)}$: $X^{(F)}_t\circ \theta_s=X^{(F)}_{t+s}$   for every $ t,s\geq 0$.
 The above inequality holds due to  the fact that $\bP_x^{(F)}$-a.s.
 $$
 \sigma_{A_{n-1}}\circ \theta_{\sigma_{A_{n}}}+\sigma_{A_n}=\sigma_{A_{n-1}} 
 \quad \hbox{and} \quad 
 \dot\sigma_{\partial_oF}\leq\sigma_{\partial_oF}\circ \theta_{\sigma_{A_n}}+\sigma_{A_n} .
 $$
 Since   $\lim_{n\to \infty} \sigma_{A_n}=\sigma_{F\setminus \partial_oF}=0$  $\bP_x^{(F)}$-a.s., we have
$
\bP^{(F)}_x(\dot{\sigma}_{\partial_oF}=0)\geq C_1>0$.
Hence by Blumenthal's zero–one law, $\bP^{(F)}_x(\dot{\sigma}_{\partial_oF}=0)=1$, proving that 
 $x$ is a regular point for $\partial_oF$. 
 
   \smallskip 
	
We now proceed to show \eqref{eqnprop310}. 
Let 
$$
E:=\bigcup \left\{\partial_{i,s}F: i\in \{1,2,\cdots,d\},s\in\{0,1\}\text{ and }(i,s)\neq (1,0) \right\},
$$
 and fix a closed set $D\subset \partial_{1,0}F$ such that $\nu(D)>0$ and $\rho(D,E)>0$. 
In view of \eqref{e:3.7},  both $D$ and $E$ have positive capacity. Define 
 \[
h(z):=\bP^{(F)}_z(\sigma_D<\sigma_{E}) \quad\hbox{for }  z\in F. 
\]
Observe that  $h\in \mcF^{(F)}$, $h=0$ on $E$, $h=1$ on $D$ and $h$ is $\mcE^{(F)}$-harmonic in $F\setminus (D\cup E)$.  
 By EHI (Theorem \ref{thm32}) and the connectivity of $F\setminus \partial_oF$,   $h(z)>0$ for each $z\in F\setminus \partial_oF$. 
 Fix a   $z_o\in F\setminus \partial_oF$  so  that $\rho(z_o,\partial_{1,0}F_0)>1/2$.

Let   $n\geq 3$ and $z\in A_n$. Take  $Q\in \mcB_n(F)$  so that there are 
 $i\in \{1,2,\cdots,d\},s\in\{0,1\}$ such that $z \in \partial_{i,s}F_Q$ and $\partial_{i,1-s}F_Q\subset \partial_oF$. Define
\[
h_z(y):=\bP^{(F)}_y(\sigma_{\partial_{i,1-s}F_Q}<\sigma_{A_{n-1}}) \quad \hbox{for }  y\in F. 
\]
Denote by $\Psi:\R^n\to\R^n$ the similarity map such that $\Phi (F)=F_Q$ and $\Psi(\partial_{1,0}F)=\partial_{i,1-s}F_Q$ (so $\Psi(D)\subset \partial_{i,1-s}F_Q\subset \partial_oF$, and $\Psi(E)$ disconnect $F_Q$ from other $n$-cells), then 
\[
h_z(y)\geq \bar h_z(y):=\bP^{(F)}_y(\sigma_{\Psi(D)}<\sigma_{\Psi(E)})=h\big(\Psi^{-1}(y)\big)
\quad \hbox{for }  y\in F_Q\setminus \partial_oF_Q,
\]
where we use self-similarity of $X^{(F)}$ under $\Psi$ in the last equality. In particular, we have 
 \[
h_z\big(\Psi(z_o)\big)\geq \bar h_z\big(\Psi(z_o)\big)=h(z_o) >0. 
\]
Noticing that $\Psi(\{y\in F:\rho(y,\partial_{0,1}F_0)\})\geq  {1}/{2})$ is connected (by a same proof of Lemma \ref{lemmaA1}),  by using EHI (we can find a chain of balls of radius $\frac14L_F^{-n}$ connecting $\Psi(z_o)$ and $z$, and the number of balls has a uniform upper bound independent of $z\in A_n$, we conclude that $h_z(z)\geq C_1$ for some $C_1>0$ depending only on $F$, which is \eqref{eqnprop310}. 
\end{proof}

\medskip

For the development of a forthcoming paper about some quenched invariance principle, we state the result in slightly more general setting (with the assumptions to be verified there). We assume in this subsection that $(\bar{\mcE}^{(F_m)},W^{(1,2)}(F_m)),\ m\geq 0$ is a sequence of strongly local regular Dirichlet forms on $L^2(F_m; \mu_m)$ such that 
\begin{equation}\label{e:4.1}
C_0^{-1}\mcE^{(F_m)}\leq \bar{\mcE}^{(F_m)}\leq C_0\mcE^{(F_m)} \quad \hbox{for every  }   m\geq 0, 
\end{equation} 
for some $C_0\in [1,\infty)$ and $\bar{\mcE}^{(F_m)}$ is Mosco convergent to $\mcE^{(F)}$.

\medskip

\begin{remark} \rm 
The above assumption can be imposed for a subsequence $\{m_k; k\geq 1\}$ instead of for
the whole sequence $\{m; m\geq 1\}$. The following assumptions (A1), (A2) and (A3) will then be assumed for the corresponding subsequence only. In this case, all the results in this section hold for this subsequence $\{m_k,k\geq 1\}$ with the same proof.
\end{remark}

\begin{enumerate} 
\item[(A1)]  For each $g_m\in C(F_m)\cap W^{1,2}(F_m)$, there is $h_m\in C(F_m)\cap W^{1,2}(F_m)$ such that 
$h_m|_{\partial_o F_m}=g_m|_{\partial_o F_m}$ and $h_m$ is  $\bar \sE^{(F_m)}$-harmonic in $F_m\setminus \partial_o F_m$.

Suppose that $h_m\in C(F_m)\cap W^{1,2}(F_m)$ is $\bar \sE^{(F_m)}$-harmonic in $F_m\setminus \partial_o F_m$ for each $m\geq 0$, and  $h_m|_{\partial_o F_m},m\geq 0$ 
are  uniformly bounded and equicontinuous. Then, for each subsequence $\{m_k; k\geq 1\}$, 
there is a sub-subsequence $\{m_{k(l)}; l\geq 1\}$ so that  $h_{m_{k(l)}}\rightarrowtail h$ for some 
 $h\in C(F)$.  
 
 \smallskip 

\item[(A2)] Let $n\geq 1$, $Q\in \mcQ_n(F)$, $f_m\in L^2(F_m; \mu_m)$,  $m\geq 1$,  and $f\in L^2(F;\mu)$. If $f_m\to f$ strongly in $L^2$, then 
$\liminf\limits_{m\to\infty}\bar{\mu}^{(F_m)}_{\<f_m\>}(F_{m,Q}) \geq \mu^{(F)}_{\< f \>}(F_Q)$.

 \smallskip 

\item[(A3)] $ \big\{\bar{U}_\lambda^{(F_m)}f:\ m\geq 0,\ f\in L^\infty(F_m; \mu_m),\ \|f\|_\infty\leq 1 \big\}$ is equicontinuous, where $\bar{U}_\lambda^{(F_m)}$ is the resolvent operator associated with $(\bar{\mcE}^{(F_m)},W^{1,2}(F_m))$ on $L^2(F_m; \mu_m)$.\\
\end{enumerate}

Conditions (A1)-(A3) hold for $\bar\sE^{(F_m)}=\alpha_m\sE^{(F_m)}$, where $\alpha_m$ is the constant in Theorem \ref{T:3.11} as well as in Theorem \ref{T:3.12}. Indeed in this case, (A3) is just Lemma \ref{lemma34}, while (A1) is proved in Lemma \ref{lemma41} by using Proposition \ref{P:3.5}. It is shown at the beginning of Section \ref{sec5} that 
 condition (A2) holds for a suitable sub-sequence, and then it holds for the whole sequence after we have established that the limit of 
$\{a_m; m\geq 1\}$ exists in \eqref{e:5.12} in view of Theorem \ref{T:3.12}.

When $F$ is a GSC equipped with the uniformly elliptic i.i.d random conductance as considered in \cite{CC}, conditions (A1) and (A3) can be shown to hold using the stability theorem of elliptic Harnack inequality \cite{BCM, BM}. Condition (A2) holds in this case as well and  its proof will be given in the forthcoming paper \cite{CC}.

\medskip

To state the main result of this section, we need  some notation.

 \medskip
 
\noindent(\textbf{\emph{Sub-faces}}).  Let $m,n\geq 0$.
\begin{enumerate} 
\item[\rm (a)] Denote by $\eth_nF$  the collection of level $n$ sub-faces of $\partial_o F$:
\[
\eth_n F= \left\{\partial_{i,s}F_Q: Q\in \mcQ_n(F),\ i\in \{1,2,\cdots, d\},\ s\in \{0, 1 \}\hbox{ such that }\partial_{i,s}F_Q\subset\partial_o F\right\}.
\]

\item[\rm(b)] Denote by $\eth_nF_m$  the collection of level $n$ sub-faces of $\partial_o F_m$:
\[
\eth_n F_m= \left\{\partial_{i,s}F_{m,Q}: Q\in \mcQ_n(F_m),\ i\in \{1,2,\cdots, d\},\ s\in \{0, 1\}\hbox{ such that }\partial_{i,s}F_{m,Q}\subset \partial_o F_m
\right\}.
\]
\end{enumerate}

\medskip
 Recall that $\nu$ is the normalized $d_I$-dimensional Hausdorff measure on  $\bigcup_{n=0}^\infty\bigcup_{Q\in \mcQ_n(F)}\partial_o F_{Q}$ with $\nu(\partial_oF)=1$
and $\nu_m$ is the normalized $(d-1)$-dimensional Hausdorff measure on  $\bigcup_{n=0}^\infty\bigcup_{Q\in \mcQ_n(F_m)}\partial_o F_{m,Q}$ with $\nu_m(\partial_o F_m)=1$.

\medskip

\noindent(\textbf{\emph{Discrete energy on sub-face graphs}}). Let $m\geq 0,k\geq 1$.
\begin{enumerate} 
\item[\rm (a)]   For $A,A'\in \eth_kF$ such that $A\neq A'$, we say $A\sim A'$ if 
\begin{equation}\label{e:4.2a}
A\cap A'\neq \emptyset \ \hbox{ or } \ 
A,A'\subset B \hbox{ for some } B \in \eth_{k-1}F. 
\end{equation} 
For each $f\in L^1(\partial_o  F; \nu)$,  define 
\[
I_k[f]=\sum_{A,A'\in \eth_k F\atop A\sim A'}([f]_{\nu|_A}-[f]_{\nu|_{A'}})^2.
\]

\item[(b)] For $A,A'\in \eth_kF_m$ such that $A\neq A'$, we say $A\sim A'$ if $A\cap A'\neq \emptyset$ or $A,A'\subset B$ for some $B\in \eth_{k-1}F_m$. For each $f\in L^1(\partial_o  F_m; \nu_m)$, define 
\[
I^{(m)}_k[f]=\sum_{A,A'\in \eth_k F_m\atop A\sim A'}([f]_{\nu_m|_A}-[f]_{\nu_m|_{A'}})^2.
\]
\end{enumerate}

\begin{remark} \rm We need to consider two possibilities when defining the relation $\sim$ because $\partial_o  F$ may not be connected. We remark 
that the graphs $(\eth_1F,\sim)$ and $(\eth_1F_m,\sim)$ are always connected.
\end{remark} 
  
\noindent(\textbf{\emph{Besov type spaces on faces}}).  
(a)  For each $n\geq 1$ and $f\in L^1(\partial_o  F; \nu)$,   define  
\[
\Lambda_n[f]:=\sum_{k=n}^\infty L_F^{k(d_w-d_f)}I_k[f],
\] 
and   the space
\begin{equation} \label{e:4.3} 
\Lambda(\partial_o  F) :=\{f\in L^2(\partial_o  F; \nu): \Lambda_1 [f]<\infty\}.
\end{equation}

(b) For each $m\geq 0,n\geq 1$ and $f\in L^1(\partial_o  F_m; \nu_m)$,   define (where $\varphi_m$ is defined in \eqref{eqn35})  
\[
\Lambda^{(m)}_n[f]:=\sum_{k=n}^{\infty}\varphi_m( L_F^{-k})I^{(m)}_k[f],
\]
and  the space 
\begin{equation} \label{e:4.4a} 
{\Lambda^{(m)} }
(\partial_o  F_m) :=\{f\in L^2(\partial_o  F_m; \nu_m): \Lambda^{(m)}_1  [f]<\infty\}. 
 \end{equation}

\medskip

 \noindent(\textbf{\emph{Boundary shells}}). For $n\geq 0$, define $\mcB_n:=\{Q\in \mcQ_n:Q\cap \partial F_0\neq\emptyset\}$. For $n\geq 0$ and $A\subset \R^d$,    let $\mcB_n(A):=\{Q\in \mcB_n:\operatorname{int}(Q)\cap A\neq\emptyset\}$.

\medskip

Note that $\mathcal{B}_n(F_m)=\mathcal{B}_n(F)$ if $n\leq m$. See Figure \ref{fig3} for the pictures of $F_{\mcB_n(F)}=\bigcup_{Q\in \mcB_n(F)} F_Q$ of the standard {S}ierpi\'{n}ski  carpet $F$ for $n=1, 2, 3$.

\begin{figure}[htp]
\includegraphics[width=3.5cm]{standardSC.pdf}\qquad 
\includegraphics[width=3.5cm]{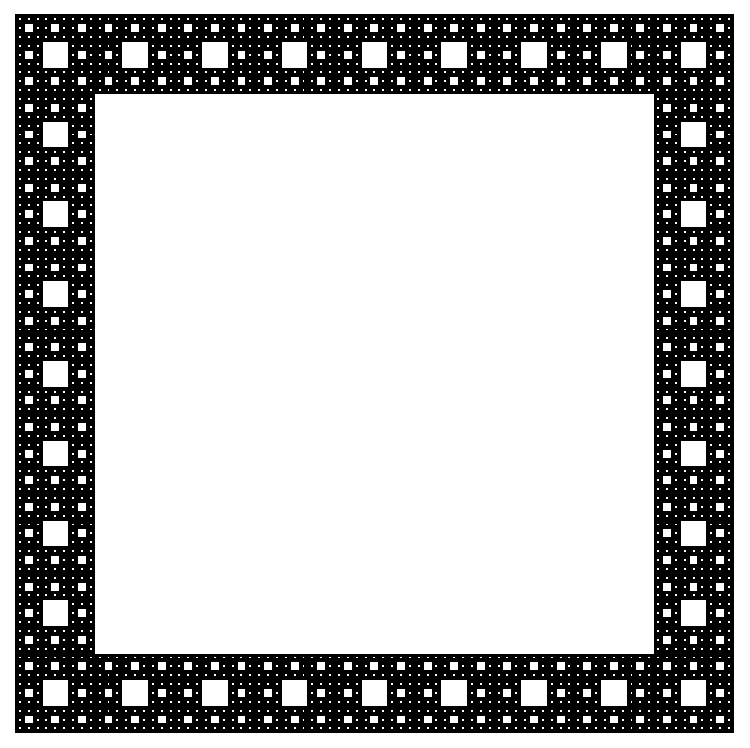}\qquad
\includegraphics[width=3.5cm]{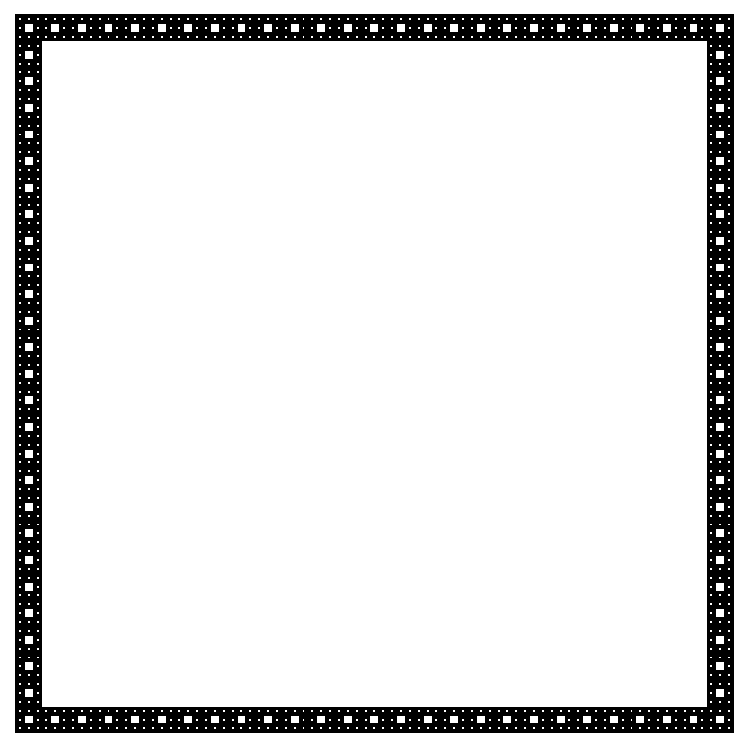}
\caption{$F_{\mcB_1(F)}$, $F_{\mcB_2(F)}$ and $F_{\mcB_3(F)}$ of the standard {S}ierpi\'{n}ski carpet $F$ in $\R^2$}\label{fig3}
\end{figure}

\medskip

The following is the main result of this section.

\begin{proposition}\label{prop41}
Let $n\geq 1$ and assume that (A1), (A2), (A3) hold. Then, there are positive finite constants $C$ depending only on the constant $C_0$ in \eqref{e:4.1} and $N\geq n$ (which depends on $n$ and $\bar{\mcE}^{(F_m)},m\geq 0$) such that 
\[
\left|\sqrt{\bar{\mcE}^{(F_{m})}(g_{m})}-\sqrt{\mcE^{(F)}(g)}\right|\leq C  \left(\sqrt{\mu^{(F)}_{\<g\>}(F_{\mcB_{n-1}(F)})}+\sqrt{\Lambda^{(m)}_n[g_{m}|_{\partial_o  F_{m}}]}\right)
\]
  for any $m\geq N$, any $g\in \mcF^{(F)}\cap C(F)$ that is $\mcE^{(F)}$-harmonic in $F\setminus \partial_o  F$, and any $g_{m}\in W^{1,2}(F_m)\cap C(F_m)$ that  is $\bar{\mcE}^{(F_{m})}$-harmonic in $F_{m}\setminus \partial_o  F_{m}$ having 
\[\fint_{A_m}  g_m d\nu_m = \fint_A g d\nu  
\quad \hbox{for every }   A\in \eth_nF\text{ and }A_m\in \eth_nF_m\text{ with  }A\subset A_m.
\]
\end{proposition}\vspace{0.15cm}

The proof of this proposition will be given in subsection \ref{sec43}, after a series of preparations. 
 Our proof also needs some trace theorems, which are given in Appendix \ref{secA}.

\subsection{Good approximation sequence of harmonic functions}\label{sec41} 

In this subsection, we show that 
under condition (A1) and (A3), every  $h\in \mcF^{(F)}\cap C(F)$ that is $\mcE^{(F)}$-harmonic in $F\setminus \partial_o  F$ can be approximated by some nice functions in $W^{1,2}(F_m)\cap C(F_m)$; see Lemma \ref{lemma43}.
 First we show that condition (A1) holds when $\bar{\mcE}^{(F_m)} =\alpha_m\mcE^{(F_m)}$ for every $ m\geq 0$. 

\begin{lemma}\label{lemma41}
(A1) holds for $(\alpha_m\mcE^{(F_m)},W^{1,2}(F_m))$ with  $m\geq 0$.
\end{lemma}   

\begin{proof} 
Since $F_m$ are Lipschiz domains,
 it is well known that  each $x\in \partial  F_m$ is regular for $\partial  F_m$ with respected to the normally  reflected Brownian motion on $F_m$. 
Thus  for every   $g \in C(F_m)\cap W^{1,2}(F_m)$, $h(x)= \bE_x [ g(X^{(F_m)}_{\sigma_{\partial_0 F_m}})]$ is a continuous function on $F_m$
with $h=g$ on $\partial_o F_m$. On the other hand, since $F_m$ is compact, it follows from \cite{CF, FOT} that $h\in W^{1,2}(F_m)$. 

\medskip

 Suppose that for each $m\geq 0$, $h_m\in C(F_m)\cap W^{1,2}(F_m)$ is  $  \sE^{(F_m)}$-harmonic in $F_m\setminus \partial_o F_m$   
 and  $h_m|_{\partial_o F_m}$ is uniformly bounded and equicontinuous. 
   For each subsequence $\{m_k; k\geq 1\}$,  by (EHI) (Theorem \ref{thm32}) and Lemma \ref{lemma28}(a),
   there is a sub-subsequence $\{m_{k(l)}; l\geq 1\}$ so that    $h_{m_{k(l)}}|_{F_{m_{k(l)}}\setminus\partial_o F_{m_{k(l)}}}\rightarrowtail h|_{F\setminus \partial_o F}$ and $h_{m_{k(l)}}|_{\partial_o F_{m_{k(l)}}}\rightarrowtail h|_{\partial_o F}$ as $l\to \infty$
  for some bounded function $h$ on $F$ that is  continuous on $F\setminus \partial_o F$ and on $\partial_o F$.  
  In view of Theorem \ref{T:3.11}, we can further assume that the constants $\{\alpha_{m_{k(l)}}; l\geq 1\}$ there converge  to a positive number $\alpha$.  
    If we can show that 
\begin{equation}\label{eqn42}
\lim_{l\to \infty} h_{m_{k(l)}}(x_{m_{k(l)}})= h(x) \quad\hbox{for every }   x_{m_{k(l)}}\in F_{m_{k(l)}} \hbox{ that converges to  }  x\in \partial_o F, 
\end{equation}
then $h_{m_k}\rightarrowtail h$ and so $h\in C(F)$ by Lemma \ref{lemma28}(b). This would establish (A1). 
 
To show \eqref{eqn42},   fix $x\in \partial_oF$ and $x_{m_{k(l)}}\in F_{m_{k(l)}},l\geq 1$ such that $x_{m_{k(l)}}$ converges to $x$.
 By  Lemma \ref{lemma28} (b), there are 
 $f_{m_{k(l)}}\in C(\partial F_0)$, $l\geq 1$ and $f\in C(\partial F_0)$ such that $h_{m_{k(l)}}|_{\partial_o F_{m_{k(l)}}}=f_{m_{k(l)}}|_{\partial_o F_{m_{k(l)}}}$, 
$l\geq 1$, $h|_{\partial_o F}=f|_{\partial_o F}$ and $\|f_{m_{k(l)}}-f\|_\infty\to 0$. So,
\begin{equation} \label{e:4.2} 
\begin{aligned}
&\lim\limits_{l\to\infty} \left|  h_{m_{k(l)}}({x_{m_{k(l)}}}) -   \bE_{x_{m_{k(l)}}}  \left[ f (X^{(F_{m_{k(l)}})}_{\sigma_{\partial_0 F_{m_{k(l)}}}} ) \right]  \right|
\\=&\lim\limits_{l\to\infty}\left| \bE_{x_{m_{k(l)}}} \left[   ( h_{m_{k(l)}} -f) (X^{(F_{m_{k(l)}})}_{\sigma_{\partial_0 F_{m_{k(l)}}}} ) \right]  \right|
\leq \lim\limits_{l\to\infty}\| f_{m_{k(l)}} -f\|_\infty=0.
\end{aligned}
 \end{equation} 
 Recall that  $\sigma_{\partial F_0} (\omega):=\inf\{t\geq 0: \omega (t) \in  \partial F_0\}$  and $\tau_{B(x,r)}:=\inf\{t\geq 0: \omega (t) \notin B(x,r)\}$  for $\omega \in  C([0, \infty); \R^d)$. 
 As $\{  \sigma_{\partial F_0}\geq \eps \}$ is a closed  set in $C([0, \infty); \R^d)$, we have by Theorem \ref{T:3.11} that 
$$
\limsup_{l\to \infty}  \bP_{x_{m_k(l)}}^{(F_{m_{k(l)}})}  (   \sigma_{\partial F_0} \geq \eps/\alpha_{m_{k(l)}}) 
\leq   \bP_x^{(F)}  (   \sigma_{\partial F_0} \geq \eps) =0. 
 $$
 On the other hand, for any $r>0$, $\{\tau_{B(x, r)} > \eps\}$ is an open set in $C([0, \infty); \R^d)$, we have by Theorem \ref{T:3.11}
 again  that 
$$
\liminf_{l\to \infty}  \bP_{{x_{m_{k(l)}} }}^{(F_{m_{k(l)}})}  ( \tau_{B(x, r)} > \eps /\alpha_{m_{k(l)}}) 
\geq   \bP_x^{(F)}  ( \tau_{B(x, r)} > \eps ).
 $$
It follows that 
\begin{eqnarray*}
&& \liminf_{l\to \infty}  \bP_{{x_{m_{k(l)}} }}^{(F_{m_{k(l)}})}  \left( X^{(F_{m_{k(l)}})}_{  \sigma_{\partial F_0}} \in B(x, r)  \right)  \nonumber \\
&\geq&  \liminf_{l\to \infty}  \bP_{{x_{m_{k(l)}} }}^{(F_{m_{k(l)}})}  \left(     \{  \sigma_{\partial F_0} < \eps/\alpha_{m_{k(l)}} \}  \cap 
  \{ \tau_{B(x, r)} > \eps /\alpha_{m_{k(l)}})  \}     \right)   \nonumber   \\
  &\geq & \liminf_{l\to \infty}  \left( \bP_{{x_{m_{k(l)}} }}^{(F_{m_{k(l)}})}  \left(   \tau_{B(x, r)} > \eps /\alpha_{m_{k(l)}})     \right) 
  - \bP_{{x_{m_{k(l)}} }}^{(F_{m_{k(l)}})}  \left(     \sigma_{\partial F_0} \geq  \eps/\alpha_{m_{k(l)}}   \right)  \right)  \nonumber \\
&\geq & \bP_x^{(F)}  ( \tau_{B(x, r)} > \eps ).    \label{e:4.4} 
\end{eqnarray*}
  Since $X^{(F)}$ has right continuous sample paths, $\lim_{\eps\downarrow 0} \bP_x^{(F)}(\tau_{B(x,r)}>\eps)
  = \bP_x^{(F)}  ( \tau_{B(x, r)} > 0 ) =1$ for every $r>0$ and so 
  $$
  \liminf_{l\to \infty}  \bP_{{x_{m_{k(l)}} }}^{(F_{m_{k(l)}})} \left( X^{(F_{m_{k(l)}})}_{  \sigma_{\partial F_0}} \in B(x, r) \right)=1.
  $$
   Since this holds for every $r>0$, we conclude   that 
  $X^{(F_{m_{k(l)}})}_{  \sigma_{\partial F_0}}$ converges to $x$ in law as $l\to \infty$.      
    This combined with \eqref{e:4.2} yields that 
 $$\lim_{l\to \infty} h_{m_{k(l)}}(x_{m_{k(l)}})
 =\lim_{l\to \infty} \bE_{x_{m_{k(l)}}}  \left[ f (X^{(F_{m_{k(l)}})}_{  \sigma_{\partial_0 F_{m_{k(l)}}}} ) \right] 
 =\lim_{l\to \infty} \bE_{x_{m_{k(l)}}}  \left[ f (X^{(F_{m_{k(l)}})}_{  \sigma_{\partial  F_0}} ) \right]=f(x)=h(x),\\
 $$
which is \eqref{eqn42}.
\end{proof}

By Proposition \ref{P:3.5},  each point $x\in \partial_o F$ is regular for $\partial_o F$ with respect to $X^{(F)}$. So,  by the Harnack inequality (Theorem \ref{thm32})
and its implication, 
for  $g\in C(F)\cap \mcF^{(F)}$, the function 
$$
h(x):=\bE_x \left[ g(X^{(F)}_{\sigma_{\partial_o F}} )\right], \quad  x\in F,
$$
is in $C(F)\cap \mcF^{(F)}$ and is $\sE^{(F)}$-harmonic in $F\setminus \partial_o F$
 with $h|_{\partial_o  F}=g|_{\partial_o  F}$.

\begin{lemma}\label{lemma43}
Assume that (A1) and (A3) hold, and $h\in C(F)\cap \mcF^{(F)}$ 
is $\mcE^{(F)}$-harmonic in $F\setminus \partial_o F$. 
Then there exist $h_m\in W^{1,2}(F_m)\cap C(F_m)$, $m\geq 0$, so that  
\begin{enumerate}
\item[\rm (i)]   $h_m$ is $\bar{\mcE}^{(F_m)}$-harmonic in $F_m\setminus \partial_o F_m$.

\item[\rm (ii)]  $h_m\rightarrowtail h$ as $m\to \infty$. 

\item[\rm (iii)] $\lim\limits_{m\to\infty}\bar{\mcE}^{(F_m)}(h_m)=\mcE^{(F)}(h)$.   
\end{enumerate}
\end{lemma}

\begin{proof}
By the Mosco convergence of $\bar{\mcE}^{(F_n)}$ to $\mcE^{(F)}$, there is $f_n\in W^{1,2}(F_n)$ so that $f_n\to h$ strongly in $L^2$ and $\limsup\limits_{n\to\infty}\bar{\mcE}^{(F_n)}(f_n)\leq\mcE^{(F)}(h)$. Replacing $f_n$ by $(-\| h\|_\infty )\vee (h_n \wedge \| h\|_\infty)$ if needed, we may and do assume that $\|f_n\|_\infty\leq \|h\|_\infty$. By the Mosco convergence of $\bar{\mcE}^{(F_n)}$ to $\mcE^{(F)}$ and Theorem \ref{thm211}, for each $m\geq 1$, $ \bar{U}^{(F_n)}_m  f_n\to  U^{(F)}_mh$ strongly in $L^2$ as $n\to \infty$.
 Moreover, by (A3) and Lemma \ref{lemma28},  $\bar{U}^{(F_n)}_m  f_n\rightarrowtail  U^{(F)}_m h$ as $n\to \infty$. Next, as $mU^{(F)}_m h\to h$ uniformly as $m\to\infty$,   by Lemma \ref{lemma29}(a), there is an increasing subsequence $\{m_n; n\geq 0\}$ so that $g_n\rightarrowtail h$, 
 where $g_n:= m_n\bar{U}^{(F_n)}_{m_n} f_n$, which by condition (A3) is continuous on $F_n$.
 Applying  Lemma \ref{lemma28}(b) by taking $A_n=\partial_o F_n$ and $A=\partial_o F$,
 we conclude from Remark \ref{R:2.9} that 
 \begin{equation}\label{e:4.5} 
\{  g_n |_{\partial_o F_n}; n\geq 0\} \hbox{ are equicontinuous}   \hbox{ and }   g_n |_{\partial_o F_n} \rightarrowtail h|_{\partial_o F} \quad \hbox{as } n\to \infty.
 \end{equation} 
Moreover,  as $\bar{\mcE}^{(F_n)}(g_n)\leq \bar{\mcE}^{(F_n)} (f_n)$ for each $n\geq 1$, we have 
$$
\limsup\limits_{n\to\infty}\bar{\mcE}^{(F_n)}(g_n) \leq \limsup\limits_{n\to\infty}\bar{\mcE}^{(F_n)} (f_n)
 \leq\mcE^{(F)}(h). 
$$

Let $h_m\in W^{1,2}(F_m)\cap C(F_m)$ be the unique function such that $h_m|_{\partial_o  F_m}=g_m|_{\partial_o  F_m}$ and $h_m$ is $\bar{\mcE}^{(F_m)}$-harmonic
in $F_m\setminus \partial_o  F_m$, so (i) holds. By \eqref{e:4.5} and condition (A1), there is  a subsequence $\{m_k; k\geq 1\}$ such that
 $h_{m_k}\rightarrowtail \wt h$ for some $\wt h\in C(F)\cap \mcF^{(F)}$.  Hence by \eqref{e:4.5} and Lemma \ref{lemma29}(b)
 that 
\begin{equation}\label{eqn43} 
 \wt h =h  \quad \hbox{on } \partial_o  F 
 \end{equation}
  and $ h_{m_k} \rightarrow \wt h$  strongly in $L^2$ as $k\to \infty$.  By the Mosco convergence of $\bar{\mcE}^{(F_m)}$,  
 \begin{equation}
 \mcE^{(F)}(\wt h)= \lim\limits_{k\to\infty}\bar{\mcE}^{(F_{m_k})}(h_{m_k})\leq \limsup\limits_{k\to\infty}\bar{\mcE}^{(F_{m_k})}(g_{m_k})=\mcE^{(F)}(h).  \label{eqn44} 
\end{equation}
 Since $h$ is the unique function in $C(F)\cap \mcF^{(F)}$ that minimizes the energy among $\{f\in C(F)\cap \mcF^{(F)}:f|_{\partial_o  F}=h|_{\partial_o  F}\}$,
we conclude from \eqref{eqn43}  that $\wt h=h$ and, consequently,  $h_{m_k}\rightarrowtail h$ and   $\lim\limits_{k\to\infty}\bar{\mcE}^{(F_{m_k})}(h_{m_k})=\mcE^{(F)}(h)$. Since the above argument works for any subsequence as well, by Lemma \ref{lemma29}(a), we have  (ii) and (iii).  
\end{proof}

\subsection{A finite dimensional kernel}\label{sec42}
In this subsection, we show that we can find continuous harmonic functions with assigned mean values on every  level-$n$ sub-face, that almost minimizes the energies. Recall that $\eth_nF $  and $\eth_nF_m$  are  the collections of level $n$ sub-faces of $\partial_o F$ and $\partial_o F_m$, respectively. \smallskip

\medskip

 In the remainder  of this paper, the following elementary lemma  will be used several times.
 Let $G=(V,E)$ be a locally finite graph, where $V$ is a finite or countable set and $E\subset\big\{ (x,y): x,y\in V\big\}$ is the set of undirected edges.
 We assume assume that $G$ has bounded degrees, that is there is some constant $N \geq 1 $ so that
\[
\deg(x):=\#\big\{y\in V:  (x,y) \in E\big\}\leq N 
\quad \hbox{for every }  x\in V .
\] 
Let $d_G$ be the graph distance between $x,y\in V$, i.e. 
\begin{eqnarray*}
d_G(x,y)&:=&\min\big\{L:\hbox{ there exists } \{ x^{(0)},x^{(1)},x^{(2)},\cdots,  x^{(L)}  \}\subset  V  \hbox{ such that } \\
&& \hskip  0.6truein   x^{(0)}=x, \, x^{(L)}=y \hbox{ and  } (x^{(l-1)},x^{(l)}) \in E\hbox{ for each }1\leq l\leq L \big\}.
\end{eqnarray*}

 \begin{lemma}  for every $L\geq 1$ and $f\in l(V)$, 
  \begin{equation}\label{e:4.7}
\sum_{x\in V}\sum_{y\in V\atop d_G(x,y)\leq L}\big(f(x)-f(y)\big)^2\leq  L  (N+1)^{2(L-1)}\sum_{x,y\in V: (x,y)\in E}\big(f(x)-f(y)\big)^2  .
\end{equation}
\end{lemma}
\begin{proof}
	 For each $x, y\in V$ with $d_G(x,y)\leq L$, there is a geodesic path $ x^{(0)},x^{(1)},x^{(2)},\cdots, x^{(d_G(x,y))} $
	connecting $x$ and $y$, and so by Cauchy-Schwarz inequality,
	$$
	\big(f(x)-f(y)\big)^2 \leq  L   \sum_{l=1}^{d_G(x,y)}\big(f(x^{(l-1)})-f(x^{(l)})\big)^2.
	$$
	For each directed edge $(a, b)$ in $ E$ to be within a geodesic path from  some $x\in V$ to  $y\in V$ with $d_G(x, y) \leq L$, 
	$x$ and $y$ have to be within distance $L-1$ from $a$ and $b$, respectively. There are at most  $(N+1)^{L-1}$ of such $x$ and at most
	$(N+1)^{L-1}$ of such $y$.  Hence each directed edge $(a, b)$ can be used no more than $(N+1)^{2(L-1)}$ times 
	in directed geodesics of length at most $L$.
	This establishes \eqref{e:4.7} where $x$ and $y$ are un-ordered on both sides.  
\end{proof}

\begin{lemma}\label{lemma44}
There is a positive finite constant $C$ such that the following hold.
\begin{enumerate} 
\item[\rm (a)]  For each $n\geq 0$, there is a $(\#\eth_nF)$-dimensional linear subspace 
$\mathscr{C}_n\subset C(\partial_o  F)$ 
 that contains 
constant functions so that 

\smallskip

	\begin{itemize}
\item[\rm (a.1)]  for each $u\in l(\eth_n F)$, there is a unique $f\in \mathscr{C}_n$ so that    $\fint_A f(x)\nu(dx)  =u(A)$ for every 
$ A\in \eth_n F$. Moreover,  $\|f\|_\infty\leq C\|u\|_\infty$. 
	
	\smallskip

\item[\rm (a.2)]  $\Lambda_n[f]\leq CL_F^{n(d_w-d_f)}I_n[f]$ for every $ f\in \mathscr{C}_n$.
\end{itemize} 
	
	\medskip
	
\item[\rm (b)] For each $m\geq n\geq 0$, there is a $(\#\eth_nF_m=\#\eth_nF)$-dimensional  linear 
subspace $\mathscr{C}_{m,n}\subset C(\partial_o  F_m)$ 
 that contains 
constant functions so that  
	
	\smallskip
	
\begin{itemize}
\item[\rm (b.1)]   for each $u\in l(\eth_n F_m)$, there is a unique $f\in \mathscr{C}_{m,n}$ so that $\fint_A f(x)\nu_m (dx)=u(A)$
for every $ A\in \eth_n F_m$. Moreover, $\|f\|_\infty\leq C\|u\|_\infty$.
	
	\smallskip
	
\item[\rm (b.2)]  $\Lambda^{(m)}_n[f]\leq CL_F^{n(d_w-d_f)}I^{(m)}_n[f]$ for every $ f\in \mathscr{C}_{m,n}$.
\end{itemize} 
\end{enumerate}
\end{lemma}

\medskip

\begin{remark}\label{R:4.7} \rm 
If $u=c$  in  $l(\eth_nF)$  for some constant $c$, then by the uniqueness, the corresponding function $f=c$ in both (a.1) and (b.1) 
\end{remark}
 
 \medskip
 
\noindent{\it Proof of Lemma \ref{lemma44}}. 
Here, we only give a proof  for  (a), as  (b) follows by a similar argument.
Fix $u\in l(\eth_n F)$.
We construct a good $f_u\in C(\partial_o  F)$ in two steps. The function $f_u$ will depend on $u$ linearly
 so we get a linear space $\mathscr{C}_n$. \\

\noindent\textbf{Step 1}. For each $A\in \eth_nF$,   let $Q_A$ be the unique cube in $\mcQ_n(F_n )$ such that $A\subset Q_A$, and   
$g_A:=w_{Q_A}  \in C(F)\cap \sF^{(F)}$ be as in Lemma \ref{lemma37} for which we fix one, so
\[
g_A =1 \hbox{ on } F_{Q_A},\ g_A=0 \hbox{ on } F\setminus F_{\mathcal{S}_{Q_A}} ,\ 0\leq  g_A \leq 1\text{  and } \, \mcE^{(F)}(g_A)\leq C_1L_F^{(d_w-d_f)n},
\]
where $C_1$ the constant in Lemma \ref{lemma37} and  $\mathcal{S}_{Q_A} :=\{Q\in\mcQ_n(F):Q \cap Q_A\neq \emptyset\}$. For each $A\in \eth_nF$, 
define 
\[
g^*_A :=\frac{g_A}{(\sum_{B\in \eth_nF}g_{B})\vee 1} \in C(F)\cap \sF^{(F)}. 
\]
Observe that $\sum_{A\in \eth_nF} g^*_A =1$ on $  F$. Recall that 
$$
\mcB_n (F):=\{Q\in \mcQ_n:Q\cap \partial F_0\neq \emptyset \hbox{ and }{\rm int}(Q) \cap F  \not=\emptyset \} . 
$$
For each $Q\in\mcB_n(F)\cap \mathcal{S}_A$, 
 by the strongly local and derivation properties of the energy measure $\mu^{(F_Q)}_{\<u\>}$ for $\mcE^{(F_Q)}$ (see, e.g. \cite[Theorem 4.3.7]{CF}), 
 \begin{align*}
	\mu^{(F)}_{\<g^*_A\>}(F_Q)
 	&\leq \mu^{(F)}_{\<g_A\>}(F_Q)+2\mu^{(F)}_{\< ( \sum_{B\in \eth_nF}g_{B})\vee 1\>} (F_Q)\\
	&\leq 2\mu^{(F)}_{\<g_A\>}(F_Q)+2\mu^{(F)}_{\< \sum_{B\in \eth_nF}g_{B}  \>} (F_Q)\\
	&\leq 2\mu^{(F)}_{\<g_A\>}(F_Q)+2\mu^{(F)}_{\< \sum_{B\in \eth_nF:Q\in \mathcal{S}_{Q_{B}}}g_{B}\>} (F_Q)\\
  	& \leq (2+2C_2^2)\, C_1L_F^{(d_w-d_f)n} ,
\end{align*}  
where $\#\{B\in \eth_nF:Q\subset \mathcal{S}_{Q_B}\}\leq C_2=3^d$. 

\medskip

Let $f^*_u= \sum\limits_{A\in \eth_n F}u(A)g^*_A |_{\partial_o  F}$. Clearly,
\begin{equation}  \label{eqn45}
\min\limits_{B\in \eth_n F:A\cap B\neq \emptyset}u(B)\leq  f^*_u(x)\leq \max\limits_{B\in \eth_n F:A\cap B\neq \emptyset}u(B)
\quad \hbox{for }  A\in\eth_nF \hbox{ and } x\in A.
\end{equation}
 Set $g^*_u=\sum\limits_{A\in \eth_nF}u(A)g^*_A$. Then
\begin{align*}
\mu^{(F)}_{\<g^*_u\>}(F_{\mcB_n(F)})& 
 \leq\sum_{Q\in\mcB_n(F)}\mu^{(F)}_{\<g^*_u\>}(F_Q)=\sum_{Q\in\mcB_n(F)}\mu^{(F)}_{\<g^*_u-u(A_Q)\>}(F_Q) \\
& =\sum_{Q\in \mcB_n(F)}\mu^{(F)}_{\big\<\Big(\sum_{B\in \eth_nF:Q\in \mathcal{S}_{Q_{B}}}\big(u(B)-u(A_Q)\big)g^*_{B}\Big)\big\>}(F_Q) \\
&\leq C_2\, (2+2C_2^2)C_1L_F^{(d_w-d_f)n}\, \sum_{A\in \eth_nF}\sum_{B\in \eth_nF:B\cap Q_A \neq\emptyset}\big(u(B)-u(A)\big)^2\\
&\leq  2C_2^2 \,C_2\, (2+2C_2^2)C_1L_F^{(d_w-d_f)n}\cdot\sum\limits_{A,B\in\eth_nF: A\cap B\neq \emptyset}\big(u(A)-u(B)\big)^2,
\end{align*}
where in the first and second lines for each $Q\in \mcB_n(F)$,  $A_Q\in\eth_n F$
so that $A_Q\subset Q$,  and  \eqref{e:4.7} is used  in the last line. 
Thus by Theorem \ref{T:A.3}(a) there is some constant $C_3>0$ depending only on $F$ so that 
 \begin{equation} \label{eqn46}
\Lambda_{n+1}[f^*_u]\leq C_3 L_F^{(d_w-d_f)n}\sum\limits_{A,B\in\eth_nF: A\sim B}\big(u(A)-u(B)\big)^2 .
\end{equation}

\noindent\textbf{Step 2}. Applying Lemma \ref{lemma37} to $(n+1)$-cells, for each $A\in \eth_nF$, there is (and we fix one)
$g^{**}_A\in C(F)\cap \mcF^{(F)}$ such that 
\[
g^{**}_A|_{F\setminus Q_A}=0,\quad  [g^{**}_A]_{\nu|_A}=1,\quad  \|g^{**}_A\|_\infty\leq C_4 
\quad \text{and} \quad \mcE^{(F)}(g^{**}_A)\leq C_4L_F^{(d_w-d_f)n}
\]
for some constant $C_4$ depending only on $F$. 
Let  $f_u=f^*_u+f^{**}_u$, where $f^{**}_u=\sum_{A\in \eth_nF}\big(u(A)-[f^*_u]_{\nu|_A}\big) g^{**}_A$. 
 Clearly,  $[f_u]_{\nu|_A}=u(A)$ for $ A\in \eth_nF$. In addition, $\|f^*_u\|_\infty\leq \|u\|_\infty$ by \eqref{eqn45}.
 On the other hand,  $\|f^{**}_u\|_\infty\leq 2C_4\|u\|_\infty$ by the property of $g^{**}_A$ for $A\in\eth_nF$ and  \eqref{eqn45}. Hence, $\|f_u\|_\infty\leq (2C_4+1)\|u\|_\infty$. 
Note that $u\mapsto f_u$ is one-to-one on $l(\eth_n F)$. Indeed, if $f_u=f_v$ for $u, v \in l(\eth_n F)$, then for each $A\in \eth_n F$,
$$
u(A)= \fint_A f_u(x)\nu(dx)=\fint_A f_v(x)\nu(dx)=v(A);
$$
 that is, $u\equiv v$. 
 Let $ \mathscr{C}_n$ denote the collection of such constructed $f_u$ with $u \in  l(\eth_n F)$. 
 Since $l(\eth_n F) \ni u \mapsto  f_u \in C(\partial_o F)$ is linear,   
 $ \mathscr{C}_n  $ is a finite dimensional linear subspace of $ C(\partial_o F)$ with basis $\{f_{u_A}; A\in  \eth_n F\}$, where $u_A(A) :=1$ and $u_A(B):=0$
 for $B\in \eth_nF\setminus\{A\}$. Note that when $u=1$ on $\partial_o F$, by the above construction $f_u=1$ on $F$ so constants are in the space $ \mathscr{C}_n  $.   This establishes (a.1). 
 
For (a.2),  it follows from  \eqref{eqn45}, the local property of $\mcE^{(F)}$ and the energy estimates of $g^{**}_A$ that 
  \[
  \mcE^{(F)}(f^{**}_u)\leq C_5L_F^{n(d_w-d_f)}\sum\limits_{A,B\in\eth_nF: A\sim B}\big(u(A)-u(B)\big)^2=L_F^{n(d_w-d_f)}I_n[f_u]
  \] 
  for some constant $C_5>0 $ depending only on $F$. So by Theorem \ref{T:A.3}(a), 
  $\Lambda_{n+1}[f^{**}_u]\leq C_6L_F^{n(d_w-d_f)}I_n[f_u]$ for some $C_6 >0 $ depending only on $F$. Combining the estimate with \eqref{eqn46}
  gives $\Lambda_n[f_u]\leq (2C_3+2C_6+1)L_F^{n(d_w-d_f)}I_n[f_u]$. 
\qed 

\smallskip
 
\begin{remark} \rm Lemma \ref{eqn44}(a) implies that $\Lambda(\partial_o  F)\cap C(\partial_o  F)$ is dense in $\Lambda(\partial_o  F)$ with respect to the norm  
$ \sqrt{\Lambda_1[\cdot]+\|\cdot\|^2_{L^2(\partial_o  F; \nu)}}$.  
 This improves \cite[Theorem 2.6]{HK}, where the inequality stated in our Lemma \ref{lemma38} is assumed as a condition and there is a discrepency between two Besov spaces $\widehat \Lambda^\beta_{2, 2}(L)$ and $\Lambda^\beta_{2, 2}(L)$ there.   
\end{remark} 

In the sequel, for each $n\geq 0$, we fix a finite dimensional linear subspace $\mathscr{C}_n$ of $C(\partial_o F)$ 
as   in Lemma \ref{lemma44} (a); for each $m\geq n\geq 0$, we fix a finite dimensional linear subspace $\mathscr{C}_{m,n}$ of $C(\partial_o  F_m)$ as in Lemma \ref{lemma44} (b).
  
  \medskip
  
\begin{definition}\label{def45}
For each $n\geq 0$ and $f\in \mathscr{C}_n$,  let ${\mathcal H} f   $  be the unique function in $C(F)\cap\mcF^{(F)}$ such that ${\mathcal H} f =f$ on $\partial_o F$  and ${\mathcal H} f$ is $\mcE^{(F)}$ harmonic in $F\setminus \partial_o  F$.
\end{definition}

\medskip

\subsection{Proof of Proposition \ref{prop41}}\label{sec43} 

We finish the proof of Proposition \ref{prop41} in this subsection. The key idea is that we have some uniform control over approximating sequences of finite  dimensional cores ${\mathcal H} (\mathscr{C}_n)$ by compactness, which translates to general cases by trace theorems. 

\begin{lemma}\label{lemma46}
Assume (A1) and (A3). For every  $m\geq n\geq 0$,  
there is a  linear map $\Theta^{(m)}_n:\mathscr{C}_n\to W^{1,2}(F_m)\cap C(F_m)$   such that the following  hold for every $f\in \mathscr{C}_n$.
	
\begin{enumerate}
	
\item[\rm 	(1)]  $\Theta^{(m)}_nf$ is $\bar{\mcE}^{(F_m)}$-harmonic in $F_m\setminus \partial_o  F_m$  
and $\Theta^{(m)}_n1=1$.

\item[\rm 	(2)] 
 $\fint_{A_m} \Theta^{(m)}_nf (x) \nu_m (dx)= \fint_A f(x) \nu (dx)$
for each $m\geq n\geq 0$ and $A\in\eth_n F, A_m\in\eth_n F_m$ with  $A\subset A_m$. 
	
\item[\rm 	(3)]  For each $n\geq 0$, $\lim\limits_{m\to\infty}\bar{\mcE}^{(F_m)}(\Theta^{(m)}_nf)
=\mcE^{(F)} ({\mathcal H} f)$.
	
\item[\rm 	(4)]  For each $n\geq 0$, $\Theta^{(m)}_nf\to{\mathcal H} f$ strongly in $L^2$ as $m\to\infty$. 
\end{enumerate}
\end{lemma}

\begin{proof}
Fx $n\geq 0$, and   let $\{f^{(i)}\}_{i=1}^{\#\eth_nF}$ be a basis of the linear subspace $\mathscr{C}_n$, which  contains  
 the constant function $1$. We simply let $f^{(1)}=1$. Write $h^{(i)}={\mathcal H} f^{(i)}$. In particular, $h^{(1)}=1$. 

We first define a map $\Theta_n^{(m)}$ on $\{f^{(i)}\}_{i=1}^{\#\eth_nF}$. 
Define  $\Theta_n^{(m)}f^{(1)}=1$ for $m\geq n$. For $2\leq i\leq \#\eth_nF$, by  Lemma \ref{lemma43} there is
 a sequence of functions $h^{(i)}_m\in C(F_m)\cap W^{1,2}(F_m)$ that is $\mcE^{(F_m)}$-harmonic in $F_m\setminus \partial_o  F_m$ and
\begin{eqnarray}
\label{eqn47}&h^{(i)}_m\rightarrowtail h^{(i)} \quad \hbox{as } n\to \infty,  \\ 
\label{eqn48}&\lim\limits_{m\to\infty}\bar{\mcE}^{(F_m)}(h^{(i)}_m)=\mcE^{(F)}(h^{(i)}).
\end{eqnarray}
Let $u_m^{(i)}\in l(\eth_nF_m)$ be such that  $u_m^{(i)}(A_m)=[h^{(i)}_m]_{\nu_m|A_m}-[h^{(i)}]_{\nu|A}$ for each $A_m\in \eth_nF_m$ and $A\in \eth F$ with $A\subset A_m$. Let $g^{(i)}_m$ be the $\bar{\mcE}^{(F_m)}$-harmonic extension of the unique $f^{(i)}_m\in \mathscr{C}_{m,n}$ associated with $u_m^{(i)}$ as in Lemma \ref{lemma44} (b.1). 
By Lemma \ref{lemma28}(b), \eqref{eqn47} and Lemma \ref{lemma44}(b.1),
\begin{equation}\label{eqn49}
\lim\limits_{m\to\infty}\|f_m^{(i)}\|_\infty 
\leq C_1\lim\limits_{m\to\infty}\|u_m^{(i)}\|_\infty=0. 
\end{equation}
   If follows from \eqref{eqn49} that $\lim_{m\to\infty}L_F^{(d_w-d_f)k} I^{(m)}_k [f^{(i)}_m]=0$ for each $k\geq 1$ as there are no more than $2dL_F^{(d-1)k}$ many sub-faces in $\eth_kF_m$. Then, by Lemma \ref{lemma44}(b.2),  
\begin{align*}
\lim_{m\to \infty}\Lambda_1^{(m)}[g^{(i)}_m|_{\partial_o  F_m}]
&=\lim_{m\to \infty}  \Big(\sum_{k=1}^{n-1}L_F^{(d_w-d_f)k} I^{(m)}_k [f^{(i)}_m]+\Lambda_n^{(m)}[f^{(i)}_m]\Big)\\
&\leq\lim_{m\to \infty} \Big(\sum_{k=1}^{n-1}L_F^{(d_w-d_f)k} I^{(m)}_k [f^{(i)}_m]+C_1I_n^{(m)}[f^{(i)}_m]\Big)=0.
\end{align*}
Hence by Theorem \ref{T:A.4}(b),   
\begin{equation}\label{eqn410}
	\lim\limits_{m\to\infty}\bar{\mcE}^{(F_m)}(g_m^{(i)})=0.
\end{equation}
Define  $\Theta_n^{(m)}f^{(i)}=h^{(i)}_m+g^{(i)}_m$ for $m\geq n$. Then on $\{f^{(i)}\}_{i=1}^{\#\eth_nF}$, properties (1) and (2) hold by the construction, (3) holds by \eqref{eqn48} and \eqref{eqn410}, and (4) holds by \eqref{eqn47}, \eqref{eqn49} and Lemma \ref{lemma29}(b). 
We extend the definition of $\Theta^{(m)}_n$ linear to $\mathscr{C}_n$, that is, define 
\[
\Theta^{(m)}_n(\sum_{i=1}^{\#\eth_n F}c_if^{(i)}):=\sum_{i=1}^{\#\eth_n F}c_i\Theta_n^{(m)}(f^{(i)})
\quad\hbox{for every }  m\geq n \hbox{ and } \{c_i\}_{i=1}^{\#\eth_nF}\subset  \R .
\] 
 Properties  (1), (2) and (4) clearly hold on $\mathscr{C}_n$, while (3) follows directly from   the following property.
 If $f_m\to f$ weakly in $L^2$, $g_m\to g$ weakly in $L^2$, $\lim\limits_{m\to\infty}\bar{\mcE}^{(F_m)}(f_m)=\mcE^{(F)}(f)$ and $\lim\limits_{m\to\infty}\bar{\mcE}^{(F_m)}(g_m)=\mcE^{(F)}(g)$, then 
 \begin{equation}\label{e:4.14}
 \lim\limits_{m\to\infty}\bar{\mcE}^{(F_m)}(f_m+g_m)=\mcE^{(F)}(f+g).
 \end{equation} 
Indeed,  since $f_m+g_m$ and $f_m-g_m$ converges weakly in $L^2$ to $f+g$  and $f-g$, respectively, it follows from the Mosco convergence  of $\bar{\mcE}^{(F_m)}$ to ${\mcE}^{(F )}$ that 
  \[
\liminf\limits_{m\to\infty}\bar{\mcE}^{(F_m)}(f_m+g_m)\geq \mcE^{(F)}(f+g),\qquad \liminf\limits_{m\to\infty}\bar{\mcE}^{(F_m)}(f_m-g_m)\geq \mcE^{(F)}(f-g).
\]
On the other hand, by the parallelogram equality, we know that 
\[
\begin{aligned}
	\lim\limits_{m\to\infty}\bar{\mcE}^{(F_m)}(f_m+g_m)+\bar{\mcE}^{(F_m)}(f_m-g_m)
	=&\lim\limits_{m\to\infty} 2\bar{\mcE}^{(F_m)}(f_m)+2\bar{\mcE}^{(F_m)}(g_m)\\
	=&2\mcE^{(F)}(f)+2\mcE^{(F)}(g)=\mcE^{(F)}(f+g)+\mcE^{(F)}(f-g).
\end{aligned}
\]
So identity \eqref{e:4.14} must hold.
\end{proof}

\medskip

\begin{lemma}\label{lemma47}
 Let  $n\geq 1$. 
\begin{enumerate} 	
\item[\rm (a)] There is a constant $C\in (0,1)$ such that $\mu^{(F)}_{\<{\mathcal H} f\>}(F_{ \mcB_{n-1}})\geq C\, \mcE^{(F)}({\mathcal H} f)$ for each $f\in \mathscr{C}_n$. 
	
\item[\rm (b)]   Suppose that (A1) and (A3) hold. There is a sequence of positive numbers $\{C_m\geq 1; m\geq n\}$ such that $\lim\limits_{m\to\infty}C_m=1$ and 
$$
C_m^{-1}\mcE^{(F)}({\mathcal H} f)\leq \bar{\mcE}^{(F_m)}(\Theta^{(m)}_nf)\leq C_m\mcE^{(F)}({\mathcal H} f)
\quad \hbox{for every  }  f\in \mathscr{C}_n.
$$

\item[\rm (c)]  Suppose that  (A1),(A2) and (A3) hold. There is a sequence of positive numbers $\{C'_m\geq 1; m\geq n\}$ such that $\lim\limits_{m\to\infty}C'_m=1$ and 
$$
\bar{\mu}^{(F_{m})}_{\<\Theta^{(m)}_nf\>}(F_{m,\mathcal{B}_{n-1}})\leq C'_m \mu^{(F)}_{\<{\mathcal H} f\>}(F_{\mathcal{B}_{n-1}}) 
\quad \hbox{for every } f\in\mathscr{C}_n.
$$
\end{enumerate}
\end{lemma}

\begin{proof}   
Let $M_n=\{f\in \mathscr{C}_n:\mcE^{(F)}({\mathcal H} f)=1, \int_{\partial_oF}f(x)\nu(dx)=0\}$, which is a compact subset of 
of the finite dimensional vector space $\mathscr{C}_n \subset C(\partial_o F)$. 

(a)  By compactness, we only need to show $\mu^{(F)}_{\<{\mathcal H} f\>}(F_{\mcB_{n-1}})>0$ for each $f\in M_n$. Since ${\mathcal H} f$ is continuous on $F$ and the form $(\mcE^{(F_Q)},\mcF^{(F_Q)})$ is irreducible for each $Q\in \mathscr{Q}_{n-1}$, it suffices to show ${\mathcal H}f$ not  constant on ${F_Q}=F\cap Q$  for some $Q\in \mcB_{n-1}$.  We prove this by contradiction. Suppose that there is a function $f\in M_n$ so that ${\mathcal H}f$ is constant on $F_Q$, so in particular 
$f$ is constant on  $F_Q\cap\partial_o  F$, for each $Q\in \mcB_{n-1}$. In this case, we take $Q,Q'\in\mcB_{n-1}$ such that 
$$
f|_{A_Q}=\min_{x\in \partial_o  F}f(x) < f|_{A_{Q'}}=\max_{x\in \partial_o  F}f(x),
$$
 where $A_Q=F\cap Q\cap\partial_o  F$ and $A_{Q'}=F \cap {Q'}\cap\partial_o  F$. By Proposition \ref{P:3.5}  and strong Markov property,
\[
\mathbb{P}_x^{(F)} \Big( X^{(F)}_{\tau_{A_{Q'}}}<X^{(F)}_{\tau_{A_Q}} \Big) >0
\quad  \hbox{for }  x\in(F\setminus \partial_o F) \cap Q, 
\]
which implies that ${\mathcal H} f(x)>f|_{A_Q}$ for $ x\in (F\setminus \partial_o F) \cap Q$. Hence  ${\mathcal H}f$ is not a constant on $F\cap Q$, which is a contradiction. Thus we get $\mu^{(F)}_{\<{\mathcal H} f\>}(F_{\mcB_{n-1}})>0$ for every  $f\in M_n$.

\smallskip

(b)  Note that by the irreducibility of the reflected Brownian motion on $F_m$, 
for each $m\geq n$ and $f\in \mathscr{C}_n$,  $\bar{\mcE}^{(F_m)}(\Theta^{(m)}_n f) =0$ if and only if $\Theta^{(m)}_n f$ is a constant on $F_m$ by \eqref{e:4.1},  which happens if and only if $f$ is a constant in view of  Lemma \ref{lemma46}(2) and Remark \ref{R:4.7}. On the other hand, ${\mcE}^{(F)} ({\mathcal H} f)=0$ if and only if ${\mathcal H} f$ is constant, which happens if and only if $f$ is constant on  
$\partial_o F$.  For two functions $f, g$ on $\partial_o F$, if we define them to be equivalent, denoted by $f \sim g$, if  and only if they differ by  a constant, 
then  both   $f\mapsto  \bar{\mcE}^{(F_m)}(\Theta^{(m)}_n f)^{1/2}$  and $f\to {\mcE}^{(F)} ({\mathcal H} f)^{1/2}$ are  norms on the quotient space $\mathscr{C}_n/\sim$, which is a linear finite dimensional space. Hence there is a constant $C_m\geq 1$ so that 
\begin{equation}\label{e:4.15} 
C_m^{-1}\mcE^{(F)}({\mathcal H} f)\leq \bar{\mcE}^{(F_m)}(\Theta^{(m)}_nf)\leq C_m\mcE^{(F)}({\mathcal H} f)
\quad \hbox{for every  }  f\in \mathscr{C}_n.
\end{equation} 

 Next, we introduce a norm on the finite dimensional space  $\mathscr{C}_n$ as follows, which will also be used in the proof of (c). Let $\{f^{(i)}\}_{i=1}^{\#\eth_nF}$ be a basis of $\mathscr{C}_n$ as in the proof of Lemma \ref{lemma46}, and we define the norm $\|\cdot\|_{\mathscr{C}_n}$ by
\begin{equation}\label{e:416}
\big\| \sum_{i=1}^{\#\eth_nF}c_if^{(i)}  \big\|_{\mathscr{C}_n}:=\sum_{i=1}^{\#\eth_nF}|c_i|\quad\hbox{for each }c_1,c_2,\cdots,c_{\#\eth_nF}\in \R.
\end{equation}
By Lemma \ref{lemma46}(3), there exists $C^*_1>0$ so  that 
\begin{equation}\label{e:417}
\sup_{m\geq n}\sqrt{\bar{\mcE}^{(F_m)}(\Theta^{(m)}_n f^{(i)})}\leq C^*_1\quad\hbox{ for each }i=1,2,\cdots,\#\eth_nF.
\end{equation}
Since $\sqrt{\bar{\mcE}^{(F_m)}(\Theta^{(m)}_n\bullet)}$ is a 
seminorm on $\mathscr{C}_n$, 
we have by  \eqref{e:416}, \eqref{e:417} and the triangle inequality that    for any $f=\sum_{i=1}^{\#\eth_nF}c_if^{(i)}\in \mathscr{C}_n$
\begin{equation}\label{e:4.18}
\sup_{m\geq n}\sqrt{\bar{\mcE}^{(F_m)}(\Theta^{(m)}_n f)}\leq C_1^*\|f\|_{\mathscr{C}_n}.
\end{equation}
Hence
 $$
 \left| \sqrt{\bar{\mcE}^{(F_m)}\big(\Theta^{(m)}_n f\big)}-\sqrt{\bar{\mcE}^{(F_m)}\big(\Theta^{(m)}_ng\big)} \right|
 \leq \sqrt{\bar{\mcE}^{(F_m)}\big(\Theta^{(m)}_n (f-g)\big)}\leq C_1^*\|f-g\|_{\mathscr{C}_n}
 $$
  for any $f,g\in \mathscr{C}_n$ and $m\geq n$, which means that $\sqrt{\bar{\mcE}^{(F_m)}\big(\Theta^{(m)}_n\bullet)}$ is equicontinuous on $\mathscr{C}_n$. 
  Since by Lemma \ref{lemma46}(3),   $ \lim_{m\to \infty} \sqrt{\bar{\mcE}^{(F_m)}(\Theta^{(m)}_nf)}=\sqrt{\sE^{(F)}(\mathcal H f)}=1$ for every $f\in M_n$ and $M_n$ is compact, it follows that   the converges is uniform on $M_n$. So we can take the constant $C_m\geq 1$ in \eqref{e:4.15} in such a way that it converges to $1$ as $m\to \infty$.

\smallskip

(c)  Let  $f\in M_n$, and set  $h_m :=\Theta^{(m)}_nf$ for  $m\geq n$ and $h :={\mathcal H} f$ for short. Since $h_m$ converges to $h$ strongly in $L^2$
by  Lemma \ref{lemma46}(4) ,  it follows from   the fact  that  each $ \bar{\mu}^{(F_{m})}_{\<h_{m}\> }$ does not charge on
sets having zero Lebesgue measure  from \eqref{e:4.1} and condition (A2) that 
\[
\begin{aligned}
\liminf\limits_{m\to\infty}\bar{\mu}^{(F_{m})}_{\<h_{m}\>}(F_{m}\setminus F_{m,\mcB_{n-1}})&=\liminf\limits_{m\to\infty}\sum_{Q\in \mathscr{Q}_{n-1}\setminus \mcB_{n-1}}\bar{\mu}^{(F_{m})}_{\<h_{m}\>}(F_{m,Q})\\
&\geq \sum_{Q\in \mathscr{Q}_{n-1}\setminus \mcB_{n-1}}\mu^{(F)}_{\<h\>}(F_Q)\geq
\, \mu^{(F)}_{\<h\>}(F\setminus F_{\mcB_{n-1}}).
\end{aligned}
\]
On the other hand,  by Lemma \ref{lemma46}(3)
$$
\lim\limits_{m\to\infty}\bar{\mu}^{(F_{m})}_{\<h_{m}\>}(F_{m})= 
\lim\limits_{m\to\infty}\bar{\mcE}^{(F_{m})}(h_{m})=\mcE^{(F)}(h)
= \mu^{(F)}_{\<h\>}(F). 
$$
Hence 
\begin{equation}\label{e:4.19}
\limsup\limits_{m\to\infty}\bar{\mu}^{(F_{m})}_{\<\Theta^{(m)}_n f \>}(F_{m,\mathcal{B}_{n-1}})\leq \mu^{(F)}_{\< {\mathcal H} f\>}(F_{\mathcal{B}_{n-1}})  \quad \hbox{for every } f\in M_n.
\end{equation}

Note that  for every Borel subset $A\subset F_m$, 
$$
 \bar{\mu}^{(F_{m})}_{\<f, g\>} (A):=    L_F^{(d_w- 2)n} \int_{A}  \nabla f(x) \,\nabla g(x) \mu_m(dx)
$$
is a non-negative symmetric bilinear form in $f, g\in W^{1,2}(F_m)$ with $\bar{\mu}^{(F_{m})}_{\<f, f\>}= \bar{\mu}^{(F_{m})}_{\<f\>}$.   
This together with  \eqref{e:4.18} implies that for any $f,g\in \mathscr{C}_n$ and $m\geq n$, 
$$
\left| \sqrt{\bar{\mu}^{(F_{m})}_{\<\Theta^{(m)}_nf\>}(F_{m,\mathcal{B}_{n-1}})}-\sqrt{\bar{\mu}^{(F_{m})}_{\<\Theta^{(m)}_ng\>}(F_{m,\mathcal{B}_{n-1}})} \right| 
\leq \sqrt{\bar{\mu}^{(F_{m})}_{\<\Theta^{(m)}_n(f-g)\>}(F_{m,\mathcal{B}_{n-1}})}\leq C_1^*\|f-g\|_{\mathscr{C}_n}.
$$
Hence $\big\{ \bar{\mu}^{(F_{m})}_{\<\Theta^{(m)}_nf\>}(F_{m,\mathcal{B}_{n-1}})^{1/2}; m\geq n \big\}$ is equi-continuous   on $\mathscr{C}_n$. The desired conclusion follows from this and \eqref{e:4.19}
\end{proof}

\medskip

\begin{proof}[Proof of Proposition \ref{prop41}]
By Lemma \ref{lemma47}, there is an integer  $N\geq n$ such that 
for every $m\geq N$ and $f\in \mathscr{C}_n$, 
\begin{eqnarray}\label{eqn411}
	|\bar{\mcE}^{(F_m)}(\Theta_n^{(m)}f)-\mcE^{(F)}({\mathcal H} f)|
	&\leq &   C\,\mcE^{(F)} ({\mathcal H} f)  
	\, \leq \, \mu^{(F)}_{\< {\mathcal H} f\>}(F_{\mathcal{B}_{n-1}}), 
	\\
	\label{eqn412}
	\bar{\mu}^{(F_{m})}_{\<\Theta_n^{(m)}f\>}(F_{m,\mcB_{n-1}})
	&\leq & 2\mu^{(F)}_{\< {\mathcal H} f\>}(F_{\mathcal{B}_{n-1}}),
\end{eqnarray}
where $ C\in (0,1)$ is the constant in Lemma \ref{lemma47}(a).

Let  $m\geq N$, and  let $g\in \mcF^{(F)}$ and $g_{m}\in W^{1,2}(F_{m})$  satisfy the assumption of the proposition. 
Take $f\in \mathscr{C}_n$ so that $[f]_{\nu|_A}= [g]_{\nu|_A}$ for every $ A\in\eth_nF$. 
For notational convenience, set $h :={\mathcal H} f$ and $h_m :=\Theta^{(m)}_nf$.
We need to estimate $\mcE^{(F)}(g-h)$, $|\mcE^{(F)}(h)-\bar{\mcE}^{(F_m)}(h_m)|$ and $\bar{\mcE}^{(F_{m})}(h_m-g_m)$. 
	 By Lemma \ref{lemma44} (a) and the definition of $I_n$ and $\Lambda_n$, 
\begin{equation}\label{eqn413}
\Lambda_n[h|_{\partial_o  F}]
=\Lambda_n[f ] \leq C_1  L_F^{n(d_w-d_f)} I_n [f]
=C_1 L_F^{n(d_w-d_f)} I_n [g|_{\partial_o  F}] \leq 
C_1 \Lambda_n[g|_{\partial_o  F}],
\end{equation}
where $C_1>0$ denotes the constant $C$ in Lemma \ref{lemma44}(a). As $g-h\in C(F)\cap \sF^{(F)}$ is $\sE^{(F)}$-harmonic in $F\setminus \partial_o F$, 
by Theorem \ref{T:A.4}(a) with the positive constant $C$ there denoted by $C_2$ below and the fact that $[g-h]_{\nu|_A}=0$ for every $A\in \eth_nF$, 
\begin{equation}\label{eqn414}
\begin{aligned}
	\mcE^{(F)}(g-h)&\leq C_2\, \Lambda_1[(g-h)|_{\partial_o  F}]\\
	&=C_2\, \Lambda_n[(g-h)|_{\partial_o  F}] \\
	&\leq 2C_2\, \big(\Lambda_n[g|_{\partial_o  F}]+\Lambda_n[h|_{\partial_o  F}]\big)\\
	&\leq 2C_2\, \big(\Lambda_n[g|_{\partial_o  F}]+C_1\, \Lambda_n[g|_{\partial_o  F}]\big)\\
	&\leq 2(C_1+1)C_2C_3\, \mu^{(F)}_{\<g\>}(F_{\mcB_{n-1}}), 
\end{aligned}
\end{equation}
where  in the last line, we used Theorem \ref{T:A.3}(a) with the positive constant $C$ there denoted by $C_3$.
Thus  by \eqref{eqn411} and  \eqref{eqn414},  
\begin{eqnarray}\label{eqn416}
	  |\mcE^{(F)}(h)-\bar{\mcE}^{(F_m)}(h_m)|
	  &\leq & \mu^{(F)}_{\< {\mathcal H} f\>}(F_{\mathcal{B}_{n-1}}) \nonumber \\
	  &\leq &   2\mu^{(F)}_{\<g\>}(F_{\mathcal{B}_{n-1}})+2\mcE^{(F)}(g-h) \nonumber \\
  	   	  &\leq & \big(2+4(C_1+1)C_2C_3\big)\, \mu^{(F)}_{\<g\>}(F_{\mcB_{n-1}}).
\end{eqnarray}
Since $g_m-h_m \in C(F_m) \cap W^{1,2}(F_m)$ are $\bar \sE^{(F_m)}$-harmonic in $F_m \setminus \partial_o F_m$, by \eqref{e:4.1}, Theorem \ref{T:A.4}(b), the fact $[g_m-h_m]_{\nu_m|_{A_m}}=0$ for every $A_m\in \eth_nF_m$ and Theorem \ref{T:A.3}(b), 
we have 
\begin{equation}\label{eqn417}
 \begin{aligned}
	\bar{\mcE}^{(F_{m})}(h_m-g_m)
	&\leq C_2C_0\, \Lambda^{(m)}_1[(h_m-g_m)|_{\partial_o  F_{m}}]\\&=C_2C_0\cdot\Lambda^{(m)}_n[(h_m-g_m)|_{\partial_o  F_{m}}]\\
	&\leq C_2C_0\, 2\big(\Lambda^{(m)}_n[h_m|_{\partial_o  F_{m}}]+\Lambda^{(m)}_n[g_m|_{\partial_o  F_{m}}]\big)\\
	&\leq C_2C_0\,   2\big(C_3C_0\cdot\bar{\mu}^{(F_m)}_{\<h_m\>}(F_{m,\mcB_{n-1}})+\Lambda^{(m)}_n[g_m|_{\partial_o  F_{m}}]\big)\\
	&\leq C_2C_0\cdot2\big(2C_3C_0\cdot\mu^{(F)}_{\<h\>}(F_{\mcB_{n-1}})+\Lambda^{(m)}_n[g_m|_{\partial_o  F_{m}}]\big)\\
	&\leq C_4\cdot(\mu^{(F)}_{\<g\>}(F_{\mcB_{n-1}})+\Lambda^{(m)}_n[g_m|_{\partial_o  F_{m}}]),
\end{aligned}
\end{equation}
where we used \eqref{eqn412} in the second to the last inequality, and a part of \eqref{eqn416} in the last inequalty with the constant $C_4>0$ being  a polynomial of $C_1,C_2,C_3$ and $C_0$. The proposition follows by combining  \eqref{eqn414}, \eqref{eqn416} and \eqref{eqn417}.  
\end{proof}

We end this section with a corollary that will not be used in this paper, but will be needed in a forthcoming paper \cite{CC}. 

\begin{corollary}\label{coro48}
	Let $(\bar{\mcE}^{(F_m,1)},W^{1,2}(F_m))$ and $(\bar{\mcE}^{(F_m,2)},W^{1,2}(F_m))$ be two sequences of strongly local regular Dirichlet forms on $L^2(F_m;\mu_m)$ such that the following  hold for $i=1,2$.
\begin{enumerate} 
\item[\rm (i)] $C_0^{-1}\mcE^{(F_m)}\leq \bar{\mcE}^{(F_m,i)}\leq C_0 \mcE^{(F_m)}$ for every $ m\geq 0$; 
	
	\smallskip
	
\item[\rm (ii)]  $\bar{\mcE}^{(F_m,i)}$ Mosco converges to $\mcE^{(F)}$ as $m\to\infty$; 
	
	\smallskip
	
\item[\rm (iii)]  (A1),  (A2) and (A3) hold for $\{(\bar{\mcE}^{(F_m,i)},W^{1,2}(F_m));  m\geq 0\}$.
\end{enumerate} 
Then for each $n\geq 0$ and $\varepsilon>0$, there exist $M\geq n$ (depending on $n,\varepsilon,\bar{\mcE}^{(F_m,1)},\bar{\mcE}^{(F_m,2)}$) and $C>0$ (independent of $n,\varepsilon,\bar{\mcE}^{(F_m,1)},\bar{\mcE}^{(F_m,2)}$ but dependent on $F$ and $C_0$) such that 
	\[
	\left|\sqrt{\bar{\mcE}^{(F_m,1)}(h_1)}-\sqrt{\bar{\mcE}^{(F_m,2)}(h_2)} \right|
	\leq C \left(\varepsilon\, \sqrt{\bar{\mcE}^{(F_m,1)}(h_1)}
	+\sqrt{\bar{\mu}^{(F_m,1)}_{\<h_1\>}(F_{m,\mcB_n})}\right)
	\] 
	for any $m\geq M$ and $h_i\in W^{1,2}(F_m)\cap C(F_m),i\in\{1,2\}$ 
	that  is $\bar{\mcE}^{(F_m,i)}$-harmonic in $F_m\setminus \partial_o  F_m$ for $i\in \{1,2\}$ with $h_1|_{\partial_o  F_m}=h_2|_{\partial_o  F_m}$. 
\end{corollary}

\begin{proof}   
For $n\geq 1$, let $\mathscr{C}_{n }$  be the  finite  dimensional linear subspace of $C(\partial_o F)$ in Lemma \ref{lemma44}(a). Since  $\mu^{(F)}_{\<g\>}(\partial_o  F)=0$ for each $g\in \mcF^{(F)}$ by Corollary \ref{coroA5}, there is some ${n_1}\geq n+2$ such that 
\begin{equation}\label{eqn418}
\mu^{(F)}_{\<\mathcal{H} f\>}
(F_{\mcB_{{n_1}-1}})\leq \varepsilon^2\, \mcE^{(F)}(\mathcal{H}f)
\quad\hbox{for every }  f\in \mathscr{C}_{n+1},
\end{equation}
where $\mathcal{H}f$ is the $\sE^{(F)}$-harmonic extension of $f$. For $i=1,2$, let $N_i\geq {n_1}$ be the constant Proposition of \ref{prop41} with   $\bar{\mcE}^{(F_m,i)},m\geq 0$ and ${n_1}$ in place of $\bar{\mcE}^{(F_m)},m\geq 0$ and $n$ there. Define  $N=N_1\vee N_2$.
	
Fix $m\geq N$. Let $h_1,h_2$ be the functions in the statement of this corollary. 
Let $f$ be the unique function in $\mathscr{C}_{n+1}$ so that $[f]_{\nu|_A}=[h_1]_{\nu_m|_{A_m}}$ for every  $A\in \eth_{n+1}F$ and $A_m\in \eth_{n+1}F_m$ with
 $A\subset A_m$. Set  $h=\mathcal{H}f$. Then by Theorem \ref{T:A.3} (a), there are positive constants $C_1,C_2,C_3$ depending on $F$ and $C_0$ so that 
\begin{equation}\label{eqn419}
\begin{aligned}
	\mcE^{(F)}(h)&\leq C_1\,\Lambda_1[h|_{\partial_o F}]=C_1\,\Lambda_1[f]\\
	&=C_1\,\big(\sum_{k=1}^{n}L_F^{k(d_w-d_f)}I_k[f]+\Lambda_{n+1}[f]\big)\\
	&\leq C_1\,\sum_{k=1}^{n}L_F^{k(d_w-d_f)}I^{(m)}_k[h_1|_{\partial_oF_m}]+C_1C_2\,L_F^{(n+1)(d_w-d_f)}I_{n+1}^{(m)}[h_1|_{\partial_oF_m}]\\
	&\leq (C_1\vee C_1C_2)\ \Lambda_1^{(m)}[h_1|_{\partial_o {F_m}}] \\
	&\leq C_3\bar{\mcE}^{(F_m,1)}(h_1), 
\end{aligned}
\end{equation}
  where   we used Lemma \ref{lemma44}(a.2) and the fact that $I_k[f]=I_k^{(m)}[h_1|_{\partial_oF}]$ for $1\leq k\leq n+1$ in second inequality, and Theorem \ref{T:A.4} (b) in the last inequality.

Next, let $\wt f$ be the unique function in $   \mathscr{C}_{{n_1}}$ such that $[\wt f]_{\nu|_A}=[h_1]_{\nu_m|_{A_m}}$ for each $A\in \eth_{{n_1}}F$ and $A_m\in \eth_{{n_1}}F_m$ 
with  $A\subset A_m$. Set  $\wt h :=\mathcal{H}\wt f$. Then   by   Theorem \ref{T:A.3}(a),
\begin{equation}\label{eqn420}
\begin{aligned}
	\mcE^{(F)}(h-\wt h)&\leq C_1\Lambda_1[(h-\wt h)|_{\partial_oF}]=C_1\Lambda_1[f-\wt f]\\
	&=C_1\Lambda_{n+1}[f-\wt f]\\
	&\leq 2C_1\Lambda_{n+1}[f]+2C_1\Lambda_{n+1}[\wt f]\\
	&\leq 4(C_1\vee C_1C_2)\,I_{n+1}^{(m)}[h_1|_{\partial_o  F_m}]\\
	&\leq 4C_3\,\bar{\mu}^{(F_m,1)}_{\<h_1\>}(F_{m,\mcB_n(F_m)}),
\end{aligned}
\end{equation}
where   the third inequality holds by a  similar argument as that for \eqref{eqn419}, while  the last inequality is due to Theorem \ref{T:A.4}(b). 
It follows from \eqref{eqn418}, \eqref{eqn419} and \eqref{eqn420}  that 
\begin{eqnarray} \label{e:4.32} 
\mu ^{(F)}_{\<\wt h \>}(F_{\mcB_{{n_1}-1}(F)})
&\leq & 2\mu^{(F)}_{\<h\>}(F_{\mcB_{{n_1}-1}(F)})+2\mcE^{(F)}(h-\wt h)  \nonumber \\
&\leq &2\varepsilon^2 C_3 \,  \bar{\mcE}^{(F_m,1)} (h_1)
+4C_3\bar{\mu}^{(F_m,1)}_{\<h_1\>}(F_{m,\mcB_n(F_m)}).
\end{eqnarray}
By Proposition \ref{prop41} applied to the pairs  $(h_1,\wt h)$ and $(h_2,\wt h)$,  condition (i) and $h_1|_{\partial_o  F_m}=h_2|_{\partial_o  F_m}$, 
\begin{eqnarray*}
&&\left|\sqrt{\bar{\mcE}^{(F_m,1)}(h_1)}-\sqrt{\bar{\mcE}^{(F_m,2)}(h_2)} \right| \\
&\leq & \Big| \sqrt{\bar{\mcE}^{(F_m, 1)}(h_1)}-\sqrt{\mcE^{(F)}(\wt h)} \Big| + \Big|\sqrt{\bar{\mcE}^{(F_m, 2)}(h_2)}-\sqrt{\mcE^{(F)}(\wt h)} \Big|  \\
&\leq& C_4  \left(\sqrt{\mu^{(F)}_{\<\wt h\>}(F_{\mcB_{{n_1}-1}(F)})}+\sqrt{\Lambda^{(m, 1)}_{n_1}[h_1  |_{\partial_o  F_{m}}]}
+\sqrt{\Lambda^{(m, 2)}_{n_1}[h_2  |_{\partial_o  F_{m}}]}  \right)\\
&\leq& C_5  \left(\sqrt{\mu^{(F)}_{\<\wt h\>}(F_{\mcB_{{n_1}-1}(F)})}+\sqrt{\Lambda^{(m, 1)}_{n_1}[h_1  |_{\partial_o  F_{m}}]}
  \right)\\
 &\leq & C_6  \left( \sqrt{\mu^{(F)}_{\<\wt h\>} (F_{\mcB_{{n_1}-1}(F)})}+\sqrt{\mu^{(F_m, 1)}_{\<h_1\>}(F_{m,\mcB_{{n_1}-1}(F_m)})}
    \right),
 \end{eqnarray*}
 where the last inequality is due to Theorem \ref{T:A.3}(b). The corollary then follows from this, \eqref{e:4.32} and that $n_1\geq n+2$. 
 \end{proof}

\medskip

\section{Convergence of resistances}\label{sec5}

In this section, we present the proof for Theorems \ref{T:1.1}, \ref{T:1.3} and \ref{T:1.4}. Recall from Theorem \ref{T:3.11} and 
Theorem \ref{T:3.12} that $\{\alpha_m; m\geq 1\}$ is a sequence of real numbers that are bounded between two positive constants
so that   $\alpha_m\mcE^{(F_m)}$ on $L^2(F_m; \mu_m)$ is Mosco convergent to $\mcE^{(F)}$ on $L^2(F; \mu)$. Let 
\[
\alpha:= \limsup\limits_{m\to\infty}\alpha_m .
\] 
 In the rest of the section, we fix a subsequence $\{m_k; k\geq 1\}$ so that 
 \begin{equation}\label{e:5.1a}
\lim_{k\to\infty}\alpha_{m_k}=\alpha.
\end{equation} 
 Hence $ \alpha \sE^{(F_{m_k})}$ is Mosco convergent to $  \sE^{(F)}$ as $k\to  \infty$. For $Q\in \mcQ_n(F_n)$ with $n\geq 1$, $F_{m_k,Q}$ is similar to $F_{m_k-n}$ when $k$ is sufficiently large. 
Hence by self-similarity, $\alpha_{m_k-n}\mcE^{(F_{m_k,Q})}$ is Mosco convergent to $\mcE^{(F_Q)}$ as $k\to \infty$.
Thus  for any $f_{m_k}\in W^{1,2}(F_{m_k})$, $ k\geq 1$,  and $f\in \mcF^{(F)}$ so that $f_{m_k}\to f$ strongly, we have 
\[
\liminf\limits_{k\to\infty} \alpha \mcE^{(F_{m_k,Q})}(f_{m_k})\geq \liminf\limits_{k\to\infty}\alpha_{m_k-n}\mcE^{(F_{m_k,Q})}(f_{m_k})\geq \mcE^{(F_Q)}(f).
\]
This shows that condition (A2) holds for  $(\bar \sE^{(k)}, \bar \sF^{(k)}):=\big(\alpha \mcE^{(F_{m_k})},W^{1,2}(F_{m_k})\big)$
for $k\geq 1$.

 \medskip
 
In the case that $\partial_o  F\neq \partial F_0$, we need one more preparation for the  proofs of Theorems \ref{T:1.1} and \ref{T:1.4}. 

\begin{lemma}\label{lemma51}
Let $n,m\geq 0$ and $h\in \mcF^{(F)}\cap C(F)$ that is $\mcE^{(F)}$-harmonic in $F\setminus \partial_o  F$. Then there exists  $g=g_{m,n}\in W^{1,2}(F_{n+m})\cap C(F_{n+m})$ such that the following  hold for some constant $C>0$ independent of $m,n$ and $h$.
\begin{enumerate}	
\item[\rm (i)]  For each $Q\in \mcQ_n(F_n)$, $A\in \eth_m F$ and $A_m\in \eth_mF_m$ with  $A\subset A_m$, 
$$ \fint_{\Psi_Q(A_m)} g(x) \nu_{m+n} (dx)=   \fint_{\Psi_Q(A )}  h(x) \nu (dx).
$$
	
	\smallskip
	
\item[\rm (ii)]  $
\sum\limits_{Q\in \mcQ_n(F_n)}\Lambda^{(m)}_m[(g\circ \Psi_Q)|_{\partial_o  F_m}]
 \, \leq \, C \sum_{Q\in \mcQ_n(F_n)}  L_F^{(d_w-d_f)m} I_m [(h\circ \Psi_Q)|_{\partial_o  F}]     \hfill \break 
{} \hskip 1.9truein    \leq C\sum_{Q\in \mcQ_n(F_n)}\Lambda_m[(h\circ \Psi_Q)|_{\partial_o  F}].
$
	
		\smallskip

\item[\rm (iii)]  For each $Q\in \mcQ_n(F_n)$, $x\in F_Q$, $y\in F_{m+n,Q}$,
$$|h(x)-g(y)|\leq C\max\limits_{x',y'\in F_{\mathcal{S}_Q}}|h(x')-h(y')|, 
$$
where   $\mathcal{S}_Q :=\{Q'\in \mcQ_n(F_n):Q'\cap Q\neq \emptyset\}$. 

	\smallskip
	
\item[\rm (iv)]  $g$ is $\mcE^{(F_{m+n})}$-harmonic in $\Psi_Q(F_m\setminus \partial_o  F_m)
=F_{m+n,Q}\setminus \partial_oF_{m+n,Q}$ for every  $Q\in \mcQ_n(F_n)$. 
\end{enumerate}
\end{lemma}

\begin{proof}
We first record two observations.  With a slight abuse of notations, in this proof, 
$\nu_0$ also denotes the Lebesgue measure on $[0,1]^{d-1}$, which is the same as 
the normalized $(d-1)$-dimensional Hausdorff measure on $[0, 1]^{d-1}$.
	
	\smallskip
	
\noindent{\it Observation 1.} For each  function $u$ defined on $\{0,1\}^{d-1}$, it can be uniquely extended to be a continuous function $v_u$ on $[0,1]^{d-1}$ that is
linear on each line segment parallel to a coordinate axis, i.e. $v_u\big(ta+(1-t)b\big)=tv_u(a)+(1-t)v_u(b)$ for any $0\leq t\leq 1$ and $a,b\in [0,1]^{d-1}$ such that $\#\{i\in \{1,\cdots,d-1\}:a_i\neq b_i\}=1$. Indeed, the uniqueness of such $v_u$ is clear as $u$ uniquely determines by linearity 
the values of $v_i$ on each faces of the unit hypercube $[0, 1]^{d-1}$, which in turn uniquely determine the values of $v_u$ in the interior of $[0, 1]^{d-1}$.
 For the existence, we can define $v_u$ recursively as follows.
 Let $u_0=u$, and   for $1\leq k\leq d-1$, define $u_k\in C\big([0,1]^k\times \{0,1\}^{d-k-1}\big)$ by 
\[
u_k(a,t,b)=(1-t)u_{k-1}(a,0,b)+tu_{k-1}(a,1,b)  \quad \hbox{for }  (a, t) \in [0,1]^{k-1}\times [0,1],b\in \{0,1\}^{d-k-1},
\]
with the convention that for a set $A$,  $A^0:= \emptyset$. 
Then the function  $v_u= u_{d-1}$ is a continuous function $[0,1]^{d-1}$ that has the desired property. In particular, 
$$
\int_{[0, 1]^{d-1}} v_u (x) \nu_0 (dx) = 2^{1-d} \sum_{y\in \{0,1\}^{d-1}} u(y).
$$

The restriction of $v_u$ on a face of $[0,1]^{d-1}$ depends only on the values of $u$ on the corner points in the face, so, we can glue such functions on copies of $[0,1]^{d-1}$. Another useful property, which can been easily seen through induction,  is that 
$$
|\nabla v_u(x)|\leq (d-1)\max_{y,z\in \{0,1\}^d \atop |y-z|=1}|u(y)-u(z)| 
\quad  \hbox{for }  x\in (0,1)^{d-1}.
$$

\medskip

\noindent {\it Observation 2.}  Let $w^*\in C^1([0,1]^{d-1})$ be such that $w^*=0$ on $\partial ([0,1]^{d-1})
:= [0, 1]^{d-1} \setminus (0, 1)^{d-1}$ and $\int_{[0, 1]^{d-1}} w^*(x) \nu_0(dx) =1$. For each $u\in l(\{0,1\}^{d-1})$ and $a\in \R$, define 
\begin{equation}\label{e:5.1}
w_{u,a} (x)=v_u (x) + 2^{1-d} \sum_{y\in \{0,1\}^{d-1}} (a-u(y)) \,w^* (x) \quad \hbox{for } x\in [0,1]^{d-1},
\end{equation} 
	where $u_v\in C[2^{-d+1}]$ is the function in {\it Observation 1}. Clearly, $w_{u, a}\in C^1([0, 1]^{d-1})$ with 
	$w_{u,a}|_{\partial ([0,1]^{d-1})}=v_u|_{\partial ([0,1]^{d-1})}$,
	$[w_{u,a}]_{\nu_0|_{[0,1]^{d-1}}}=a$, and for some constant $C_1>0$ depending only on $d$,
\begin{equation}\label{e:5.2}
|\nabla w_{u,a}(x)|\leq C_1\, \max_{y\in \{0,1\}^{d-1}}|u(y)-a|
\quad  \hbox{for } x\in(0,1)^{d-1}.
\end{equation}

\medskip
	
Fix $n, m\geq 0$. We will use \eqref{e:5.1}  to construct a function $g$ as follows. For $k\geq 0$, we let
\[
\mathcal{A}_k=\{\Psi_Q(B):Q\in \mcQ_n(F_n),\ B\in\eth_{m+k}F_m\},
\]
and, for each $A\in \mathcal{A}_0$,  let $\Phi_A:[0,1]^{d-1}\to A$ be an affine map (here the orientation is not important). Let $V_A=\Phi_A (\{0,1\}^{d-1})$ for each $A\in \mathcal{A}_0$ and $V=\bigcup_{A\in \mathcal{A}_0}V_A$.
	
	\smallskip
	
Let $h\in \mcF^{(F)}\cap C(F)$ that is $\mcE^{(F)}$-harmonic in $F\setminus \partial_oF$.

	\smallskip
	
\noindent\textbf{Step 1}. We construct $g^*$ on $V$ first. For each $x\in V$,  let 
\[
g^*(x)=\frac{1}{\#\{S\in\mathcal{A}_0,x\in V_S\}}\sum_{S\in \mathcal{A}_0:x\in V_S}[h]_{\nu|_{S\cap F}}.
	\]
	
\noindent\textbf{Step 2}. 
We extend $g^*$ to $g^{**}\in C(\bigcup_{A\in \mathcal{A}_0} A )$ using {\it Observation 2} \, as follows. 
For each $S\in \mathcal{A}_0$, let $v_S:=g^*|_{V_S}\circ \Phi_S$,  $a_S:=\fint_{S\cap F} h(x) \nu (dx)$, and  
\begin{equation}\label{e:5.3}
g^{**}|_{S}:=w_{v_S,a_S}\circ \Phi_S^{-1},
\end{equation}
where $w_{v_S,a_S}$ is defined as in \eqref{e:5.1}. By the construction of $g^*$ in Step 1 and \eqref{e:5.2}, there is some constant $C_2 >0$ depending only on $d$ so that 
\begin{equation}\label{eqn52}
\max_{  y\in [0,1]^{d-1}}|\nabla w_{v_S,a_S}(y)|\leq C_2\, \max_{S'\in \mathcal{A}_0:S\cap S'\neq \emptyset}|[h]_{\nu|_{S\cap F}}-[h]_{\nu|_{S'\cap F}}|.
\end{equation}
	
\noindent\textbf{Step 3}. Let $g$ be the unique function in $W^{1,2}(F_{n+m})$ 
so that $g=g^{**}$ on $\bigcup_{A\in \mathcal{A}_0} A$ and $g$ is harmonic in $F_{m+n}\setminus
\bigcup_{A\in \mathcal{A}_0} A$. Since every point on $\bigcup_{A\in \mathcal{A}_0} A$ is regular for $\bigcup_{A\in \mathcal{A}_0} A$
and $g^{**}\in C(\bigcup_{A\in \mathcal{A}_0} A)$, we have $g\in C(F_{n+m})$.
	
	\medskip
	
	From the construction, we see immediately that  $g_{n, m}:=g$ enjoys properties (i) and (iv) of the Lemma. 
	
	\medskip
	
	To see (iii), fix $Q\in \mcQ_n(F_n)$ and $x\in F_Q$. Then, for each $y\in \partial_o F_{m+n,Q}$ and $ S\in \mathcal{A}_0$ such that $y\in S\subset \partial_o F_{m+n,Q}$, we have by (i), \eqref{e:5.3}
	and \eqref{eqn52} that
	\begin{align}
	|g(y)-[h]_{\nu|_{S\cap F}}|
	&=\big|g^{**}(y)-[g^{**}]_{\nu_0|_S}\big|=\big|w_{v_S,a_S}\big(\Phi_S^{-1}(y)\big)-[w_{v_S,a_S}]_{\nu_0|_{[0,1]^{d-1}}}\big|\nonumber\\
	&\leq   \sqrt{d-1} \,   \sup_{z\in [0,1]^{d-1}} |\nabla w_{v_S,a_S}(z)|    \label{e:5.4}  \\
	&\leq  C_2  \, 
	\sqrt{d-1}\max\limits_{S'\in \mathcal{A}_0:S\cap S'\neq \emptyset}|[h]_{\nu|_{S\cap F}}-[h]_{\nu|_{S'\cap F}}|.
	\nonumber 
	\end{align}
	Hence 
	\[
	|h(x)-g(y)|\leq |h(x)-[h]_{\nu|_{S\cap F}}|+|g(y)-[h]_{\nu|_{S\cap F}}|\leq (1+C_2\sqrt{d-1})\max\limits_{x',y'\in
	F_{\mathcal{S}_Q}}|h(x')-h(y')|.
	\]
	Property (iii) now  follows immediately since $g$ is harmonic in $F_{m+n,Q}\setminus 
	\bigcup_{Q\in \mcQ_n(F_n) } \partial_o  F_{m+n,Q}$.
	
	\medskip
	
	It remains to show (ii). For $k\geq 0$ and $A,B\in \mathcal{A}_k$, we say $A\sim B$ if there is $Q\in \mcQ_n(F_n)$ such that $A,B\subset Q$ and $\Psi_Q^{-1}(A)\sim \Psi_Q^{-1}(B)$ in the sense of \eqref{e:4.2a} but with $m+k$ in place of $k$ there. 
	Then 
	\[
	\sum\limits_{Q\in \mcQ_n(F_n)}I^{(m)}_{m+k}[(g\circ \Psi_Q)|_{\partial_o  F_m}]=\sum_{A,B\in \mathcal{A}_k:A\sim B}\big([g]_{\nu_{m+n}|_A}-[g]_{\nu_{m+n}|_{B}}\big)^2.
	\] 
	
	  Next,   for $A\in \mathcal{A}_k$ with 
	$k\geq 1$ and $S\in \mathcal{A}_0$ such that $A\subset S$
	\begin{eqnarray}\label{e:5.6}
	\big|g^{**}(y)-[g^{**}]_{\nu_0|_A}\big|&=&
	\big|w_{v_S,a_S}\big(\Phi_S^{-1}(y)\big)-[w_{v_S,a_S}]_{\nu_0|_{\Phi_S^{-1}(A)}}\big|  \nonumber \\
	&\leq & L_F^{-k}\sqrt{d-1}\sup_{z\in [0,1]^{d-1}} |\nabla w_{v_S,a_S}(z)|   \nonumber \\
	&\leq &  L_F^{-k}C_2\sqrt{d-1}\max\limits_{T\in \mathcal{U}(S)}|[h]_{\nu|_{S\cap F}}-[h]_{\nu|_{T\cap F}}|,
	\end{eqnarray}
where in the first inequality we used the fact that  $\Phi_S^{-1}(A)$ is a $(d-1)$-dimensional cube with side length $L_F^{-k}$, and in the second inequality we used \eqref{eqn52} with 
$\mathcal{U}_S := \{T\in \mathcal{A}_0:\ T\cap S\neq \emptyset\}$ for $S\in \mathcal{A}_0$.
 Fix $k\geq 1$ and $A,B\in \mathcal{A}_k$ such that $A\sim B$. 
	Let $S,S'\in \mathcal{A}_0$ be such that $A\subset S$, $B\subset S'$. Then either $S=S'$ or $S\cap S'\neq \emptyset$.  Then there are positive constants 
	$C_3$ and $C_4$ depending only on $d$ so that 
	\begin{eqnarray}
	&& \big|[g]_{\mu_{m+n}|_A}-[g]_{\mu_{m+n}|_{B}}\big|^2  \nonumber \\
	& \leq &{C_3\, L_F^{-2k}}\Big(\sum_{T\in {\mathcal U}_S} \big([h]_{\nu|_{T\cap F}}-[h]_{\nu|_{S\cap F}}\big)^2+\sum_{T\in \mathcal{U}_{S'}} \big([h]_{\nu|_{T\cap F}}-[h]_{\nu|_{S'\cap F}}\big)^2\Big)  \nonumber \\
	&\leq &C_4\, {L_F^{-2k}}\, 
	\Big(\sum_{T,T'\in  {\mathcal U}_S:T\sim T'} \big([h]_{\nu|_{T\cap F}}-[h]_{\nu|_{T'\cap F}}\big)^2+\sum_{T,T'\in\mathcal{U}_{S'}:T\sim  T'}\big([h]_{\nu|_{T\cap F}}-[h]_{\nu|_{T'\cap F}}\big)^2\Big),  \nonumber \\
	\label{e:5.7}
	\end{eqnarray}
	 	 where we used \eqref{e:5.6} (by taking $y\in S\cap S'$) in the first inequality, 
	while for the second inequality, we used (SC3) on the non-diagonality of {\it GSC} $F$,
	which implies that if $T\cap S\neq \emptyset$, 
	there are $S=S_1,S_2,\cdots, S_l=T\in \mathcal{A}_0$ with  $l\leq 2^{2d}$ so that 
	$S\cap S_i\neq \emptyset$ for every $ i\in \{1,\cdots, l\}$ and $S_i\cup S_{i+1}\subset Q$ for some $Q\in \mathcal{Q}_{m+n}(F)$
	(so $S_i\sim S_{i+1}$) for every $i\in \{1,\cdots,l-1\}$.

	For $A,B\in \mathcal{A}_k$ such that $A\sim B$, 
	take and then fix $S(A),S(B)\in \mathcal{A}_0$  so that $A\subset S(A)$ and $B\subset S(B)$.
	Summing \eqref{e:5.7} over all possible $A,B\in \mathcal{A}_k$ with $ A\sim B$, we have 
	\begin{align*}
	&\sum_{Q\in \mcQ_n(F_n)}I_{m+k}^{(m)}[g\circ \Psi_Q|_{\partial_o  F_m}]\\
	=&\sum_{A,B\in \mathcal{A}_k:A\sim B}\big([g]_{\nu_{m+n}|_A}-[g]_{\nu_{m+n}|_{B}}\big)^2  \\
	 \leq &\sum_{A,B\in \mathcal{A}_k\atop A\sim B}
		C_4 L_F^{-2k}\cdot\Big(\sum_{T,T'\in {\mathcal U}_{S(A)} \atop T\sim T'} \big([h]_{\nu|_{T\cap F}}-[h]_{\nu|_{T'\cap F}}\big)^2+\sum_{T,T'\in \mathcal{U}_{S(B)}\atop T\sim T'}\big([h]_{\nu|_{T\cap F}}-[h]_{\nu|_{T'\cap F}}\big)^2\Big)\\
    \leq& 2C_4L_F^{-2k}\sum_{S\in \mathcal{A}_0}\sum_{A\in \mathcal{A}_k\atop A\subset S}\sum_{B\in \mathcal{A}_k\atop  B\sim A}\sum_{T,T'\in {\mathcal U}_S\atop T\sim T'} \big([h]_{\nu|_{T\cap F}}-[h]_{\nu|_{T'\cap F}}\big)^2 \\
	\leq &C_5L_F^{-2k}\sum_{S\in \mathcal{A}_0}\sum_{A\in \mathcal{A}_k\atop A\subset S}\sum_{T,T'\in {\mathcal U}_S\atop T\sim T'} \big([h]_{\nu|_{T\cap F}}-[h]_{\nu|_{T'\cap F}}\big)^2\\
	=&C_5L_F^{-2k}\,L_F^{(d-1)k}\, \sum_{S\in \mathcal{A}_0}\sum_{T,T'\in {\mathcal U}_S\atop T\sim T'} \big([h]_{\nu|_{T\cap F}}-[h]_{\nu|_{T'\cap F}}\big)^2\\=&C_5L_F^{(d-3)k}\sum_{T,T'\in \mathcal{A}_0\atop T\sim T'}\sum_{S\in \mathcal{A}_0\atop T,T'\in {\mathcal U}_S}\big([h]_{\nu|_{T\cap F}}-[h]_{\nu|_{T'\cap F}}\big)^2\\
	\leq&  C_6L_F^{(d-3)k}\sum\limits_{Q\in \mcQ_n(F_n)}I_m[(h\circ \Psi_Q)|_{\partial_o  F}]
	\end{align*}
	for some $C_5,C_6$ depending only on $d$, where the third inequality is due to the fact that for each $A\in \mathcal{A}_k$	there are at most $ 3^d +L_F^{d-1}$ number of $B\in \mathcal{A}_k$ with $B\sim A$. Hence 
	\begin{eqnarray*}
	&& \sum_{Q\in \mcQ_n(F_n)} \Lambda_{m}^{(m)}[( g\circ \Psi_Q )|_{\partial_o  F_m}]\\
	&=& \sum_{Q\in \mcQ_n(F_n)}  \sum_{k=0}^\infty \varphi_m (L_F^{-(m+k)})  I_{m+k}^{(m)}[ (g\circ \Psi_Q ) |_{\partial_o  F_m}]\\
	&\leq &   \sum_{k=0}^\infty   L_F^{(d_w-d_f + d-2)m} L_F^{-(d-2)(m+k)} \, C_6 \, L_F^{(d-3)k} \sum_{Q\in \mcQ_n(F_n)}   
	I_m[(h\circ \Psi_Q)|_{\partial_o  F}] \\
	&=& \frac{C_6}{1-(1/L_F)}   \sum_{Q\in \mcQ_n(F_n)}    L_F^{(d_w-d_f)m}  I_m[(h\circ \Psi_Q)|_{\partial_o  F}] \\
	&\leq & \frac{C_6}{1-(1/L_F)} \sum\limits_{Q\in \mcQ_n(F_n)}\Lambda_m[(h\circ \Psi_Q)|_{\partial_o  F}].
	\end{eqnarray*}
	This establishes part (ii) of the Lemma. 
\end{proof}

\medskip

\begin{proof}[Proof  of Theorems \ref{T:1.1}, \ref{T:1.3}  and \ref{T:1.4}]
 Let $h$ be the unique function that is $\sE^{(F)}$-harmonic in $F\setminus (\partial_{1,0}F\cup \partial_{1,1}F)$
	having $h|_{\partial_{1,0}F}=0$ and $h|_{\partial_{1,1}F}=1$. For each $m,n\geq 0$, we construct a function $g_{m,n}$ using Lemma \ref{lemma51}. Recall that $\{m_k,k\geq 1\}$ is the subsequence in \eqref{e:5.1a}, and without loss of generality, we assume $m_1=0$. For 
	any integer $l\geq 0$, define 
	\[
	k(l):=\max\{k\geq 1:m_k\leq l/2\}
	\] 
	and set 
	$$
	       g_l :=g_{m_{k(l)},l-m_{k(l)}}   .
	$$
    We claim that $g_l\rightarrowtail h$  as $l\to \infty$. 
	For this, we need to show $\lim_{l\to \infty} g_l(y_l) = h(x)$   for any $x\in F$ and $y_l\in F_l$ so that $y_l\to x$ as $l\to \infty$. 
	Fix such $x\in F$ and $y_l\in F_l$. For $l\geq 0$, let $x_l\in F$ be such that $x_l$ and $y_l$  belong to the 
	same $Q(l)\in \mathcal{Q}_{l-m_k(l)}(F_l)=\mathcal{Q}_{l-m_k(l)}(F)$. By Lemma \ref{lemma51}(iii),
	$$
	|g_l(y_l)-h(x_l)|\leq C'_1\max_{x',y'\in F_{\mathcal{S}_{Q(l)}}}|h(x')-h(y')|.
	$$
	Since $l-m_k(l)\geq l/2$ and $h$ is uniform continuous on $F$, $ \lim_{l\to \infty} |g_l(y_l)-h(x_l)|= 0$.
	In addition, $\lim_{l\to \infty} h(x_l)=h(x)$ because $\rho(x_l,x)\leq \rho(x_l,y_l)+\rho(y_l,x)\to 0$ as $l\to\infty$. 
	Hence we have $\lim_{l\to \infty} g_l(y_l)= h(x)$. This proves the claim that $g_l\rightarrowtail h$ as $l\to \infty$. 
	
	We next estimate $\mcE^{(F_l)}(g_l)$. For each $n,n'\geq 0$, define $D_{n,n'}:=\bigcup_{Q\in \mathcal {Q}_{n'} (F_{n'})} \Psi_Q (F_{\mcB_{n}(F)})$. 	By Lemma \ref{lemmaA2} and by applying Theorem \ref{thmB1} on each $Q\in \mcQ_{n'} (F_{n'})\setminus \mcB_{n'} (F)$,  we have
	\begin{equation}\label{eqn53}
	\begin{aligned}
	\mu^{(F)}_{\<h\>}(D_{n,n'})&=\mu^{(F)}_{\<h\>}(F_{\mcB_{n'}(F)})+\sum_{Q\in \mcQ_{n'}(F)\setminus\mcB_{n'}(F)}\mu^{(F)}_{\<h\>}\big(\Psi_Q (F_{\mcB_{n}(F)})\big)\\
	&\leq \mu^{(F)}_{\<h\>}(F_{\mcB_{n'}(F)})+C_1e^{-c_1n}\sum_{Q\in \mcQ_{n'}(F)\setminus\mcB_{n'}(F)}\mu^{(F)}_{\<h\>}(F_{\mathcal{S}_Q})\\
	&\leq \mu^{(F)}_{\<h\>}(F_{\mcB_{n'}(F)})+3^dC_1e^{-c_1n}\mcE^{(F)}(h),
	\end{aligned}
    \end{equation}
 where $C_1$ and $c_1$ are the positive constants from Theorem \ref{thmB1} that depend only on $F$. In the last line we used the fact that each $Q'\in \mathcal{Q}_{n'}(F)$, $F_{Q'}$ is covered by at most $3^d$ number of $F_{S_{Q}}$ with $Q\in \mathcal{Q}_{n'}(F)$ as well as Corollary \ref{coroA5}. Note that (A1),(A2) and (A3) are satisfied for $(\bar \sE^{(k)}, \bar \sF^{(k)})=(\alpha\sE^{(m_k)}, W^{1,2}(F_{m_k}))$ by Lemma \ref{lemma41}, \eqref{e:5.1a} and 
the first paragraph of this section, and Lemma \ref{lemma34}, respectively. Note also that for non-negative $a, b$ with $ c:=|a-b|$ and any $\eps \in (0, 1)$,
	\begin{equation}\label{e:5.5}
	a^2 \leq (b+c)^2 \leq (1+\eps)b^2 + (1+\eps^{-1}) c^2.
	\end{equation}
	Fix $n\geq 1$ and $\varepsilon \in (0, 1)$. Since $ \lim_{l\to \infty} k(l) = \infty$, 
	applying Proposition \ref{prop41} locally on $F_Q$ for each $Q\in \mathcal{Q}_{l-m_{k(l)}}(F)$ with sufficiently large $l$, we have for some constants $C_2,C_3\geq 1$ depending only on $F$ that
	\begin{eqnarray*}
	&&\limsup\limits_{l\to\infty}\Big(\alpha \mcE^{(F_l)}(g_l)-(1+\varepsilon)\mcE^{(F)}(h)\Big)\\
	&= &\limsup\limits_{l\to\infty}L_F^{(d_w-d_f)(l-m_{k(l)})}\cdot\sum_{ Q\in \mathcal{Q}_{l-m_{k(l)}}(F)}
	\Big(\alpha\mcE^{(F_{m_{k(l)}})}(g_l\circ\Psi_Q)-(1+\varepsilon)\mcE^{(F)}(h\circ\Psi_Q)\Big)\\
	&\leq & \limsup\limits_{l\to\infty}L_F^{(d_w-d_f)(l-m_{k(l)})}\cdot\sum_{ Q\in \mathcal{Q}_{l-m_{k(l)}}(F)}
	 \left(1+\frac1{\eps}\right) \left| \sqrt{\alpha\mcE^{(F_{m_{k(l)}})}(g_l\circ\Psi_Q)}-\sqrt{\mcE^{(F)}(h\circ\Psi_Q)} \right|^{2} 
	\\
	&\leq&\limsup\limits_{l\to\infty}L_F^{(d_w-d_f)(l-m_{k(l)})}\cdot\sum_{ Q\in \mathcal{Q}_{l-m_{k(l)}}(F)}\frac{C_2}{\varepsilon}
	\big(\mu^{(F)}_{\< h\circ\Psi_Q\>}(F_{\mcB_{n-1}})+\Lambda^{(m_{k(l)})}_n[(g_l\circ \Psi_Q)|_{\partial_o  F_{m_k(l)}}]\big)\\
	&\leq&\limsup\limits_{l\to\infty}L_F^{(d_w-d_f)(l-m_{k(l)})}\cdot\sum_{ Q\in \mathcal{Q}_{l-m_{k(l)}}(F)}\frac{C_3}{\varepsilon}\mu^{(F)}_{\< h\circ\Psi_Q\>}(F_{\mcB_{n-1}})\\
	&=&\limsup\limits_{l\to\infty}\frac{C_3}{\varepsilon}\mu^{(F)}_{\<h\>}(D_{n-1,l-m_{k(l)}})\\
	&\leq &\frac{C_3}{\varepsilon}\,{ 3^d}C_1e^{-c_1(n-1)}\mcE^{(F)}(h),
    \end{eqnarray*}	
	where we used self-similarity in the first equality, \eqref{e:5.5} in the first inequality, 
	  Proposition \ref{prop41} in the second inequality, Lemma \ref{lemma51}(ii)  and Theorem \ref{T:A.3}(b) in the third inequality, self-similarity and Corollary \ref{coroA5} in the second equality, and 
	  \eqref{eqn53} and the fact that $\lim\limits_{l\to\infty}\mu_{\<h\>}^{(F)}(F_{\mcB_{l-m_{k(l)}}})=\mu_{\<h\>}^{(F)}(\partial_o F)=0$ from Corollary \ref{coroA5} in the last inequality. 
		By taking $n\to\infty$, we get  
		$
		\limsup\limits_{l\to\infty}\alpha\mcE^{(F_l)}(g_l)\leq (1+\varepsilon)\mcE^{(F)}(h)$.
		Letting $\varepsilon\to 0$ yields 
		$$\limsup\limits_{l\to\infty}\alpha\mcE^{(F_l)}(g_l)\leq \mcE^{(F)}(h).
		$$
		On the other hand,
		since $g_l\rightarrowtail h$, $g_l \in C(F_l)$ and $h\in C(F)$, 
		by Lemma \ref{lemma29}, $g_l\to h$ strongly in $L^2$ as $l\to \infty$.
		Hence  by the Mosco convergence (M1) of $\alpha_m \sE^{(F_m)}$ to $\sE^{(F)}$,  
		$$
		\liminf_{l\to \infty} \alpha \mcE^{(F_l)}(g_l) \geq \liminf_{l\to \infty} \alpha_l \mcE^{(F_l)}(g_l) \geq \mcE^{(F)}(h).
		$$
		Thus we have 
	\begin{equation}\label{eqn55}
	\lim_{l\to\infty}\alpha\mcE^{(F_l)}(g_l)=\mcE^{(F)}(h).
	\end{equation}

	 Now we show that  $\lim\limits_{m\to\infty}\alpha_m=\alpha := \limsup\limits_{m\to\infty}\alpha_m$. Suppose not, 
	 then there is a subsequence $\{l_k; k\geq 1\}$ such that $\lim\limits_{k\to\infty}\alpha_{l_k}=\delta<\alpha$. 
	By the Mosco convergence (M1) of $\alpha_m \sE^{(F_m)}$ to $\sE^{(F)}$,
	$$
	\delta\lim_{k\to\infty}\mcE^{(F_{l_k})}(g_{l_k})=\lim_{k\to \infty} \alpha_{l_k}\mcE^{(F_{l_k})}(g_{l_k})\geq \mcE^{(F)}(h).
	$$
	This contradicts \eqref{eqn55}. So we must have  
	\begin{equation}\label{e:5.12} 
	\lim\limits_{m\to\infty}\alpha_m=\alpha .
	\end{equation} 
	This together with Theorems \ref{T:3.11} and \ref{T:3.12} and \eqref{e:3.11a}  establishes Theorems \ref{T:1.1} and \ref{T:1.3}.
	
	\medskip 
	
	Finally, we show the convergence of the effective resistances. First, since $g_l \rightarrowtail h$ as $l\to \infty$, 
	\eqref{eqn55} implies that 
	\[
	\liminf\limits_{n\to\infty}R_n\geq\lim\limits_{l\to\infty}\big(\mcE^{(F_l)}(g_l)\big)^{-1}=\alpha\big(\mcE^{(F)}(h)\big)^{-1}.
	\]
	To see the other direction, let $h_l\in C(F_l)\cap W^{1,2}(F_l)$ be the unique function such that $h_l|_{\partial_{1,0}F_l}=0$, $h_l|_{\partial_{1,1}F_l}=1$ and $h_l$ is $\mcE^{{(F_l)}}$-harmonic in $F_l\setminus (\partial_{1,0}F_l\cup \partial_{1,1}F_l)$. Then, by the same proof of Lemma \ref{lemma41}, for each subsequence $l_k,k\geq 1$ there is a further 
	subsequence $\{ l'_k,k\geq 1\}$ and $h'\in C(F)$ such that $h_{l'_k}\rightarrowtail h'$. As $h'|_{\partial_{1,0}F}=0$ and $h'|_{\partial_{1,1}F}=1$, by the Mosco convergence (M1),
	\[
	\limsup\limits_{k\to\infty}R_{l'_k}=\big(\liminf\limits_{k\to\infty}\mcE^{(F_{l'_k})}(h_{l'_k})\big)^{-1}\leq\alpha \big(\mcE^{(F)}(h')\big)^{-1}\leq\alpha/ \mcE^{(F)}(h).
	\]  
	Since the argument works for each subsequence, we have $\limsup_{n\to \infty} R_n \leq \alpha /\mcE^{(F)}(h)$ and so $\lim\limits_{n\to\infty}R_n=\alpha/\mcE^{(F)}(h)$.
	This proves Theorem \ref{T:1.4}.
\end{proof}

\appendix
\renewcommand{\appendixname}{Appendix~\Alph{section}}

\section{Trace theorems and energy measures on the boundary}\label{secA}

We prove some trace theorems in this appendix, based on the Poincar\'e inequalities in Lemma \ref{lemma39}, the estimate $d_f-d_w<d_I$ 
from Lemma \ref{lemma38}, and the capacity estimates in Lemma \ref{lemmaHK}. 
For simplicity, the theorems are stated for continuous functions. 
 The proof of the restriction part,  Theorem \ref{T:A.3},  is based on the method in \cite{CQ1}, while the extension part,
Theorem \ref{T:A.4},  is proved by adapting the approach in \cite{HK} where part of the outer boundary $\partial_{1,0}F$ is considered as oppose to the whole outer boundary $\partial_o F$ in this paper. 
 We also remark that the study of trace theorems on the boundary of fractals goes back to \cite{Jonsson}.
 
 \medskip

We begin with a lemma  on the connectedness  of certain subdomains of $F$ and $F_m$. 

\begin{lemma}\label{lemmaA1}
	Suppose that $I\subset \{(i,s):i\in \{1,2,\cdots,d\},s\in \{0,1\}\}$ and $j=1,2,\cdots$. Set   
	\[
	\begin{aligned}
		K&=F\setminus \bigcup_{(i,s)\in I}  \left\{x\in F:\rho(x,\partial_{i,s}F_0)< L_F^{-j}/2  \right\},\\
		K_m&=F_m\setminus \bigcup_{(i,s)\in I} \left\{x\in F_m:\rho(x,\partial_{i,s}F_0)< L_F^{-j}/2  \right\}
		\quad \hbox{for }  m\geq 0.
	\end{aligned}
	\]
	Then $K$ and $K_m,m\geq 0$ are pathwise connected. 
\end{lemma}

\begin{proof}
	We prove only for the case that $I=\{(1,0)\}$. The general case follows by iterating the same argument. Let $x,y\in K\subset F$. Since $F$ is pathwise connected, we can find a continuous path $\gamma:[0,1]\to F$ such that $\gamma(0)=x,\gamma(1)=y$. 
	We let $\bar{\gamma}:[0,1]\to K$ be defined as 
	\[
	\bar{\gamma}(t)=
	\begin{cases}
		\gamma(t),&\text{ if }\gamma(t)\in K,\\
		\Gamma\circ\gamma(t),&\text{ if }\gamma(t)\notin K,
	\end{cases}
	\]
	where $\Gamma:\R^d\to\R^d$ is the reflection map $\Gamma(z_1,z_2,\cdots,z_d)=(L_F^{-j}-z_1,z_2,\cdots,z_d)$ with respect to the hyperplane $x_1=L_F^{-j}/2$.
	Note that both $F_m$ and $F$ are symmetric with respect to $x_1=1/2$ so for each $Q\in {\mathcal Q}_j$, 
	$F_Q$ and $F_{m, Q}$ are symmetric with respect to the hyperplane passing through the center $Q$ that is parallel to $x_1=0$. Thus $\bar{\gamma}$ is a continuous path in $K$ connecting $x,y$. Hence, $K$ is path connected. The same argument shows that $K_m$ is path connected for each $m\geq 0$. 
\end{proof}

The next lemma considers the energy measure, where the second statement can be improved by replacing `$\geq$' with `$=$' by using Corollary \ref{coroA5}. 

\begin{lemma}\label{lemmaA2}
For  each $f\in \mcF^{(F)}$, $n\geq 1$ and Borel $A\subset F$, 
	$$
	\mu^{(F)}_{\<f\>}(A)=\sum\limits_{Q\in \mcQ_n(F)}\mu^{(F_Q)}_{\<f\>}(A\cap F_Q).
		$$
		In particular, if $A \subset F_Q$ for some $Q\in 
	  \cup_{n=1}^\infty \mcQ_n(F)$, then $\mu^{(F)}_{\<f\>}(A)\geq \mu^{(F_Q)}_{\<f\>}(A)$.
\end{lemma}

\begin{proof}
By  the inner regular property of Radon measures, 
	it suffices to prove the lemma for compact subset $A\subset F$.   
	Moreover, by the the regularity of the Dirichlet form $(\mcE^{(F)},\mcF^{(F)})$ on $L^2(F; \mu)$, 
	it sufficies to consider $f\in \mcF^{(F)}\cap C(F)$. For each $m\geq 1$, let $g_m\in \mcF^{(F)}\cap C(F)$ such that $0\leq g_m\leq 1$, $g_m|_A=1$ and $g_m(x)=0$ for each $x\in F$ satisfying $\rho(x,A)\geq \frac{1}{m}$. Then, by using the self-similar property of 
	$(\mcE^{(F)},\mcF^{(F)})$ from Lemma \ref{lemma31}, we see
	\[
	\begin{aligned}
		\mu^{(F)}_{\<f\>}(A)&=\lim_{m\to\infty}\int_F g_m(x)\mu^{(F)}_{\<f\>}(dx)\\
		&=\lim_{m\to\infty}\Big(\mcE^{(F)}(g_mf,f)-\frac12\mcE^{(F)}(g_m^2,f)\Big)\\
		&=\lim_{m\to\infty}\sum_{Q\in \mcQ_n(F)}\Big(\mcE^{(F_{Q})}(g_mf,f)-\frac12\mcE^{(F_{Q})}(g_m^2,f)\Big)\\
		&=\lim_{m\to\infty}\sum_{Q\in \mcQ_n(F)}\int_{F_{Q}} g_m(x)\mu^{(F_{Q})}_{\<f\>}(dx)=\sum\limits_{Q\in \mcQ_n(F)}\mu^{(F_Q)}_{\<f\>}(A\cap F_Q).
	\end{aligned}
	\]
\end{proof}

Recall that $\mcB_{k}(F)$ and $\mcB_{k} (F_m)$ are the $k$-level boundary shells of $F$ and $F_m$ as defined in Section \ref{sec4}. For a subset $\mathcal A$ of ${\mathcal Q}_n$, $F_{\mathcal A}:= \cup_{Q\in \mathcal A} F_Q$   and $F_{m,\mathcal{A}}:=\cup_{Q\in \mathcal{A}}F_{m,Q}$.

\begin{theorem}\label{T:A.3}
	There is a constant $C>0$ depending on $F$ only such that the following hold.
	\begin{enumerate}
		\item[\rm (a)] $\Lambda_n[f|_{\partial_o F}]\leq C\, \mu_{\< f\>}^{(F)}(F_{\mcB_{n-1}(F)})$ for each $f\in \mcF^{(F)}\cap C(F)$ and $n\geq 1$. 
		
		\item[\rm (b)] $\Lambda^{(m)}_n[f|_{\partial_o F_m}]\leq C\, \mu_{\< f\>}^{(F_m)}(F_{m,\mcB_{n-1} (F_m)} )$ for each $m\geq 0$, $f\in W^{1,2}(F_m)\cap C(F_m)$ and $n\geq 1$.
	\end{enumerate}
\end{theorem}

\begin{proof}
We will only give a proof for (b) as  the proof for (a) is very similar.
	We introduce a few more notations. For $k\geq 1$ and $A\in \eth_kF_m$, 
	let $Q_A\in \mcB_k(F_m)$ be such that $A\subset Q_A$, and define 
	$$
    U_{k, m, A} 
	:=F_{m,Q_A}\setminus F_{m,\mcB_{k+1}(F_m)},
	$$
	and 
	\begin{equation}\label{e:A.1} 
	u^{(m)}_0(f,A):= [f]_{\mu_m|_{U_{k, m,A}} = \fint_{U_{k, m, A}} f (x) \mu_m (dx)}.
	\end{equation} 
	Define for $i\geq 0$,  
	$$
	\eth_{k+i}F_m(A):=\{B\in \eth_{k+i}F_m:\ B\subset A\}
	$$
	and
	\begin{equation}\label{e:A.2} 
	u^{(m)}_i(f,A):=\frac{\sum_{B\in \eth_{k+i}F_m(A)}[f]_{\mu_m|_{U_{k+i, m,B}}}}{\#\eth_{k+i}F_m(A)}. 
	\end{equation} 
	Observe that for $A\in  \eth_{k}F_m$, 
	\begin{equation} \label{e:A.3} 
	\#\eth_{k+i}F_m(A)=m_I^{i\wedge((m-k)\vee 0)}\,  L_F^{(d-1)\left(i-i\wedge((m-k)\vee 0)\right)}.
	\end{equation}

	By the continuity of $f$, we have
	\[
	[f]_{\nu_m|_A}=\lim\limits_{i\to\infty}u^{(m)}_i(f,A)=u^{(m)}_0(f,A)+\sum_{i=1}^\infty\big(u^{(m)}_i(f,A)-u^{(m)}_{i-1}(f,A)\big),
	\]
	Define  
	\[
	D_{k,i}^{(m)}[f]=
	\begin{cases}
		\sum\limits_{A,B\in \eth_kF_m\atop A\sim B}\big(u^{(m)}_i(f,A)-u^{(m)}_i(f,B)\big)^2  &\text{if }i=0,\\
		\sum\limits_{A,B\in \eth_kF_m\atop A\sim B}\big(u^{(m)}_i(f,A)-u^{(m)}_{i-1}(f,A)-u^{(m)}_i(f,B)+u^{(m)}_{i-1}(f,B)\big)^2
		&\text{if }i\geq 1.
	\end{cases}
	\]
    By the triangle inequality for the $l^2$-norm, 
	\begin{equation}\label{eqnA1}
		\sqrt{I_k^{(m)}[f|_{\partial_oF_m}]}\leq \sum_{i=0}^\infty\sqrt{D_{k,i}^{(m)}[f]}.
	\end{equation}
	We next estimate $D^{(m)}_{k,i}[f]$.\\
	
	\noindent\textbf{Claim 1}. Let $k\geq 1$  and $A,A'\in \eth_kF_m$  with $A\cap A'\neq \emptyset$.
	Let   $Q,Q'\in \mathcal{B}_k(F_m)$ be such that $A\subset Q,A'\subset Q'$. Define 
	\[
	\begin{aligned}
		B'&=(F_{m,Q}\cup F_{m,Q'})\setminus  \big\{x\in F_m: \rho(x,\partial_oF_0)< L_F^{-k-1} /2  \big\},\\
		B&= \big\{x\in F_m: \rho(x,B')< L_F^{-k-2}/2 \big\}.
	\end{aligned}
	\]
	 Clearly, 
	 $$ 
	 B\subset \bigcup_{Q^*\in \mcB_{k-1}(F_m)\atop Q^*\cap (A\cup A')\neq\emptyset}F_{m,Q^*} 
	  \setminus F_{m,\mcB_{k+2}(F_m)}\subset F_{m,\mcB_{k-1}(F_m)}\setminus F_{m,\mcB_{k+2}(F_m)} . 
	 $$ 
	      Then, for some $C_1$ depending only on $F$, we claim that 
	\begin{equation}\label{e:A.5} 
	\varphi_m(L_F^{-k})\,  \big(u^{(m)}_0(f,A)-u^{(m)}_0(f,A')\big)^2\leq C_1\mu_{\< f\>}^{(F_m)}(B).
	\end{equation}

	First, note that $B'$ is connected by Lemma \ref{lemmaA1}. Indeed,
	$F_{m,Q}\setminus \{x\in F_{m,Q}: \rho(x,\partial_oF_0)<L_F^{-k-1}/2\}$ is a scaled version of $F_{(m-k)\vee0}$ in Lemma \ref{lemmaA1} (for some $I$ and $j=1$) and so is $F_{m,Q'}\setminus \{x\in F_{m,Q'}: \rho(x,\partial_oF_0)< L_F^{-k-1}/2\}$, and these two connected sets do intersect as $A\cap A' \not= \emptyset$.
	
	Next, choose a $(c/8) L_F^{-k-2}$-net $\{x^{(i)}\}_{i=1}^N$ of $B'$, where $c$ is the constant in the Poincar\'e inequalities (Lemma \ref{lemma39} (b)), and $N\geq 2$ is an integer depending only on $d$ and $c$. Let 
	 $B_i:=B_{F_m}(x^{(i)},L_F^{-k-2}/2)$ and  $B_i'=B_{F_m}(x^{(i)},cL_F^{-k-2}/2)$ for each $1\leq i\leq N$. Recall the definition of $\varphi_m$ from \eqref{eqn35} and note that 
	 \begin{equation}\label{e:A.6}
	  \varphi_m(L_F^{-k-2} /2)  \leq  \varphi_m(L_F^{-k})\leq (2L_F^2)^{(d-2)\vee (d_f-d_w)}\varphi_m(L_F^{-k-2}/2) . 
	 \end{equation} 
	By the Poincar\'e inequality from Lemma \ref{lemma39}(b),  there is a constant $C_1'>0$ depending only on $F$
	 so that for $1\leq i\leq N$,  
	 \begin{equation}\label{eqnA2}
	\varphi_m(L_F^{-k})\fint_{B_i'}\big(f(x)-[f]_{\mu_m|_{B_i'}}\big)^2\mu_m(dx) \leq  C'_1\mu^{(F_m)}_{\<f\>}(B_i) . 
	\end{equation} 
For each $1\leq i,j\leq N$ such that $\rho(x^{(i)},x^{(j)})<cL_F^{-k-2}/4$,
 $$ 
 B'_i\cap B'_j   \supset  B_{F_m} \big(x^{(i)},cL_F^{-k-2}/4 \big)   \cup   B_{F_m} \big( x^{(j)},cL_F^{-k-2}/4 \big)   .
$$
Moreover, there is a constant $ C_2' >0$  depending only on $F$ such that 
$$ 
m(B_{F_m}(x^{(i)},cL_F^{-k-2}/4)) \geq C_2' \mu_m(B'_i)  \quad \hbox{and} \quad 
m(B_{F_m}(x^{(j)},cL_F^{-k-2}/4)) \geq C_2' \mu_m(B'_j).
$$
Thus by  \eqref{eqnA2}, 
\begin{eqnarray*}
&&  \varphi_m(  L_F^{-k} )^{1/2}  \big|[f]_{\mu_m|_{B'_i}}-[f]_{\mu_m|_{B_j'}}  \big|   \nonumber \\
&\leq & \frac{\varphi_m(  L_F^{-k} )^{1/2}}{\mu_m (B_i' \cap B_j')} \int_{B_i'\cap B_j'} \left( \Big| f(x) - [f]_{\mu_m|_{B'_i}} \Big|  + 
\Big| f(x) - [f]_{\mu_m|_{B_j'}} \Big|  \right) \mu_m(dx)  \nonumber \\
&\leq & \frac{\varphi_m(  L_F^{-k} )^{1/2}}{C_2' \mu_m(B_i')} \int_{B_i' }   \Big| f(x) - [f]_{\mu_m|_{B'_i}} \Big|  \mu_m(dx) 
 + \frac{\varphi_m(  L_F^{-k} )^{1/2}}{C_2' \mu_m(B_j')} \int_{B_j' }   \Big| f(x) - [f]_{\mu_m|_{B'_i}} \Big|\mu_m(dx)  \nonumber \\
 &\leq & C_3'\sqrt{\mu^{(F_m)}_{\<f\>}(B_i\cup B_j)} \nonumber \\ 
  &\leq&  C_3'\sqrt{\mu^{(F_m)}_{\<f\>}(B)} 
  \end{eqnarray*}	
 for some $C'_3$ depending only on $F$.	 Next, noticing that $B'$ is connected and $\{x^{(i)}\}_{i=1}^N$ is a $cL_F^{-k-3}/8$ net of $B'$, we have by the triangle  inequality that 
	\begin{equation} \label{e:A.8}
\varphi_m(  L_F^{-k} )^{1/2} \, 
\big|[f]_{\mu_m|_{B'_1}}-[f]_{\mu_m|_{B_i'}}\big|\leq C'_3 N\sqrt{\mu^{(F_m)}_{\<f\>}(B)} \quad 
\hbox{for every } 1\leq i\leq N.
	\end{equation}
  It follows from \eqref{eqnA2} and \eqref{e:A.8} that 
	\begin{eqnarray}
	&&    \Big( \varphi_m(L_F^{-k})  \fint_{B'}\big(f(x)-[f]_{\mu_m|_{B'}}\big)^2\mu_m(dx)   \Big)^{1/2}   \nonumber \\  
	&\leq &   \Big(  \varphi_m(L_F^{-k})  \fint_{B'}\big(f(x)-[f]_{\mu_m|_{B'_1}}\big)^2\mu_m(dx)  \Big)^{1/2}  \nonumber \\
	&\leq &  \varphi_m(L_F^{-k})^{1/2}  \Big(    \sum_{i=1}^N  \frac{\mu_m(B_i')}{\mu_m (B')}  \fint_{B_i'}\big(f(x)-[f]_{\mu_m|_{B'_1}}\big)^2\mu_m(dx)  	\Big)^{1/2}   \nonumber \\ 
	&\leq &  C_4'     \varphi_m(L_F^{-k})^{1/2}    \sum_{i=1}^N    
	\left( \Big(   \fint_{B_i'}\big(f(x)-[f]_{\mu_m|_{B'_i}}\big)^2\mu_m(dx)   \Big)^{1/2} 
	     +  \big| [f]_{\mu_m|_{B'_1}}-[f]_{\mu_m|_{B_i'}} \big|  \right)   \nonumber \\ 
	&\leq & C_4'   \sum_{i=1}^N   \left(   \sqrt{ C'_1 \mu^{(F_m)}_{\<f\>}(B_i) } + C_3'N \sqrt{\mu^{(F_m)}_{\<f\>}(B) } \right) \nonumber \\
	  &\leq &  C'_5 \sqrt{\mu^{(F_m)}_{\<f\>}(B)},    \label{e:A.9} 
	\end{eqnarray}
	where the constants $C_4'$ and $C'_5$ depend only on $F$. Recall that $u_0^{(m)}(f,A):=[f]_{\mu_m|_{U_{k,m,A}}}$. Notice that $B' \supset U_{k,m,A} \cup U_{k,m,A'}$, and there is a constant $C_6'\in (0,1)$ depending only on $F$ so that 
 	$\mu_m(U_{k,m,A})\geq C_6' \, \mu_m(B')$ and $\mu_m(U_{k,m,A'})\geq C_6'  \, \mu_m(B')$.
	 We have   by \eqref{e:A.9} that 
	\begin{eqnarray*}
	&& \varphi_m(L_F^{-k})\big(u_0^{(m)}(f,A)-u_0^{(m)}(f,A')\big)^2\\
	& \leq& 2\varphi_m(L_F^{-k})  \left( \left( [f]_{\mu_m|_{U_{k,m,A}}}-[f]_{\mu_m|_{B'}}\right)^2+ 
	\left( [f]_{\mu_m|_{B'}}-[f]_{\mu_m|_{U_{k,m,A'}}}\right)^2 \right) \\
	&\leq&2\varphi_m(L_F^{-k}) \left( \fint_{U_{k,m,A}}\big(f(x)-[f]_{\mu_m|_{B'}}\big)^2\mu_m(dx)+ \fint_{U_{k,m,A'}}\big(f(x)-[f]_{\mu_m|_{B'}}\big)^2\mu_m(dx) \right) \\
	&\leq&  \frac{4  \varphi_m(L_F^{-k})}{ C_6'}\,   \fint_{B'}\big(f(x)-[f]_{\mu_m|_{B'}}\big)^2\mu_m(dx)\\
	& \leq& \frac{4(C'_5)^2 }{ C_6'}\, \mu^{(F_m)}_{\<f\>}(B).
\end{eqnarray*}
    This proves the claim  \eqref{e:A.5}.

	\medskip
	
	\noindent\textbf{Claim 2}. There is a constant $C_2>0$ depending only on $F$  so that the following hold. 
	\begin{enumerate}
		\item[\rm (2.a)] Let $k\geq 1$ and  $A,A'\in \eth_kF_m$ that there is some   $Q\in\mcB_{k-1}(F_m)$ such that  $A\cup A'\subset Q$.
		Then 
		\[
		\varphi_m(L_F^{-k})\,  \left( u^{(m)}_0(f,A)-u^{(m)}_0(f,A')\right)^2
		    \leq C_2\, \mu_{\< f\>}^{(F_m)} \left( F_{m,Q}\setminus F_{m,\mcB_{k+2}(F_m)} \right).
		\]
		
		\item[\rm (2.b)]  Let $k\geq 1$ and $A\in \eth_kF_m$, $A'\in \eth_{k+1}F_m$ that there is some  $Q\in\mcB_k(F_m)$ such    that $A'\subset A\subset Q$.
		Then 
		\[
		\varphi_m(L_F^{-k-1})\,  \left( u^{(m)}_0(f,A)-u^{(m)}_0(f,A')\right)^2
		\leq C_2\,  \mu_{\< f\>}^{(F_m)} \left( F_{m,Q}\setminus F_{m,\mcB_{k+3}(F_m)} \right).
		\]
	\end{enumerate}
	
	\medskip
	
	Claim 2 follows from an argument similar to that for Claim 1, by applying the Poincar\'e inequalities locally on $\mcE^{(F_{m,Q})}$, which is essentially a rescaled version of $\mcE^{(F_{(m-k+1)\vee 0})}$ in (2.a) and $\mcE^{(F_{(m-k)\vee 0})}$ in (2.b). 
	
	\medskip
	
	\noindent\textbf{Remark}. For the proof of  Theorem \ref{T:A.3} (a), we can still use the Poincar\'e inequalities locally on cells to prove a corresponding version of Claim 2
	in view of Lemma \ref{lemmaA2} and the self-similarity \eqref{e:3.3} of the Dirichlet form $(\sE^{(F)}, \sF^{(F)})$.

	\medskip

	\noindent\textbf{Claim 3}. Let $k\geq 1$, $i\geq 1$, $A\in \eth_{k}F_m$ and $Q\in \mcB_k(F_m)$ such that $A\subset Q$. Then
	\begin{align*}
		&\varphi_m(L_F^{-k-i})\,  \left( u^{(m)}_i(f,A)-u^{(m)}_{i-1}(f,A)\right)^2\\
		\leq& C_2\,  (\#\eth_{k+i-1}F_m(A))^{-1}\, \mu_{\< f\>}^{(F_m)}\left(Q\cap (F_{m,\mcB_{k+i-1}(F_m)}\setminus F_{m,\mcB_{k+i+2}(F_m)})\right).
	\end{align*}
	
	\medskip
	
	When  $i=1$, Claim 3 is an immediately consequence of  part (2.b) of Claim 2. 
	Indeed,  as  $u^{(m)}_1(f,A)$ is the average of $u^{(m)}_0(f,A')$ over $ A'\in \eth_{k+1}F_m(A)$,  we have by (2.b) of Claim 2 that  
	\begin{eqnarray}
		&& \varphi_m(L_F^{-k-1})\,  \left(u^{(m)}_1(f,A)-u^{(m)}_0(f,A)\right)^2  \nonumber \\
		&\leq& \max_{A'\in \eth_{k+1}F_m(A)} \varphi_m(L_F^{-k-1})\, \left( u^{(m)}_0(f,A')-u^{(m)}_0(f,A)\right)^2 \nonumber \\
		&\leq& C_2\,  \mu_{\< f\>}^{(F_m)}  \left( F_{m,Q}\setminus F_{m,\mcB_{k+3}(F_m)} \right ) \nonumber \\
		&=&C_2\,  \mu_{\< f\>}^{(F_m)}\left(Q\cap (F_{m,\mcB_k(F_m)}\setminus F_{m,\mcB_{k+3}(F_m)})\right) .   \label{e:A.10} 
	\end{eqnarray}
	
	For $i\geq 2$,  by the definition of $u^{(m)}_i(f,A)$ in \eqref{e:A.1}-\eqref{e:A.2}, 
	$$
	u^{(m)}_{i-1} (f,A):=\frac{\sum_{A'\in \eth_{k+i-1}F_m(A)}[f]_{\mu_m|_{ U_{k+i-1, m,A'}}}}{\#\eth_{k+i-1}F_m(A)}
	= \frac{\sum_{A'\in \eth_{k+i-1}F_m(A)}u_0^{(m)}(f, A')}{\#\eth_{k+i-1}F_m(A)},
	$$
	and 
	\begin{eqnarray*}
	u^{(m)}_{i} (f,A)&:=&\frac{\sum_{B\in \eth_{k+i}F_m(A)}[f]_{\mu_m|_{ U_{k+i, m,B}}}}{\#\eth_{k+i}F_m(A)} \\
	&=& \frac{\sum_{A'\in \eth_{k+i-1}F_m(A)} \sum_{B\in \eth_{k+i}F_m(A')}u^{(m)}_0 (f,B)}{\#\eth_{k+i}F_m(A)} \\
	&=& \frac{\sum_{A'\in \eth_{k+i-1}F_m(A)}  u^{(m)}_1 (f,A')\,  (\# \eth_{k+i}F_m(A') )  }{\#\eth_{k+i}F_m(A)} \\
	&=& \frac{\sum_{A'\in \eth_{k+i-1}F_m(A)}    u^{(m)}_1 (f,A')}{\#\eth_{k+i-1}F_m(A)} .
 	\end{eqnarray*} 
	 Thus by the Cauchy-Schwarz inequality and \eqref{e:A.10},  
	\begin{eqnarray*}
		&&(\#\eth_{k+i-1}F_m(A))^2\cdot\varphi_m(L_F^{-k-i})\cdot \big(u^{(m)}_i(f,A)-u^{(m)}_{i-1}(f,A)\big)^2\\
		&=&\varphi_m(L_F^{-k-i})\Big(\sum_{A'\in \eth_{k+i-1}F_m(A)} \big(u^{(m)}_1(f,A')-u^{(m)}_{0}(f,A')\big)\Big)^2\\
		&\leq&(\#\eth_{k+i-1}F_m(A))\,   \sum_{A'\in \eth_{k+i-1}F_m(A)}\varphi_m(L_F^{-k-i})\,
		  \left(u^{(m)}_1(f,A')-u^{(m)}_{0}(f,A')\right)^2\\
		&\leq&(\#\eth_{k+i-1}F_m(A))\,  \sum_{A'\in \eth_{k+i-1}F_m(A)}C_2\, 
		  \mu^{(F_m)}_{\<f\>}\big(Q_{A'}\cap (F_{m,\mcB_{k+i-1}(F_m)}\setminus F_{m,\mcB_{k+i+2}(F_m)})\big)\\
		& =&  C_2\, (\eth_{k+i-1}F_m(A))\,  \mu_{\< f\>}^{(F_m)}(Q\cap (F_{m,\mcB_{k+i-1}(F_m)}\setminus F_{m,\mcB_{k+i+2}(F_m)})),
	\end{eqnarray*}
	where  in the second inequality $Q_{A'}\in \mcB_{k+i-1}(F_m)$ is such that $Q_{A'} \supset A'$. 
		
	\medskip 
	
	By Claim 1 and part (2.a) of Claim 2, there is a constant $C_3>0$ depending only on $F$ so that  
	\begin{eqnarray}
	\varphi_m(L_F^{-k})\,  D^{(m)}_{k,0}[f]
	&=& \varphi_m(L_F^{-k}) \sum_{A,B\in \eth_kF_m\atop A\sim B}\big(u^{(m)}_0 (f,A)-u^{(m)}_0 (f,B)\big)^2  \nonumber \\
	&\leq &  C_2 \sum_{A,B\in \eth_kF_m\atop A\sim B} \mu_{\< f\>}^{(F_m)} 
	 \Big((\bigcup_{Q\in \mcB_{k-1}(F_m)\atop Q\cap (A\cup B)\neq\emptyset}F_{m,Q})\setminus F_{m,\mcB_{k+2}(F_m)} \Big) \nonumber  \\
	&\leq & C_3\,  \mu^{(F_m)}_{\<f\>}(F_{m,\mcB_{k-1}(F_m)}\setminus F_{m,\mcB_{k+2}(F_m)}) , \label{e:A.11} 
	\end{eqnarray}  
	where in the last inequality, we used the fact that for each $Q\in\mcB_{k-1}(F_m)$,  there are at most $2(L_F+2)^{d-1}(3^d+L_F^{d-1})$ number of ordered $k$-cells $(A, B)$ from $\eth_k F_m  $  so that $A\sim B$ and $Q\cap (A\cup B)\neq\emptyset$.

	For $ m\geq 0$, $k\geq 1$ and $i\geq 0$,
	define 
	$$
	\alpha_{m, k, i}:=	m_I^{i\wedge((m-k)\vee 0)}\,  L_F^{ (d-1)  \left(i-(i\wedge(m-k))\vee 0\right)} . 
$$
Note that by \eqref{e:A.3}, $\alpha_{m, k, i}= \#\eth_{k+i}F_m(A)$ for any $A\in   \eth_{k}F_m $.
Moreover, $\alpha_{m, k, 0}=1$ and $\alpha_{m, k, i}$ is uniformly comparable to $\alpha_{m, k, i-1}$ for $i\geq 1$.
 Thus  it follows from  from Claim 3 and \eqref{e:A.11} that there is a constant $C_4>0$ depending on $F$ only so that for any
$m\geq 0$, $k\geq 1$ and $i\geq 0$, 
\begin{equation}\label{e:A.12}
	\varphi_m(L_F^{-k-i})\,  D^{(m)}_{k,i}[f]
	\leq \frac{C_4}{\alpha_{m, k, i} } \, \mu^{(F_m)}_{\<f\>}(F_{m,\mcB_{k+i-1}(F_m)}\setminus F_{m,\mcB_{k+i+2}(F_m)}) .
\end{equation} 
It follows from the definition of $\varphi_m$ in  \eqref{eqn35} that 
	\[
	\frac{\varphi_m(L_F^{-k})}{\varphi_m(L_F^{-k-i})} \, \frac1{\alpha_{m, k, i}}\leq \theta^i 
	\quad\hbox{for every } m\geq 0,  k\geq 1 \hbox{ and }    i\geq 0  ,
	\]
	where $\theta :=L_F^{d-2-(d-1)}\vee L_F^{d_f-d_w-d_I}<1$ by Lemma \ref{lemma38}. Hence, we have from \eqref{e:A.12} that 
	 for every  $m\geq 0$, $k\geq 1$   and $ i\geq 0$,
\begin{equation}\label{e:A.13}
 	\varphi_m(L_F^{-k})\,  D^{(m)}_{k,i}[f]\leq C_4\theta^i\,  \mu^{(F_m)}_{\<f\>}(F_{m,\mcB_{k+i-1}(F_m)}\setminus F_{m,\mcB_{k+i+2}(F_m)}).
\end{equation} 
Now for any $n\geq 1$ and $m\geq 0$,  we have by \eqref{eqnA1} and \eqref{e:A.13}, 

\begin{eqnarray*}
		\left( \Lambda^{(m)}_n[f]\right)^{1/2}
		 &=& 		\left( \sum_{j=0}^\infty\varphi_m(L_F^{-n-j})\,  I^{(m)}_{n+j}[f] \right)^{1/2}\\
		&=&\Big\|  \sqrt{ \varphi_m(L_F^{-n-j})\,  I^{(m)}_{n+j}[f]} \, \Big\|_{l^2(\hbox{\tiny in } j\geq 0)}\\
		&\leq& \sum_{i=0}^\infty \Big\|   \sqrt{\varphi_m(L_F^{-n-j})\, D^{(m)}_{n+j,i}[f]}  
		  \, \Big\|_{l^2(\hbox{\tiny in } j\geq 0)}\\
		&\leq & \sum_{i=0}^\infty (C_4 \theta^{i})^{1/2} 
		\Big\| \sqrt{\mu^{(F_m)}_{\<f\>}(F_{m,\mathcal{B}_{n+i+j-1}(F_m)}\setminus F_{m,\mathcal{B}_{n+i+j+2}(F_m)})} \, 
		   \Big\|_{l^2(\hbox{\tiny in } j\geq 0)}\\
        &\leq & \sum_{i=0}^\infty(C_4 \theta^{i})^{1/2}
		\Big\| \sqrt{\mu^{(F_m)}_{\<f\>}(F_{m,\mathcal{B}_{n+j-1}(F_m)}\setminus F_{m,\mathcal{B}_{n+j+2}(F_m)})} \, 
		   \Big\|_{l^2(\hbox{\tiny in } j\geq 0)}\\
 	 	&\leq &   \frac{\sqrt{C_4}}{1-\sqrt{\theta}}  \,
		\Big\|\sqrt{\mu^{(F_m)}_{\<f\>}(F_{m,\mathcal{B}_{n+j-1}(F_m)}\setminus F_{m,\mathcal{B}_{n+j+2}(F_m)})}  \, 
		 \Big\|_{l^2   (\hbox{\tiny in } j\geq 0) }  \\
		&\leq &  \frac{\sqrt{3C_4}}{1-\sqrt{\theta}}    \, \sqrt{  \mu^{(F_m)}_{\<f\>}(F_{m,\mcB_{n-1}(F_m)}) },
	\end{eqnarray*}
	where $\|a_j\|_{l^2   (\hbox{\tiny in } j\geq 0)}:=\sqrt{\sum_{j=0}^\infty a_j^2}$ denotes the $l^2$-norm of the sequence $\{a_j;j\geq 0\}$. 
	This proves Theorem \ref{T:A.3}(b).  
\end{proof}

\medskip
 
 Recall that the Besov spaces $\Lambda (\partial_o F)$ and  $\Lambda^{(m)} (\partial_o F_m)$ defined in \eqref{e:4.3} and \eqref{e:4.4a}, respectively.
 
 \medskip

\begin{theorem}\label{T:A.4}
	There is a constant $C>0$ depending only on $F$ such that the following hold.	
	\begin{enumerate}
		\item[\rm (a)]  There is an extension map $\Ex:\Lambda(\partial_o F)\cap C(\partial_o F)\to \mcF^{(F)}\cap C(F)$ such that $\Ex f|_{\partial_o F}=f$ and $\mcE^{(F)}(\Ex f)\leq C\Lambda_1[f]$ for any $f\in \Lambda(\partial_o F)\cap C(\partial_o F)$. Moreover, $\mu^{(F)}_{\<\Ex f\>}(\partial_o F)=0$ for any $f\in \Lambda(\partial_o F)\cap C(\partial_o F)$.
		
		\item[\rm (b)] For each $m\geq 0$, there is an extension map $\Ex_m:\Lambda^{(m)} (\partial_o F_m)\cap C(\partial_o F_m)\to \mcF^{(F_m)}\cap C(F_m)$ such that $\Ex_m f|_{\partial_o F_m}=f$ and $\mcE^{(F_m)}(\Ex_m f)\leq C\Lambda_1^{(m)}[f]$ for any 
		$f\in \Lambda^{(m)} (\partial_o F_m)\cap C(\partial_o F_m)$.
	\end{enumerate}
\end{theorem}
\begin{proof}
We will only present a proof for (a) as it has an additional statement. The  proof for (b) is the same. For $k\geq 1$ and $A\in \eth_kF$, let $Q_A\in \mcB_k(F)$ be such that $A\subset F_{Q_A}$, and  $U_A$ be the closure of $F_{Q_A}\setminus F_{\mcB_{k+1}(F)}$.  Define $\bar{u}_A\in C(F)\cap\mcF^{(F)}$ by
\[
\bar{u}_A(x):=\max_{Q\in \mathcal{Q}_{k+2}(U_A)}w_Q(x). \quad \hbox{for }  x\in F, 
\]
where   $w_Q \in C(F)\cap \sF^{(F)}$ is the non-negative function in   Lemma \ref{lemma37}(a). Define
$$
	\bar{u}_o = \begin{cases}
	\max_{Q\in \mcQ_2(  F\setminus F_{\mcB_1(F)})}w_Q \quad &\hbox{if } F\setminus F_{\mcB_1(F)}\neq \emptyset,  \cr 
	0 &\hbox{otherwise.}
	\end{cases}
$$
 Denote by  $\operatorname{supp}[f]$ the support of $f$. Note that for $A\in \eth_kF$  with $k\geq 1$,
\begin{equation} \label{e:A.14}
U_A\subset \operatorname{supp}[\bar{u}_A]\subset F_{S_{\mathcal{Q}_A}}\setminus F_{\mcB_{k+2(F)}}\subset F_{\mcB_{k-1}(F)}\setminus F_{\mcB_{k+2(F)}} 
 \end{equation}
 and
\begin{equation} \label{e:A.15}
F\setminus F_{\mcB_1(F)} \subset \operatorname{supp}[\bar{u}_o]\subset F\setminus F_{\mcB_2(F)}.
\end{equation}
See Figure \ref{figA1} for an illustration. 

\begin{figure}[htp]
	\includegraphics[width=4.5cm]{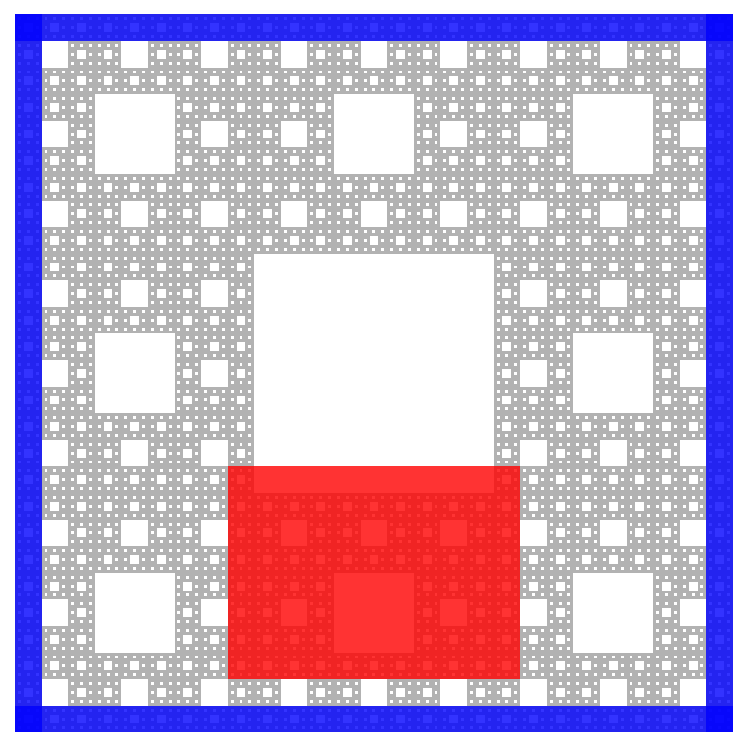}\qquad 
	\includegraphics[width=4.5cm]{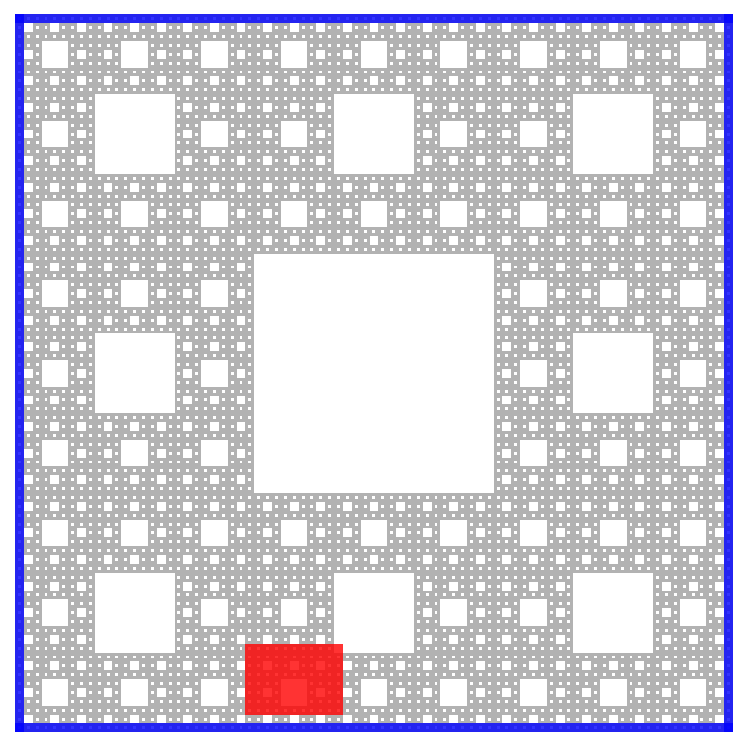}\qquad 
	\includegraphics[width=4.5cm]{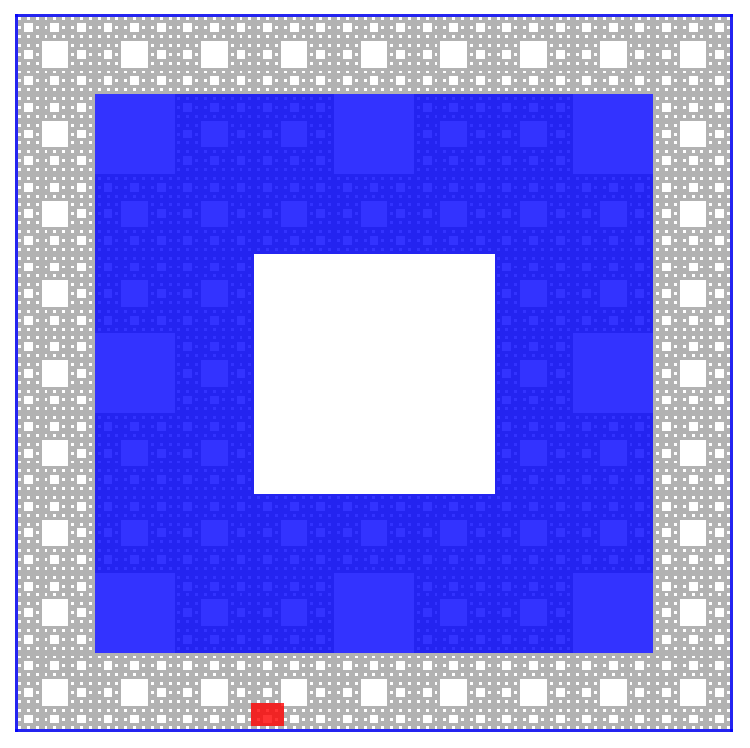}
	\caption{ An illustration of $\operatorname{supp} [\bar{u}_A]$ for $A\in \eth_kF$ with $k=1,2,3$: $\operatorname{supp} [\bar{u}_A]$ is contained in the red area, 
	while  the blue area is $F_{\mcB_{k+2}(F)}\cup (F\setminus F_{\mcB_k-1(F)})$}
	\label{figA1}
\end{figure}

\medskip

For $n\geq 3$, let $g_n:=\bar{u}_o+\sum\limits_{k=1}^{n-1}\sum\limits_{A\in \eth_kF}\bar{u}_A+\sum\limits_{A\in \eth_nF}w_{Q_A}$ 
and define 
	$$		u_o:=\bar{u}_o/g_n  \quad \hbox{ and } \quad 
		u_{A, n} := \begin{cases}
		\bar{u}_A /g_n  \qquad &\hbox{for } A\in \eth_kF  \hbox{ with } 1\leq k\leq n-1 , \cr
		 w_{Q_A}/g_n\ & \hbox{for } A\in \eth_nF.  
		 \end{cases}
	$$
Note  that  $g_n\geq 1$ on $F$,  
 and  $g_{n+k} (x)= g_n(x) $ for each $n\geq 3$, $k\geq 1$ and $x\in F\setminus F_{\mcB_{n-1}(F)}$.  
Moreover,  for $A\in \eth_kF$ with $k\leq n-2$, 
\begin{equation}\label{e:A.16} 
u_{A, n+j} = u_{A,n} \quad \hbox{ for every } j\geq 1.
\end{equation}

Now  define a linear operator $\Ex^{(n)}$ on $C(\partial_o F)$ for each $n\geq 3$ by 
\[
\Ex^{(n)}f :=[f]_{\nu|_{\partial_oF}}u_o+\sum_{k=1}^n\sum_{A\in \eth_kF}[f]_{\nu|_A}u_{A,n}\quad\hbox{ for } f\in C(\partial_oF).
\]
  For each   $n\geq 3$ and $j\geq 0$,
  by \eqref{e:A.14} 
$\operatorname{supp}[u_{A,n+j}]\subset \operatorname{supp}[\bar{u}_A]\subset F_{\mcB_{n-2}(F)}$  for $A\in \eth_kF$ 
  with $n-1\leq k\leq n+j-1 $,  
and $\operatorname{supp}[u_{A,n+j}]\subset \operatorname{supp}[w_{Q_A}]\subset F_{\mcB_{n-2}(F)}$ for $A\in \eth_{n+j}F$.  Thus on $F\setminus F_{\mcB_{n-2}(F)}$, we have by \eqref{e:A.16}
\[
\begin{aligned}
 \Ex^{(n+j)}f    &=[f]_{\nu|_{\partial_oF}}u_o+\sum_{k=1}^{n+j}\sum_{A\in \eth_kF}[f]_{\nu|_A}u_{A,n+j}\\
&=[f]_{\nu|_{\partial_oF}}u_o+\sum_{k=1}^{n-2}\sum_{A\in \eth_kF}[f]_{\nu|_A}u_{A,n+j} \\
&=[f]_{\nu|_{\partial_oF}}u_o+\sum_{k=1}^{n-2}\sum_{A\in \eth_kF}[f]_{\nu|_A}u_{A,n} . 
\end{aligned}
\]
 As a consequence,  for each  $n\geq 3$ and $j\geq 0$, 
\begin{equation}\label{eqnA17}
\Ex^{(n+j)}f =\Ex^{(n)}f 
\quad \hbox{  on }F\setminus F_{\mcB_{n-2}(F)} .
\end{equation}
 On the other hand,  we have by  \eqref{e:A.14} and \eqref{e:A.15} that for $n\geq 4$, $j\geq 0$ and $x\in F_{\mcB_{n-2}(F)}$,
\[
\Ex^{(n+j)}f(x)=\sum_{k=n-3}^{n+j}\sum_{A\in \eth_kF\atop x\in \operatorname{supp}[u_{A,n+j}]}[f]_{\nu|_A}u_{A,n+j}(x) . 
\]
  Hence, if we fix $Q(x)\in \mcB_{n-3}(F)$ so that $x\in Q(x)$, then we have
\begin{equation}\label{eqnA18}
\begin{aligned}
& \inf\{f(y):y\in A',  A'\in \eth_{n-3}F, A'\cap Q(x)\neq\emptyset\} \\
 \leq & \ \Ex^{(n+j)}f(x)\leq \sup\{f(y):y\in A', A'\in \eth_{n-3}F, A'\cap Q(x)\neq\emptyset\}.
\end{aligned}
\end{equation}
We conclude from \eqref{eqnA17} and \eqref{eqnA18} that 
$$\|\Ex^{(n+j)}f-\Ex^{(n)}f\|_\infty\leq \operatorname{Osc}(f,2\sqrt{d}L_F^{3-n}) 
\quad \hbox{for each } n\geq 3 \hbox{ and } j\geq 0,
$$
where $\operatorname{Osc}(f,r):=\sup_{x,y\in F:\rho(x,y)\leq r}|f(x)-f(y)|$.  Consequently, $\Ex^{(n)}f$ converges uniformly on $F$ as $n\to \infty$ to a bounded continuous $\Ex f$ function on $F$ with $\Ex f =f $ on $\partial_o F$.  
 In the rest of the proof, we fix $f\in C(\partial_oF)\cap \Lambda(\partial_oF)$, and write $h_n=\Ex^{(n)}f$ with $ n\geq 3$.
	
	\medskip
	
	We first show that  there is a constant $C_1$ depending only on $F$ such that 
	\begin{eqnarray}\label{e:A.19}
	\mcE^{(F)}(u_o)\leq C_1  & \hbox{and} &   \mcE^{(F)}(u_{A,n})\leq C_1L_F^{(d_w-d_f)k} \nonumber \\
	&& \hbox{for each } n\geq 3 \hbox{ and } A\in \eth_kF  \hbox{ with } 1\leq k\leq n.
	\end{eqnarray}  
	
	\medskip
	
  Note that $\mcE^{(F)}(u_o) <\infty$ as $u_o\in \sF^{(F)}$.  So we only need to estimate  $\mcE^{(F)}(u_{A,n})$. 
From now on,  fix $1\leq k<n$ (the case that $k=n$ follows by the same argument) and $A\in \eth_kF$.   Since $g_n\geq 1$ on $F$, 
$$
1/g_n \leq g_n  \quad \hbox{ and } \quad 
\Big| \frac1{g_n(x)} -\frac1{g_n(y)} \Big| \leq |g_n(x) -g_n (y)| \quad \hbox{for any } x, y\in F.
$$
Thus by the proof of \cite[Lemma 4.3.9]{CF}, $\mu^{(F)}_{\<1/g_n\>} \leq \mu^{(F)}_{\<g_n\>}$ on $F$. 
 Hence for some constant $C_1' >0$ depending only on $F$,
	\[
	\mu^{(F)}_{\<1/g_n\>} \left(\operatorname{supp} [{\bar{u}_A}] \right)
	\leq \mu^{(F)}_{\<g_n\>}(\operatorname{supp}[{\bar{u}_A}])
	\leq C_1'L_F^{(d_w-d_f)k},
	\]
	where  the second inequality is due to the definition of $g_n$, the energy estimates of  $w_{Q_A}$  from Lemma \ref{lemma37}(a)
	and of  $\bar{u}_A$  by the proof of \cite[Lemma 4.3.9]{CF}, 
	and the fact that there are no more than $ (3^d-1) m_F$ of functions among $\bar{u}_0,\bar{u}_A,w_{Q_A}$ 
	 	in the definition of $g_n$
	that are  non-zero on $\operatorname{supp} [{\bar{u}_A}]$. By the strongly local property and the derivation property of $\mu^{(F)}_{\<f\>}$ (see, e.g., \cite[Proposition 4.3.1 and Lemma 4.3.6]{CF}), 
	 $$
	\begin{aligned}
	\mcE^{(F)}(u_{A,n})& =\mu^{(F)}_{\<u_{A,n}\>}{(F)}
	=\mu^{(F)}_{\<\bar{u}_A/g_n\>}(\operatorname{supp}[{\bar{u}_A}])\\
	&\leq 2\Big(\mu^{(F)}_{\<\bar{u}_A\>}(\operatorname{supp} [u_A])+\mu^{(F)}_{\<1/g_n\>}(\operatorname{supp}[u_A])\Big)
	\leq C_1L_F^{(d_w-d_f)k},
	\end{aligned}
	$$
	where in the first inequality   we also used the facts that $\|\bar{u}_A\|_\infty\leq 1$ and $\|1/g_n\|_\infty\leq 1$.
	This proves  the claim \eqref{e:A.19}.
	 
	\medskip
	
 	We next estimate the energy of $h_n:=\Ex^{(n)}f$. 
	 	Let $A\in \eth_kF$ with $2\leq k\leq n-1$.  Note that 
		$$ 
		h_n- [f]_{\nu|_A}= ( [f]_{\nu|_{\partial_oF}} -[f]_{\nu|_A}) u_o+\sum_{j=1}^n \sum_{B\in \eth_j F}  ([f]_{\nu|_{B}} -[f]_{\nu|_A}) u_{B,n}.
		$$
		By \eqref{e:A.14}-\eqref{e:A.15} and the strong local property of the energy measure $\mu^{(F)}_{\< f\>}$ (see \cite[Proposition 4.3.1]{CF}), 
		there are positive constants   $C_2,C_3$ depending only on $F$ so that
		\begin{eqnarray*}
		\mu^{(F)}_{\<h_n\>}(U_A) &=&\mu^{(F)}_{\<h_n(x)-[f]_{\nu|_A}\>}(U_A) \\
		&=& \sum_{j=1}^n \sum_{B\in \eth_j F}  ([f]_{\nu|_{B}} -[f]_{\nu|_A})^2 \mu^{(F)}_{\< u_{B,n}\>} (U_A) \\
		&&      + \sum_{i\not=j=1}^n \sum_{B_1\in \eth_i F, B_j\in \eth_j F}  ([f]_{\nu|_{B_1}} -[f]_{\nu|_A}) ([f]_{\nu|_{B_2}} -[f]_{\nu|_A}) 
		 \mu^{(F)}_{\< u_{B_1,n},  u_{B_2,n} \>} (U_A) \\
		 &\leq & 3^d  \sum_{j=k-1}^{k+1} \sum_{\sum_{A'\in \eth_{j}F\atop A'\cap A\neq \emptyset}\ }  ([f]_{\nu|_{A'}} -[f]_{\nu|_A})^2 \mu^{(F)}_{\< u_{A',n}\>} (U_A) . 
	 	\end{eqnarray*}
		Thus by \eqref{e:A.19}
		\begin{eqnarray}
		\mu^{(F)}_{\<h_n\>}(U_A) 
		&\leq & C_2  \sum_{j=k-1}^{k+1} \sum_{A'\in \eth_{j}F\atop A'\cap A\neq \emptyset\ }   L_F^{(d_w-d_f)j} ([f]_{\nu|_{A'}} -[f]_{\nu|_A})^2 
		\nonumber \\
		&\leq &C_3\sum_{j=k-1}^{k+1}L_F^{(d_w-d_f)j}\sum_{A'\in \eth_{j}F\atop A'\cap A\neq \emptyset}
		 \sum_{B\in \eth_{j}F\atop B\sim A'}\big([f]_{\nu|_{A'}}-[f]_{\nu|_{B}}\big)^2,   \label{e:A.20} 
		\end{eqnarray} 
		where the last inequality is due to the observation that  for any $i\geq 1$, $\wt A\in  \eth_{i-1}F$ and any $\wt B\in  \eth_{i-1}F \hbox{ with } \wt B\supset \wt A$, 
		$$
		\big|[f]_{\nu|_{\wt A}}-[f]_{\nu|_{\wt B}}\big|^2\leq m_I\sum_{B\in \eth_i F\atop B\sim \wt A}\big|[f]_{\nu|_{\wt A}}-[f]_{\nu|_{B}}\big|^2 .
		 $$
	
	 Summing \eqref{e:A.20} over $A\in \eth_kF$, we get  for some positive constant $C_4$ depending only on $F$ that 
	\begin{equation}\label{e:A.21}
 		\mu_{\<h_n\>}^{(F)}(F_{\mcB_k(F)}\setminus F_{\mcB_{k+1}(F)}) \leq C_4\sum_{j=k-1}^{k+1} L_F^{(d_w-d_f)j}I_{j}[f]
		\quad \hbox{for }  2\leq k\leq n-1.
 	\end{equation} 
	By the same arguments, we get  that
	\begin{eqnarray}\label{e:A.22}
 	& \mu_{\<h_n\>}^{(F)}(F_{\mcB_n(F)}) \leq C_4\sum_{j=n-1}^{n} L_F^{(d_w-d_f)j}I_{j}[f]\\
 	& \label{e:A.23}\mu_{\<h_n\>}^{(F)}(F\setminus F_{\mcB_1(F)}) \leq C_4 L_F^{(d_w-d_f)} I_1[f]\\
 	& \label{e:A.24a}  \mu_{\<h_n\>}^{(F)}(F_{\mcB_1(F)}\setminus F_{\mcB_2(F)}) \leq C_4\sum_{j=1}^{2} L_F^{(d_w-d_f)j}I_j[f].
	\end{eqnarray} 
 For example, for  \eqref{e:A.22},  following the same argument as that for \eqref{e:A.19} we have  for each $A\in\eth_nF$,  
\begin{eqnarray*}
	\mu^{(F)}_{\<h_n\>}(F_{Q_A}) 
	&\leq & C'_2  \sum_{j=n-1}^{n} \sum_{A'\in \eth_{j}F\atop A'\cap A\neq \emptyset\ }   L_F^{(d_w-d_f)j} ([f]_{\nu|_{A'}} -[f]_{\nu|_A})^2 
	\nonumber \\
	&\leq &C'_3\sum_{j=n-1}^{j=n}L_F^{(d_w-d_f)j}\sum_{A'\in \eth_{j}F\atop A'\cap A\neq \emptyset}
	\sum_{B\in \eth_{j}F\atop B\sim A'}\big([f]_{\nu|_{A'}}-[f]_{\nu|_{B}}\big)^2. 
\end{eqnarray*}
Estimate  \eqref{e:A.22} follows by taking the summation over $A\in \eth_nF$.

By  \eqref{e:A.21}, \eqref{e:A.22}, \eqref{e:A.23}  and \eqref{e:A.24}, 
\begin{eqnarray}
\mcE^{(F)}(h_n)&=&\mu_{\<h_n\>}^{(F)}(F\setminus F_{\mcB_1(F)})+ 	\sum_{k=1}^{n-1}\mu_{\<h_n\>}^{(F)}(F_{\mcB_k(F)}\setminus F_{\mcB_{k+1}(F)})+\mu_{\<h_n\>}^{(F)}(F_{\mcB_n(F)})\nonumber \\
&\leq&3C_4\sum_{j=1}^{n} L_F^{(d_w-d_f)j}I_{j}[f]\nonumber\\
&\leq&3C_4\Lambda_1[f].\label{e:A.24}
\end{eqnarray}
In particular, we have $\sup_{n\geq 3} \mcE^{(F)}(h_n) <\infty$.  
Since $h_n$ converges  to $\Ex f$  uniformly on $F$   and hence in $L^2(F; \mu)$, 
there is a subsequence of $\{h_n; n\geq 1\}$ whose  Ces\`aro means   converges in $\sqrt{\mcE^{(F)}_1}$-norm to $\Ex f$. Thus in view of \eqref{e:A.24}, 
$\Ex f \in \sF^{(F)}\cap C(F)$ with  $\mcE^{(F)}(\Ex f )  \leq 3C_4\Lambda_1[f] $. 
Moreover, for each $k\geq 1$, it follows from \eqref{e:A.21}, \eqref{e:A.22}, \eqref{e:A.23} and \eqref{e:A.24} that for $n\geq k\vee 3$, 
\begin{eqnarray*}
 	\mu^{(F)}_{\<h_n\>}(F_{\mcB_k(F)}) 
	=\sum_{j=k}^{n-1} \mu_{\<h_n\>}^{(F)}(F_{\mcB_j(F)}\setminus F_{\mcB_{j+1}(F)}) + \mu_{\<h_n\>}^{(F)}(F_{\mcB_n(F)} )\leq \ 3C_4\Lambda_{k-1}[f]. 
\end{eqnarray*}
Consequently, we have $\mu^{(F)}_{\<\Ex f\>}(F_{\mcB_k(F)})\leq 3C_4 \Lambda_{k-1}[f]$ for every $k\geq 1$. Hence $ \mu^{(F)}_{\<\Ex f\>}(\partial_o F ) =\lim_{k\to \infty} \mu^{(F)}_{\<\Ex f\>}(F_{\mcB_k(F)}) =0$. 
This completes the proof for  part (a) of the theorem. 
\end{proof}

\smallskip

\begin{remark} \rm  After the proof of  Theorem \ref{T:A.4}, one can simply replace $\Ex$ and $\Ex_m$  there by
 the harmonic extension operators  as 
harmonic extensions minimize the corresponding energies among those functions having the same boundary data. 
\end{remark}

The next result improves a corresponding result in \cite[\S 5.3]{HK} where  the $(d-1)$-dimensional fractal $\partial_0 F$ is  additionally assumed to satisfy  conditions  (SC1)-(SC4); see \cite[Remarks 2.16 and 5.3]{BBKT}.

\begin{corollary}\label{coroA5}
 $\mu^{(F)}_{\<f\>}(\partial_o F_Q)=0$ for each $f\in \mcF^{(F)}$ and $Q\in \mcQ_n(F)$ with $n\geq 0$.
  As a consequence, for each $n\geq 1$ and $\mathcal{A}\subset \mcQ_n(F)$, we have $\mu^{(F)}_{\<f\>}(F_{\mathcal{A}})=\sum_{Q\in \mathcal{A}}\mcE^{(F_Q)}(f)$. 
\end{corollary}

\begin{proof}
First for every $\varphi \in \mcF^{(F)}\cap C_c (F\setminus \partial_o F)$, by the strong local property of the energy measure from
\cite[Proposition 4.3.1]{CF}, $ \mu^{(F)}_{\<f\>}(\partial_o F)=0$.  Denote by $(\sE^{(F)}, \sF^{(F)}_{F\setminus \partial_o F})$ the Dirichlet form
of the part process of the Brownian motion $X^{(F)}$ on $F$ killed upon hitting $\partial_0 F$.  It is well known (see, e.g.,  \cite[Theorem 3.39]{CF}) 
that $\mcF^{(F)}\cap C_c (F\setminus \partial_o F)$ is $\sqrt{\sE^{(F)}_1}$-dense in $\sF^{(F)}_{F\setminus \partial_o F}$.
Hence 
$$
\mu^{(F)}_{\<f\>}(\partial_o F)=0 \quad \hbox{for every } f\in \sF^{(F)}_{F\setminus \partial_o F}.
$$
For every $f\in  \mcF^{(F)}\cap C(F)$,  $f|_{\partial_o F} \in C(\partial_o F) \cap \Lambda ( \partial_o F)$ by Theorem \ref{T:A.3}.
Thus 
$\varphi := f-\Ex (f|_{\partial_o F}) \in \sF^{(F)}\cap C(F)$  vanishes continuously on $\partial_o F$ and so 
$\varphi \in  \sF^{(F)}_{F\setminus \partial_o F}$.  
Thus it follows from Theorem \ref{T:A.4}(a) that 
$$
\mu^{(F)}_{\<f\>}(\partial_o F) \leq 2 \mu^{(F)}_{\<\varphi \>}(\partial_o F)  + 2 \mu^{(F)}_{\<\Ex (f|_{\partial_o F})  \>}(\partial_o F) =0.
$$
By the regularity  of the Dirichlet form $(\mcE^{(F)},\mcF^{(F)})$,
$\mu^{(F)}_{\<f\>}(\partial_o F) =0$ for every $f\in \mcF^{(F)}$. 
 This implies that for each $f\in \mcF^{(F)}$ and $Q\in \mcQ_n(F)$ with $n\geq 1$,
\begin{equation}\label{e:A.25}
\mu^{(F_Q)}_{\<f\>}(\partial_o F_Q)=L_F^{n(d_w-d_f)}\mu^{(F)}_{\<f\circ \Psi_Q\>}(\partial_oF)=0
\end{equation}
due to Lemma \ref{lemma31} and \eqref{e:3.3}.   It follows from Lemma \ref{lemmaA2} and \eqref{e:A.25} that
$$
\mu^{(F)}_{\<f\>}  (\partial_o F_{Q })  = 
\sum_{{Q'}\in {\mathcal Q}_n(F)} \mu^{(F_{Q'})} _{\<f\>}  (  {Q'} \cap \partial_o F_{Q })  
\leq   \sum_{{Q'}\in {\mathcal Q}_n(F)} \mu^{(F_{Q'})} _{\<f\>}  (    \partial_o F_{Q'}) =0.
$$

For any $\mathcal{A}\subset \mcQ_n(F)$ with $n\geq 1$,  again by  Lemma \ref{lemmaA2} and \eqref{e:A.25}, 
\begin{align*}
\mu^{(F)}_{\<f\>}(F_{\mathcal{A}})&=\sum_{Q\in \mcQ_n(F)}\mu^{(F_Q)}_{\<f\>}(F_Q\cap F_{\mathcal{A}})\\
&=\sum_{Q\in\mathcal{A}}\mu^{(F_Q)}_{\<f\>}(F_Q\cap F_{\mathcal{A}})+\sum_{Q\in \mcQ_n(F)\setminus \mathcal{A}}\mu^{(F_Q)}_{\<f\>}(F_Q\cap F_{\mathcal{A}})=\sum_{Q\in \mathcal{A}}\mcE^{(F_Q)}(f),
\end{align*}
where in the last equality, we used the facts that $F_Q\subset F_{\mathcal{A}}$ if $Q\in \mathcal{A}$ and $F_Q\cap F_{\mathcal{A}}\subset \partial_oF_Q$ if $Q\in \mcQ_n(F)\setminus \mathcal{A}$.  
\end{proof}

\section{An estimate of energy measure}\label{secB}

In this appendix, we show that the energy measure on $\Psi_Q(F_{\mcB_n})$ of a function that is $\mcE^{(F)}$-harmonic in a neighborhood of the boundary of a cell $F_Q$, decreases at an exponential rate in $n\to\infty$. A similar type result is given in \cite[Proposition 3.8]{HK}  as a preparation for the restriction theorem under some additional assumptions. Our approach is different from that in \cite{HK} and is based on the idea of trace theorems.
 
\begin{theorem}\label{thmB1}
	There are positive finite contants $C,c$ depending only on $F$ such that for each $l\geq 1,Q^*\in \mathcal{Q}_l(F)$, $n\geq 0$ and $f\in \mcF^{(F)}\cap C(F)$ that is $\mcE^{(F)}$-harmonic in the interior of $F_{S_{Q^*}}$, where $\mathcal{S}_{Q^*}=\{Q\in\mathcal{Q}_l(F):Q\cap Q^*\neq\emptyset\}$, we have 
	\[
	\mu^{(F)}_{\< f \>}(\Psi_{Q^*} (F_{\mcB_n(F)}))\leq Ce^{-cn}\, \mu^{(F)}_{\< f \>}\big(F_{\mathcal{S}_{Q^*}}\big).
	\] 
\end{theorem}

\begin{proof}
	Without loss of generality, we consider neighborhoods of one face, say,  $\partial_{1,0}F_{Q^*}=\Psi_{Q^*}(\partial_{1,0}F)$. For $k\geq l$, let 
	\begin{align*}
	\sG_k:=\{Q\in \mathcal{Q}_k(F):Q\cap \partial_{1,0}F_{Q^*}\neq \emptyset\}.
	\end{align*}
    So $F_{\sG_k}$ is a neighborhood of $\partial_{1,0}F_{Q^*}$ in $(F,\rho )$. Let $D_k$ be a closed $L_F^{-k}$-neighborhood of $\partial_{1,0}  F_{0,Q^*}$ with respect to the $L^\infty$-metric; that is, if $\partial_{1,0} F_{0,Q^*}=\{s_1\}\times [s_2,s_2+L_F^{-l}]\times\cdots\times [s_k,s_k+L_F^{-l}]$, then 
     $$
     D_k=[s_1-L_F^{-k},s_1+L_F^{-k}]\times [s_2-L_F^{-k},s_2+L_F^{-l}+L_F^{-k}]\times\cdots\times [s_k-L_F^{-k},s_k+L_F^{-l}+L_F^{-k}].
     $$
     Note that  $\sG_k=\mathcal{Q}_k(D_k\cap F)$. 
    
    \medskip
    
    For this proof only,   we define for $n\geq k$,
    \[
    \eth_nF_{\sG_k}:=\{\partial_{i,s}F_Q:Q\in\mcQ_n(F_{\sG_k}),\ i\in \{1,2,\cdots,d\},\ s\in\{0,1\},\ \partial_{i,s}F_Q\subset \partial D_k\}.
    \]
    Observe that the union of the subfaces in  $\eth_nF_{\sG_k}$ contains the topological boundary of $F_{\sG_k}$. \
    
    \medskip

	\noindent\textbf{(Discrete energies)}. 
	
	(a)   Similar to the definition of $I_k$, for $n>k\geq l$, we define
	\[
	I_n(f,\sG_k)=\sum_{A,A'\in \eth_nF_{\sG_k}:A\sim A'}([f]_{\nu|_A}-[f]_{\nu|_{A'}})^2,
	\]
	where $A\sim A'$ if and only if $A\cap A'\neq\emptyset$ or there is $B\in \eth_{n-1}F_{\sG_k} $ such that $A,A'\subset  B$. 
	
	(b). For $n=k\geq l$, we define
	\[
	I_k(f,\sG_k)=\sum_{A,A'\in \eth_kF_{\sG_k}:A\sim A'}([f]_{\nu|_A}-[f]_{\nu|_{A'}})^2,
	\]
	where $A\sim A'$ in $\eth_kF_{\sG_k}$ if and only if $\nu(Q_A\cap Q_{A'})>0$, where  $Q_A,Q_{A'}\in \sG_k$ so that $A\subset Q_A,A'\subset Q_{A'}$. \medskip

	We have two comments about   (b) above. 
	\begin{enumerate}
		\item If $Q_A\cap Q_{A'}\neq\emptyset$, then by (non-diagonality) condition, there is a sequences of cells $Q_A=Q_0$, $Q_1$, $Q_2$, $\cdots$, $Q_{s-1}$, $Q_{A'}=Q_{s}$ 
		in $ \sG_k$ with $s<2^d$ such that $\nu(Q_i\cap Q_{i-1})>0$.
	 		
		\item A special case is when $Q_A\cap Q_{A'}\subset \partial_{1,0}F_{Q^*}$. In this case,    $A,A'$ are two sub-faces on the 
		opposite sides of $\partial_{1,0}F_{Q^*}$, and  we denote it as $A\sim^*A$. Let
		\[
		I^*_k(f,\sG_k)=\sum_{A,A'\in \eth_kF_{\sG_k}:A\sim^* A'}([f]_{\nu|_A}-[f]_{\nu|_{A'}})^2,
		\]
		and  $I^{**}_k(f,\sG_k)=I_k(f,\sG_k)-I_k^*(f,\sG_k)$.
	\end{enumerate}
		
		\medskip
		
	By the same arguments as that for Theorems \ref{T:A.3} and \ref{T:A.4}, we have the following Claims 1 and 2, respectively. 
	 	
	\medskip
	
	\noindent\textbf{Claim 1}. There is a constant $C_1\in (0,\infty)$ depending only on $F$
	so that for each $n\geq k\geq l$, 
	\[
	\mu_{\< f\>}^{(F)}(F_{\sG_k}\setminus F_{\mathcal{G}_{k+1}})
	\geq    C_1\Big(L_F^{(d_w-d_f)k}I^{**}_k(f,\sG_k)+
	\sum_{n=k+1}^\infty 
	L_F^{(d_w-d_f)n}I_n(f,\sG_k)\Big) .
	\]

	\noindent\textbf{Claim 2}. For some $C_2\in (0,\infty)$ depending only on $F$,  
	 \[\mu^{(F)}_{\< f \>}(F_{\sG_k})\leq C_2\sum_{n=k}^\infty L_F^{(d_w-d_f)n}I_n(f,\sG_k).\]

	\medskip

	 For each $k\geq 0$, let 
	\begin{align*}
	{J_k}:=\bigcup_{Q\in \mathcal{Q}_k(F): Q\cap \partial_{1,0}F\neq \emptyset}F_Q=F\cap ([0,L_F^{-k}]\times [0,1]^{d-1}).
	\end{align*} 
    
    \noindent\textbf{Claim 3}.  We claim the following  holds.
    \begin{enumerate}
    \item[(i)] There is a positive finite constant $C_3$ depending only on $F$ such that 
	\[
	\Big|[g]_{\nu|_{\partial_{1,0}F}}-[g]_{\nu|_{F \cap \{x_1=L_F^{-k}\}}}\Big| ^2\leq C_3L_F^{k(d_f-d_w-d_I)}\mu^{(F)}_{\<g\>}({J_k})\quad\hbox{ for every }g\in C(F)\cap \mcF^{(F)}.
	\]
	
	\item[(ii)] There is a positive finite constant $C_{4,k}$ depending on $F$ and $k$ such that 
	\[
	\Big| [g]_{\nu|_{\partial_{1,1}F}}-[g]_{\nu|_{F\cap \{x_1=L_F^{-k} \}}}
	 \Big|^2\leq C_{4,k}\mu^{(F)}_{\<g\>}(F\setminus {J_{k+1}})\quad\hbox{ for every }g\in C(F)\cap \mcF^{(F)},
	\]
	where we used $F\setminus {J_{k+1}}$ instead of $F\setminus {J_k}$ because we want  $F\cap\{x_1=L_F^{-k}\}$ and $\partial_{1,1}F$ 
	 to be in the same connected component of $F\setminus {J_{k+1}}$.
	\end{enumerate}

	\medskip
	
     We first show that there is a positive finite $C_3'$ depending only on $F$ so that 
	\begin{equation}\label{eqnB1}
	\big|[g]_{\mu|_F}-[g]_{\mu|_{J_1}}\big|^2\leq C_3'\mcE^{(F)}(g).
	\end{equation}
	Let $j$ be the smallest integer so that $\sqrt{d}L_F^{-j}<c_1$, where $c_1$ is the constant of the Poincar\'e inequalities in Lemma \ref{lemma39}. 
	 For $Q,Q'\in \mcQ_j(F)$  with $Q\cap Q'\neq\emptyset$,  take some $x_0\in F_Q\cap F_{Q'}$. Note that $B(x_0, c_1) \supset F_Q \cup F_{Q'}$. By Lemma \ref{lemma39}(a) with $r=1$, 
  
	 	\begin{align*}
	\big|[g]_{\mu|_{F_Q}}-[g]_{\mu|_{F_{Q'}}}\big|&\leq \big|[g]_{\mu|_{F_Q}}-[g]_{\mu|_{B_F(x_0,c_1)}}\big|+\big|[g]_{\mu|_{F_Q'}}-[g]_{\mu|_{B_F(x_0,c_1)}}\big|\\
	\leq&\sqrt{\fint_{F_{Q}}\big(g(y)-[g]_{\mu|_{B_F(x_0,c_1)}}\big)^2\mu(dy)}+\sqrt{\fint_{F_{Q'}}\big(g(y)-[g]_{\mu|_{B_F(x_0,c_1)}}\big)^2\mu(dy)}\\
	\leq&2\sqrt{\frac{\mu\big(B_F(x_0,c_1)\big)}{L_F^{-jd_f}}\fint_{B_F(x_0,c_1)}\big(g(y)-[g]_{\mu|_{B_F(x_0,c_1)}}\big)^2\mu(dy)}\\
	\leq&C_4'\sqrt{\mcE^{(F)}(g)},
	\end{align*}
  where  $C_4'>0 $  is a constant depending only on $F$. 
     Since   $\#\mcQ_j(F)=m_F^{j}$, and $F$ is connected, we have for $k\geq 1$, 
   $$
   \big|[g]_{\mu|_{F_Q}}-[g]_{\mu|_{F_{Q'}}}\big|\leq  (m_F^j-1)C_4'\sqrt{\mcE^{(F)}(g)}  \quad \hbox{for any }  Q,Q'\in \mcQ_j(F).
   $$
   Since $\min_{Q\in \mcQ_j(F)}[g]_{\mu|_{F_Q}}\leq [g]_{\mu|_F}\leq \max_{Q\in \mcQ_j(F)}[g]_{\mu|_{F_Q}}$ and $\min_{Q\in \mcQ_j(J_1)}[g]_{\mu|_{F_Q}}\leq [g]_{\mu|_{J_1}}\leq \max_{Q\in \mcQ_j(J_1)}[g]_{\mu|_{F_Q}}$, we conclude that \eqref{eqnB1} holds with $ C_3 =( (m_F^j-1)C_4')^2$, which depends only on $F$. 
    	
   Applying \eqref{eqnB1} on $F_Q$ and using the Cauchy-Schwarz inequality, we have
	\begin{align*}
	\big|[g]_{\mu|_{{J_k}}}-[g]_{\mu|_{{J_{k+1}}}}\big|^2
	&=\Big|\frac{1}{\#\mcQ_k(J_k)}\sum_{Q\in \mcQ_k(J_k)}[g]_{\mu|_{F_Q}}-\frac{1}{\#\mcQ_k(J_k)}\sum_{Q\in \mcQ_k(J_k)}[g]_{\mu|_{\Psi_Q(J_1)}}\Big|^2\\
	&=\Big|\frac{1}{\#\mcQ_k(J_k)}\sum_{Q\in \mcQ_k(J_k)}\big([g\circ \Psi_Q]_{\mu|_F}-[g\circ \Psi_Q]_{\mu|_{J_1}}\big)\Big|^2\\
	&\leq \frac{1}{\#\mcQ_k(J_k)}\sum_{Q\in \mcQ_k(J_k)}\big([g\circ \Psi_Q]_{\mu|_F}-[g\circ \Psi_Q]_{\mu|_{J_1}}\big)^2\\
	&\leq L_F^{-kd_I}\  \sum_{Q\in \mcQ_k(J_k)}C_3'\ \mcE^{(F)}(g\circ\Psi_Q)\\
	&= L_F^{-kd_I}\ L_F^{k(d_f-d_w)}\  C_3'\ \sum_{Q\in \mcQ_k(J_k)} \mcE^{(F_Q)}(g)\\
	&\leq L_F^{k(d_f-d_w-d_I)}C_3'\ \mcE^{(F)} (g),
	\end{align*}
    where we used the self-similarity  \eqref{e:3.3} of $(\mcE^{(F)},\mcF^{(F)})$  and Lemma \ref{lemma31} in the last two lines. 
     As $d_f-d_w-d_I<0$,  we have for every $i\geq 1$, 
	$$
	\big|  [g]_{\mu|_F} -[g]_{\mu|_{{J_{i}}}}\big|\leq \sum_{k=1}^{i-1} \big|[g]_{\mu|_{{J_k}}}-[g]_{\mu|_{{J_{k+1}}}}\big|
	\leq \frac{\sqrt{C_3' \, \mcE^{(F)} (g)  }}{1-L_F^{(d_f- d_w-d_I)/2}}  .
	$$
	Since $[g]_{\nu|_{\partial_{1,0}F}}=\lim\limits_{i\to\infty}[g]_{\mu|_{{J_i}}}$, it follows from the above  that there is a constant $C_5'>0$ depending only on $F$ so that 
	\[
	\big|[g]_{\nu|_{\partial_{1,0}F}}-[g]_{\mu|_F}\big|^2\leq C_5'\mcE^{(F)}(g),
	\]
	By symmetry, we also have $\big|[g]_{\nu|_{\partial_{i,s}F}}-[g]_{\mu|_F}\big|^2\leq C_5'\mcE^{(F)}(g)$ for  $i\in \{1,2,\cdots,d\}$ and $s\in \{0,1\}$.
	Consequently,  
	\begin{equation}\label{e:B.1a}
	\big|[g]_{\nu|_{\partial_{i,s}F}}-[g]_{\nu|_{\partial_{i',s'}F}}\big|^2\leq 4C_5'\mcE^{(F)}(g)
	\quad \hbox{ for }i,i'\in \{1,2,\cdots,d\}
	\hbox{ and } s,s'\in\{0,1\}.
	\end{equation}
	Hence for each $Q\in \mathcal{Q}_k({J_k})$,
	$$
	\big|[g]_{\nu|_{\partial_{1,0}F_Q}}-[g]_{\nu|_{\partial_{1,1}F_Q}}\big|^2\leq 4C_5'\mcE^{(F_Q)}(g\circ \Psi_Q)
	 	\leq   4C_5' L_F^{k(d_f-d_w)} \mcE^{(F_Q)}( g).  
	$$
	Thus   
	\begin{eqnarray*}
	&& \Big|  [g]_{\nu|_{\partial_{1,0}F}}-[g]_{\nu|_{F   \cap \{x_1=L_F^{-k} \}}   }  \Big|^2 \\
	&\leq & \Big(  \sum_{Q\in  \mathcal{Q}_k({J_k})}\frac{\nu (\partial_{1,0}F_Q ) } {\nu (  \partial_{1,0}F )}  \big|[g]_{\nu|_{\partial_{1,0}F_Q}}-[g]_{\nu|_{\partial_{1,1}F_Q}}\big| \Big)^2 \\
	  &\leq & \sum_{Q\in  \mathcal{Q}_k({J_k})}     \frac{\nu (\partial_{1,0}F_Q ) } {\nu (  \partial_{1,0}F )}    \big|[g]_{\nu|_{\partial_{1,0}F_Q}}-[g]_{\nu|_{\partial_{1,1}F_Q}}\big|^2  \\
	&\leq &     \sum_{Q\in  \mathcal{Q}_k({J_k})} L_F^{-kd_I}   4C_5' L_F^{k(d_f-d_w)} \mcE^{(F_Q)}(g)  \\
	&\leq &    4C_5'       L_F^{k(d_f-d_w -d_I  )}\mcE^{(F )}(g) ,
	\end{eqnarray*}
	where in the third inequality we used the fact that $\nu$ is a $d_I$-dimensional Hausdorff measure. 
 	This proves Claim 3(i). 

	 For Claim 3(ii), we fix a pair $\wt Q^*, \wt Q^{**}\in \mcQ_{k+1}(F\setminus J_k)$ such that $F_{\wt Q^*}\cap \{x_1=L_F^{-k}\}\neq \emptyset,\ F_{\wt Q^{**}}\cap \{x_1=1\}\neq\emptyset$. By Lemma \ref{lemmaA2},   $F\cap ([\frac12L_F^{-k},1]\times [0,1]^{d-1})$ is path connected.
	Thus  by  the (non-diagonality) condition of {\it GSC}, there is a sequence
	 ${\wt Q^*} =Q_1,Q_2,\cdots,Q_J={\wt Q^{**}}$ in $\mcQ_{k+1}(F\setminus J_{k+1})$ with  $J\leq \#\mcQ_{k+1}(F\setminus J_{k+1})\leq L_F^{(k+1)d_f}$ and $\nu(F_{Q_j}\cap F_{Q_{j-1}})>0$ for each $j\in \{2,3,\cdots,J\}$ (i.e. each $F_{Q_j} \cap F_{Q_{j+1}}$ is a face of a $(k+1)$-level cell  in $ F$).
	 Applying \eqref{e:B.1a}  on each $Q_j$, we have by the Cauchy-Schwarz inequality, 
	\begin{eqnarray}
	&&\big|[g]_{\nu|_{F_{\wt Q^*}\cap \{x_1=L_F^{-k}\}}}-[g]_{\nu|_{F_{\wt Q^{**}}\cap \{x_1=1\}}}\big|  \nonumber \\
 	&\leq &\big|[g]_{\nu|_{F_{\wt Q^*}\cap\{x_1=L_F^{-k}\}}}-[g]_{\nu|_{F_{Q_1}\cap F_{Q_2}}}\big|+\sum_{j=2}^{J-1}\big|[g]_{\nu|_{F_{Q_{j-1}}\cap F_{Q_j}}}-[g]_{\nu|_{F_{Q_j}\cap F_{Q_{j+1}}}}\big|   \nonumber  \\
	&&+\big|[g]_{\nu|_{F_{Q_{J-1}}\cap F_{Q_{J}}}}-[g]_{\nu|_{F_{Q_J}\cap \{x_1=1\}}}\big|   \nonumber  \\
	&\leq &2\sqrt{C_5'}\sum_{j=1}^J\sqrt{\mcE^{(F)}(g\circ \Psi_{Q_j})}   \nonumber  \\
	&\leq &2\sqrt{C_5'}\sqrt{J}\sqrt{\sum_{j=1}^J\mcE^{(F)}(g\circ \Psi_{Q_j})}    \nonumber  \\
	&\leq &2\sqrt{C_5'}\ L_F^{(k+1)d_f/2}\ L_F^{(k+1)(d_w-d_f)/2}\ \sqrt{\mu^{(\mcF)}_{\<g\>}(F\setminus J_{k+1})}, 
	\label{e:B.3a}
	\end{eqnarray}
	     where  in the last inequality we used  \eqref{e:3.3} and Lemma \ref{coroA5}. 
	 Since  $\Big|[g]_{\nu|_{\partial_{1,1}F}}-[g]_{\nu|_{F\cap\{x_1=L_F^{-k}\}}}\Big|\leq \max\Big\{|[g]_{\nu|_{F_{\wt Q^*}\cap \{x_1=L_F^{-k}\}}}-[g]_{\nu|_{F_{\wt Q^{**}}\cap \{x_1=1\}}}|:\ \wt Q^*, \wt Q^{**}\in \mcQ_{k+1}(F\setminus J_k),\ F_{\wt Q^*}\cap \{x_1=L_F^{-k}\}\neq \emptyset, F_{\wt Q^{**}}\cap \{x_1=1\}\neq\emptyset\Big\}$,   Claim 3(ii) follows from    estimate \eqref{e:B.3a}.	
	 
	\medskip 
	
    For each $j\geq 1$ and each pair $A\sim^*A'$ in $\eth_kF_{\sG_k}$, we define $\operatorname{Pr}_j(A)$ and $\operatorname{Pr}_j(A')$ to be the two parallel `faces' between $A$ and $A'$  that are isometric to $A$ and $A'$: $\operatorname{Pr}_j(A)$ is on the $A$ side,  $\operatorname{Pr}_j(A')$ is on the $A'$ side, and the distance between $\operatorname{Pr}_j(A)$ and $\operatorname{Pr}_j(A')$ is $2L_F^{-k-j}$. 
    More precisely,  suppose without loss of generality that $A=\partial_{1,1}Q_A$ and $A'=\partial_{1,0}Q_{A'}$,
    then 
    \[
	\begin{cases}
		\operatorname{Pr}_j(A):=\bigcup \big\{\partial_{1,1}F_{\wt{Q}}:\wt{Q}\in \mcQ_{k+j}(F_{Q_A})  \ \hbox{ with }  \ 
		  \wt Q \cap\partial_{1, 0}F_{Q^*}\not= \emptyset \big\},\\
		\operatorname{Pr}_j(A'):=\bigcup \big\{ \partial_{1,0}F_{\wt{Q}}:\wt{Q}\in \mcQ_{k+j}(F_{Q_{A'}})   \  
		\hbox{ with }  \   \wt Q \cap\partial_{1, 0}F_{Q^*}\not= \emptyset \big\}.
	\end{cases}
	\]
     	
	\medskip
	
  Let $C_3>0$ and $C_{4,j}>0 $, $j\geq 1$,  be the constants in Claim 3(i) and (ii),  
  respectively.   For any $k\geq l$ and $j\geq1$,  
  by Claim 3(ii), 
	 	\begin{eqnarray}\label{e:B.1} 
		 &&L_F^{(d_w-d_f)k}\sum_{A,A'\in \eth_kF_{\sG_k}\atop A'\sim^* A}\Big(([f]_{\nu|_A}-[f]_{\nu|_{\operatorname{Pr}_j(A)}})^2+([f]_{\nu|_{\operatorname{Pr}_j(A')}}-[f]_{\nu|_{A'}})^2\Big)  \nonumber \\
	&\leq& \sum_{A,A'\in \eth_kF_{\sG_k}\atop A'\sim^* A}C_{4,j}\ \Big(\mu^{(F)}_{\<f\>}(F_{Q_A}\setminus F_{\sG_{k+j+1}}) +\mu^{(F)}_{\<f\>}(F_{Q_{A'}}\setminus F_{\sG_{k+j+1}})\Big)   \nonumber \\
		&\leq&2C_{4,j}\,\mu^{(F)}_{\< f \>}(F_{\sG_k}\setminus F_{\mathcal{G}_{k+j+1}}),
 	  \end{eqnarray}
	where the first inequality is due to the fact that  $F_{Q_A}\setminus F_{\sG_{k+j+1}}$ is a scaled and rotated version of $F\setminus J_{j+1}$ and  the self-similarity property \eqref{e:3.3}, while the last inequality is due to  Corollary \ref{coroA5}.
 	For any $k\geq l$ and $j\geq1$, we also have
	\begin{align}
		&L_F^{(d_w-d_f)k}\,\sum_{A,A'\in \eth_kF_{\sG_k}\atop A'\sim^* A}([f]_{\nu|_{\operatorname{Pr}_j(A)}}-[f]_{\nu|_{\operatorname{Pr}_j(A')}})^2\nonumber\\
		\leq &L_F^{(d_w-d_f)k}\,\sum_{A,A'\in \eth_kF_{\sG_k}\atop A'\sim^* A}2\Big(([f]_{\nu|_{\operatorname{Pr}_j(A)}}-[f]_{\nu|_{F_{Q_A}\cap F_{Q_{A'}}}})^2+([f]_{\nu|_{F_{Q_A}\cap F_{Q_{A'}}}}-[f]_{\nu|_{\operatorname{Pr}_j(A')}})^2\Big)\nonumber\\
		\leq &2\sum_{A,A'\in \eth_kF_{\sG_k}\atop A'\sim^* A} C_3 \,L_F^{(d_f-d_w-d_I)j}  \ \Big(\mu^{(F)}_{\<f\>}(F_{Q_A}\cap F_{\sG_{k+j}}) +\mu^{(F)}_{\<f\>}(F_{Q_{A'}}\cap F_{\sG_{k+j}})\Big)\nonumber\\
		\leq &4C_3\,L_F^{(d_f-d_w-d_I)j}\,\mu^{(F)}_{\< f \>}(F_{\mathcal{G}_{k+j}}) ,        \label{e:B.2}
	 \end{align}
    where in the second inequality we used the fact that $F_{Q_A}\cap F_{\sG_{k+j}}$ is a scaled and rotated version of $J_{j}$,  the self-similarity \eqref{e:3.3}  and Claim 3(i), while in the last inequality we used Corollary \ref{coroA5}.

\medskip

Recall that 
	$$
	I_k^*(f,\sG_k)=\sum_{\sum_{A,A'\in \eth_kF_{\sG_k}\atop A'\sim^* A}}([f]_{\nu|_A}-[f]_{\nu|_{A'}})^2
	\quad \hbox{and} \quad 
	I_k^{**}(f,\sG_k)=I_k(f,\sG_k)-I_k^*(f,\sG_k).
	$$
 When  $L_F^{(d_w-d_f)k}I^{*}_k(f,\sG_k)\leq L_F^{(d_w-d_f)k}I^{**}_k(f,\sG_k)+\sum\limits_{n=k+1}^\infty L_F^{(d_w-d_f)n}I_n(f,\sG_k)$, 
   \begin{eqnarray}
 	&& L_F^{(d_w-d_f)k}I^{**}_k(f,\sG_k)+\sum\limits_{n=k+1}^\infty L_F^{(d_w-d_f)n}I_n(f,\sG_k). \nonumber \\
 	&\geq &\frac12 L_F^{(d_w-d_f)k} I^{*}_k(f,\sG_k)  
	+\frac12\Big(L_F^{(d_w-d_f)k}I^{**}_k(f,\sG_k)+\sum\limits_{n=k+1}^\infty L_F^{(d_w-d_f)n}I_n(f,\sG_k)\Big) \nonumber \\
 	&=&\frac12\sum\limits_{n=k}^\infty L_F^{(d_w-d_f)n}I_n(f,\sG_k). \label{e:B.6}
 \end{eqnarray}
 By taking $j=1$ in Claims 1 and 2, we have from \eqref{e:B.6} that there is a positive constant $C_5\in (0, 1)$ depending only on $F$ such that 
	\begin{equation} \label{e:B.3}
	\mu^{(F)}_{\< f \>}(F_{\sG_k}\setminus F_{\mathcal{G}_{k+j}})\geq C_5\mu^{(F)}_{\< f \>}(F_{\sG_k})
	\quad  \hbox{for }  k\geq l.
	\end{equation}
Consequently,  
	\begin{equation} \label{e:B.4}
\mu^{(F)}_{\< f \>}( F_{\mathcal{G}_{k+j}})\leq (1-C_5)\mu^{(F)}_{\< f \>}(F_{\sG_k})
	\quad  \hbox{for }  k\geq l.
		\end{equation}
When $L_F^{(d_w-d_f)k}I^{*}_k(f,\sG_k)>L_F^{(d_w-d_f)k}I^{**}_k(f,\sG_k)+\sum\limits_{n=k+1}^\infty L_F^{(d_w-d_f)n}I_n(f,\sG_k)$,  
\begin{eqnarray}
		&& L_F^{(d_w-d_f)k}I^{*}_k(f,\sG_k) \nonumber \\
		&> &\frac12 L_F^{(d_w-d_f)k}I^{*}_k(f,\sG_k)+\frac12\Big(L_F^{(d_w-d_f)k}I^{**}_k(f,\sG_k)+\sum\limits_{n=k+1}^\infty L_F^{(d_w-d_f)n}I_n(f,\sG_k)\Big)
		\nonumber \\
		&=&\frac12\sum\limits_{n=k}^\infty L_F^{(d_w-d_f)n}I_n(f,\sG_k). \label{e:B.9}
	\end{eqnarray}
It follows from \eqref{e:B.2} and Claim 2 that 
	\begin{eqnarray*}
		&& L_F^{(d_w-d_f)k}\sum_{\sum_{A,A'\in \eth_kF_{\sG_k}: A'\sim^* A}}([f]_{\nu|_{\operatorname{Pr}_j(A)}}-[f]_{\nu|_{\operatorname{Pr}_j(A')}})^2 \\
		&\leq &4C_3\,L_F^{(d_f-d_w-d_I)j}\mu^{(F)}_{\< f \>}(F_{\mathcal{G}_{k+j}})\\
		&\leq &4C_3\,L_F^{(d_f-d_w-d_I)j}\mu^{(F)}_{\< f \>}(F_{\mathcal{G}_{  k }})\\
		&\leq &8C_3 C_2\,L_F^{(d_f-d_w-d_I)j}\,L_F^{(d_w-d_f)k}I^{*}_k(f,\sG_k) .
	\end{eqnarray*}
	Hence there is $j\geq 2$ depending only on $C_2C_3$  and $F$  so that 
\begin{equation}\label{e:B.8}
	L_F^{(d_w-d_f)k}\sum_{A,A'\in \eth_kF_{\sG_k}: A'\sim^* A}([f]_{\nu|_{\operatorname{Pr}_{  j-1}(A)}}-[f]_{\nu|_{\operatorname{Pr}_{  j-1}(A')}})^2\leq \frac{1}{6}L_F^{(d_w-d_f)k}I^{*}_k(f,\sG_k). 
\end{equation} 
 Since
\begin{eqnarray*}
&& I_k^*(f,\sG_k)\\
&=&  \sum_{A,A'\in \eth_kF_{\sG_k}\atop A'\sim^* A} ([f]_{\nu|_A}-[f]_{\nu|_{\operatorname{Pr}_{j-1}(A)}}+
[f]_{\nu|_{\operatorname{Pr}_{j-1}(A)}} -[f]_{\nu|_{\operatorname{Pr}_{j-1}(A')}} +[f]_{\nu|_{\operatorname{Pr}_{j-1}(A')}}
-[f]_{\nu|_{A'}})^2 \\
&\leq & 3 \ \sum_{A,A'\in \eth_kF_{\sG_k}\atop A'\sim^* A} 
\left(([f]_{\nu|_A}-[f]_{\nu|_{\operatorname{Pr}_{j-1}(A)}})^2+
([f]_{\nu|_{\operatorname{Pr}_{j-1}(A)}} -[f]_{\nu|_{\operatorname{Pr}_{j-1}(A')}} )^2
+([f]_{\nu|_{\operatorname{Pr}_{j-1}(A')}}
-[f]_{\nu|_{A'}})^2 \right),
\end{eqnarray*}
we have by \eqref{e:B.1}, \eqref{e:B.8}, Claim 2 and \eqref{e:B.9} that 
\begin{eqnarray*}
&& 2C_{4,j}\,\mu^{(F)}_{\< f \>}(F_{\sG_k}\setminus F_{\mathcal{G}_{k+j}}) \\
	&\geq & L_F^{(d_w-d_f)k}\sum_{A,A'\in \eth_kF_{\sG_k}\atop A'\sim^* A}\Big(([f]_{\nu|_A}-[f]_{\nu|_{\operatorname{Pr}_{j-1}(A)}})^2+([f]_{\nu|_{\operatorname{Pr}_{j-1}(A')}}-[f]_{\nu|_{A'}})^2\Big) \\
	&\geq & \frac13 (1-\frac{3}{6})L_F^{(d_w-d_f)k}I^{*}_k(f,\sG_k)\\
	&\geq& \frac1{12}\sum\limits_{n=k}^\infty L_F^{(d_w-d_f)n}I_n(f,\sG_k)\\
	&\geq &  \frac1{ 12C_2} L_F^{(d_w-d_f)k}  \mu^{(F)}_{\< f \>}(F_{\sG_k}) . 
\end{eqnarray*} 	
By decreasing the value of $C_5\in (0, 1)$ if needed,  \eqref{e:B.3} and hence \eqref{e:B.4}  holds in this case as well.  
The theorem then follows by iterating the estimate \eqref{e:B.4} and taking the union over faces. 
\end{proof}

An analogous result holds on the approximation domain $F_m$ of $F$ as well with the same proof as that for Theorem \ref{thmB1}. We record it below, which will be used in a  forthcoming paper \cite{CC}. 

\begin{theorem}\label{thmB2}
	Suppose that  $\big(\bar{\mcE}^{(F_m)},W^{1,2}(F_m)\big)$ is  a strongly local regular Dirichlet form on $L^2(F_m; \mu_n)$ so that $C_0^{-1}\mcE^{(F_m)}\leq \bar{\mcE}^{(F_m)}\leq C_0\mcE^{(F_m)}$, where  $C_0\in [1,\infty)$ is  a constant  independent of $m$. Then there are positive   constants $C$ and $c$  depending only on $F$ and $C_0$ such that for each $l\geq 1$, $Q^*\in \mathcal{Q}_l(F_m)$, $n\geq 0$ and $f\in W^{1,2}(F_m)\cap C(F_m)$ that is harmonic in the interior of $F_{m,\mathcal{S}^{(m)}_{Q^*}}$, 
	where $\mathcal{S}^{(m)}(Q^*):=\{Q\in Q_l(F_m):Q\cap Q^*\neq\emptyset\}$, 
 	we have 
	\[
	\bar{\mu}^{(F_m)}_{\<f\>}(\Psi_{Q^*} ( F_{m,\mcB_n(F_m)} ))
	\leq Ce^{-cn}\, \bar{\mu}^{(F_m)}_{\<f\>}\big(F_{m,\mathcal{S}^{(m)}_{Q^*}}\big).
	\] 
	Here $\bar{\mu}^{(F_m)}_{\<f\>}$ is  the energy measure of $f$ with respect to the Dirichlet form $\big(\bar{\mcE}^{(F_m)},W^{1,2}(F_m)\big)$.
\end{theorem}

 \hskip 0.2truein
 
\end{document}